\documentclass[12pt]{article}

\usepackage{geometry}

\usepackage{amsmath}
\usepackage{amsfonts}
\usepackage{amsthm}
\usepackage{amssymb}
\usepackage{graphicx}
\usepackage{extarrows}
\usepackage{enumerate}
\usepackage{enumitem}
\usepackage{bbm}

\usepackage{mathtools}
\usepackage[all]{xy}
\usepackage{ifthen}
\usepackage{tikz}
\usepackage{tikz-cd}
\usepackage{tikzscale}

\usetikzlibrary{shapes.geometric}
\usetikzlibrary{decorations, decorations.markings} 
\usetikzlibrary{arrows, arrows.meta}
\usetikzlibrary{matrix, arrows, decorations.pathmorphing}
\usetikzlibrary{calc}

\makeatletter
\def\enumfix{%
\if@inlabel
\noindent \par\nobreak\vskip-\topsep\hrule\@height\z@
\fi}
\let\oldenumerate\enumerate
\def\enumerate{\enumfix\oldenumerate}
\makeatother

\newlist{parlist}{enumerate}{1}
\setlist[parlist]{label=(\arabic*),wide=0pt,topsep=0pt}

\newcounter{num}
\newenvironment{pflist}{\begin{list}{(\arabic{num})}{\usecounter{num} \leftmargin0cm \itemindent5pt}}{\end{list}}

\theoremstyle{definition}
\newtheorem{Satz}{Satz}[section]
\newtheorem{Lemma}[Satz]{Lemma}
\newtheorem{Proposition}[Satz]{Proposition}
\newtheorem{Cor}[Satz]{Corollary}
\newtheorem{Theorem}[Satz]{Theorem}
\newtheorem{Definition}[Satz]{Definition}

\newtheoremstyle{break}{}{}{}{}{\bfseries}{.}{\newline}{}

\theoremstyle{break}

\makeatletter
\renewenvironment{proof}[1][\proofname]{\par
\vspace{-5pt}
\pushQED{\qed}%
\normalfont
\topsep0pt \partopsep0pt 
\trivlist
\item[\hskip\labelsep
\bfseries
#1\@addpunct{.}]\ignorespaces
}{%
\popQED\endtrivlist\@endpefalse
\addvspace{6pt plus 6pt} 
}
\makeatother

\newcommand{\Z}{\mathbb{Z}}
\newcommand{\N}{\mathbb{N}}
\newcommand{\CC}{\mathcal{C}}

\DeclareMathOperator\Hom{Hom}

\DeclareMathOperator\Aut{Aut}

\DeclareMathOperator\im{im}

\DeclareMathOperator\Span{Span}
\DeclareMathOperator\id{id}

\DeclareMathOperator\Res{Res}
\DeclareMathOperator\diag{diag}

\newcommand{\GL}{\mathrm{GL}}
\newcommand{\SL}{\mathrm{SL}}
\newcommand{\Sp}{\mathrm{Sp}}

\newcommand{\fA}{\mathfrak{A}}
\newcommand{\fB}{\mathfrak{B}}
\newcommand{\fC}{\mathfrak{C}}
\newcommand{\fS}{\mathfrak{S}}
\newcommand{\fT}{\mathfrak{T}}

\newcommand{\fN}{\mathfrak{N}}
\newcommand{\fG}{\mathfrak{G}}
\newcommand{\fU}{\mathfrak{U}}
\newcommand{\fa}{\mathfrak{a}}
\newcommand{\fr}{\mathfrak{r}}
\newcommand{\fs}{\mathfrak{s}}
\newcommand{\ft}{\mathfrak{t}}
\newcommand{\fn}{\mathfrak{n}}
\newcommand{\fd}{\mathfrak{d}}

\newcommand{\hS}{\hat{\mathfrak{S}}}

\renewcommand{\P}{\mathbb{P}}
\newcommand{\Ext}{\mathrm{Ext}}
\newcommand{\Yext}{\mathrm{Yext}}

\renewcommand{\1}{\mathbbm{1}}
\newcommand{\TFT}{\mathcal{Z}}
\renewcommand{\P}{\mathbb{P}}

\newcommand{\ot}{\mathbin{\otimes}}

\newcommand{\uu}{u}
\newcommand{\n}{\nu}

\newcommand{\deq}{\vcentcolon =}

\newcommand\Cent{C}

\newcommand{\scdot}{\mathbin{\mspace{-3mu} \cdot \mspace{-3mu}}}

\begin{document}

\begin{flushright}
{\sf ZMP-HH/26-25}\\
{\sf Hamburger$\;$Beitr\"age$\;$zur$\;$Mathematik$\;$Nr.$\;$1012}
\end{flushright}
\vskip 11mm

\begin{center}{\bf \Large 
Hochschild Cohomology, Modular Tensor Categories, and Mapping Class Groups}

\large{\bf II.~Examples}

\vskip 8mm

{\large Simon Lentner, Svea Nora Mierach,\\ \vskip 1mm Christoph Schweigert, Yorck Sommerh\"auser}
\vskip 9mm
\end{center}

\vspace{15mm}

\begin{abstract}
\noindent In the first part of this work, we have, for a not necessarily semisimple modular category, generalized the action of the mapping class groups of surfaces on the spaces of conformal blocks to the so-called derived block spaces. In the second part presented here, we compute this action explicitly in the case of Drinfel'd doubles of finite groups over fields of positive characteristic. To do that, we connect Lyubashenko's approach to mapping class group representations with the theory of representation varieties. In this way, we are able to show that the mapping class group representations on the derived block spaces are in general different from those on the ordinary block spaces.
\end{abstract}

\setcounter{tocdepth}{2}

\newpage

\tableofcontents

\newpage 

\section*{Introduction}
\addcontentsline{toc}{section}{Introduction}

The present text constitutes a direct continuation of our previous work~\cite{LMSS2}, which comprises the first three sections of the complete text, whereas the present second part begins with the fourth section. References to the first three sections are therefore found in the first part.

In this first part, we have generalized the action of mapping class groups of surfaces on the spaces of chiral conformal blocks, or briefly the block spaces, to the so-called derived block spaces. For this purpose, we have used a framework created by V.~Lyubashenko, which we have reviewed in Paragraph~2.5 and which describes surfaces via nets and ribbon graphs. The key point of our argument was described in Paragraph~3.2: We considered a surface with an additional boundary component, which could be labeled with a projective resolution of the unit object,
so that the block spaces of this modified surface formed a cochain complex.
We then showed that the action of the mapping class group of the modified surface descended to an action of the mapping class group of the original surface in cohomology, i.e., on the derived block spaces.

It is natural to ask whether these mapping class group representations on the derived block spaces are really new in the sense that they are not isomorphic to the representations on the original block spaces. In Paragraph~\ref{Char3} below, we answer this question affirmatively: The representations that appear in higher cohomological degree are in general not isomorphic to those in degree zero. Furthermore, we show in Paragraph~\ref{YonChar3} that these new representations can also not always be obtained from those in degree zero by taking Yoneda products. In other words, the submodule that is generated from the degree zero component under the Yoneda product may not contain these new representations.

To reach these goals, we consider a special modular category, namely the category of representations of the Drinfel'd double of a finite group, a category that is nonsemisimple if the characteristic of the base field divides the order of the finite group under consideration. For this category, we are able to show that the mapping class group representations constructed by V.~Lyubashenko are isomorphic to those arising from the so-called representation variety, i.e., the set of group homomorphisms from the fundamental group of the surface to the given finite group. We give this isomorphism as Theorem~\ref{ThmModTilde} in Paragraph~\ref{RelRepVar}. Another key fact that we use is a correspondence between certain cohomology groups of the Drinfel'd double of our finite group and certain cohomology groups of that finite group itself, a relation that we establish as Proposition~\ref{GroupCohom} in Paragraph~\ref{RedGroupCohom}. Taken together, these facts show that the derived block spaces can be computed as the cohomology groups of our finite group with coefficients in a linearized version of the representation variety, a result that we state as Corollary~\ref{Group} directly afterwards.

Let us now review our contents in greater detail and in a more linear fashion. Like any quasitriangular Hopf algebra, also the Drinfel'd double of a finite group comes with two natural R-matrices, denoted by~$R$ and~$\tilde{R}$ in Paragraph~\ref{AltRMat}. Interchanging these two R-matrices corresponds graphically to interchanging overcrossings and undercrossings.  While none of the two R-matrices is preferred over the other, one of the two can be expressed more easily in terms of the standard basis, at least for the model of the Drinfel'd double that we are using.

In general, the mapping class group representations considered here depend on the R-matrix used. However, we show in Section~\ref{Sec:MapClFund}, more precisely in Theorem~\ref{RepIsom} of Paragraph~\ref{StandRMat}, that both~$R$ and~$\tilde{R}$ lead to isomorphic mapping class group representations in the cases that are relevant for us. As the formulas arising from~$\tilde{R}$ are considerably simpler, at least for the conventions that we are using, we work with~$\tilde{R}$ in the sequel. Another key result of Section~\ref{Sec:MapClFund} is the description of the relevant mapping class group actions via the representation variety in Theorem~\ref{ThmModTilde} that was already mentioned above.

In Section~\ref{Sec:DerBlCoh}, we reduce the computation of the relevant cohomology groups to group cohomology, as also already mentioned above. Moreover, the appearing cohomology groups are modules over the cohomology ring of the unit object of the category, which is an algebra with respect to the Yoneda product. This module structure plays an important role in Section~\ref{Sec:CompCoh}. To define this product, we pass to Yoneda's description of the cohomology groups, which is not based on projective resolutions, in contrast to the description that we gave in Paragraph~3.1.

\enlargethispage{-4mm}

The module structure over the Yoneda algebra is compatible with the passage to group cohomology just discussed. In group cohomology, however, it is customary not to use the Yoneda product itself, but rather the so-called cup product, which in a certain sense can be seen as a generalization of the Yoneda product. Although, strictly speaking, the cup product is not necessary for our treatment, we explain in the appendix not only the relation between the two products, but also discuss there the cup product in the setting of general tensor categories, which, compared to the case of groups or Hopf algebras, requires the modification of some arguments.

As stated above, we have seen in Section~\ref{Sec:MapClFund} that the relevant mapping class group representations can be realized on the representation variety. The given finite group acts on the representation variety by conjugation, and this action commutes with the action of the mapping class group. For the subsequent computations, it would be helpful if the representation variety could be decomposed into a Cartesian product in which the mapping class group acts on one factor and the given finite group acts on the other. While this is not possible in general, we show in Section~\ref{Sec:DecRep} that it is possible to come rather close to such a decomposition. The corresponding variant of this decomposition, which is stated as Proposition~\ref{PropDecomp} in Paragraph~\ref{GenCase}, descends to cohomology by a particularly simple case of the universal coefficient theorem, as we show in Theorem~\ref{ThmDecomp} directly afterwards.

We also show in Section~\ref{Sec:DecRep} that, if our finite group is abelian, the direct sum of the derived block spaces in all degrees is free as a right module over the cohomology ring of the base field under the Yoneda product. More precisely, we show in Paragraph~\ref{Abel} that the representation variety can be used as a basis for this free module, so that the arising isomorphism is equivariant with respect to the action of the mapping class group. This means that the key features that we want to exhibit cannot be realized by using only abelian groups: In the abelian case, the mapping class group representations that appear for the derived block spaces in higher degree are the same as the mapping class group representations on the ordinary block spaces in degree zero.

In Section~\ref{Sec:CompCoh}, we reach the explicit construction of examples that exhibit the features just mentioned. As it is necessary to work with a nonabelian group, we consider the smallest such group, the symmetric group on three letters. In order to keep the examples as simple as possible, we choose~$g=1$ for the genus of our surface, so that the mapping class group is the modular group~$\SL(2,\Z)$. To be in a nonsemisimple situation, we need to consider base fields of characteristic~$2$ or~$3$. If the characteristic is~$2$, we find in Paragraph~\ref{Char2} that all mapping class group representations that appear in higher cohomological degree are still already present in degree zero. However, the direct sum of the derived block spaces is no longer free over the cohomology ring of the base field. The situation changes if we pass to characteristic~$3$, as we show in Paragraph~\ref{Char3}: Here, new representations appear when the cohomological degree is congruent to~$1$ or~$2$ modulo~$4$. These new representations do also not appear in the submodule that is generated under the Yoneda product by the degree zero component, i.e., by the ordinary block spaces, as follows from Theorem~\ref{ModStrucChar3} in Paragraph~\ref{YonChar3}.

We conclude with an appendix that first reviews in Paragraph~\ref{PerResCycl} the periodic resolution for cyclic groups, which is one of our most important computational tools. The remaining parts of the appendix discuss the cup product in the context of general tensor categories, as already mentioned above.

We continue to use the notations and conventions introduced in the first part. In particular, we work over an algebraically closed base field denoted by~$K$. For a set~$X$, $K[X]$ denotes the free vector space with $X$ as a basis (\cite[Chap.~I, Sec.~1.7, p.~13]{Gr1}). For a group~$G$ and an element~$a \in G$, we denote its centralizer by~$\Cent(a)$. If this group acts on a set~$X$, we denote the fixed point set by~$X^{G}$. The notation~$U \le G$ indicates that the subset~$U$ of~$G$ is in fact a subgroup. The symbol~$0$ is used to denote not only the zero vector in a vector space, but also the zero vector space itself and more generally the zero object in an abelian category.

As explained in Paragraph~2.1, the left dual of an object~$X$ in a left rigid category is denoted by~$X^*$, and the left dual of a morphism~$f$ is denoted by~$f^*$. This applies in particular to the category of finite-dimensional vector spaces, where~$X^*$ and~$f^*$ then denote the dual vector space and the dual map, i.e., the transpose of~$f$. We will use the same notation also in the infinite-dimensional setting. Generalizing this special case, we use the notation~$f^*$ also for the map given by precomposing~$f$, and accordingly the notation~$f_*$ for the map given by postcomposing~$f$. In addition, we also use~$f_*$ for the map induced by postcomposition in cohomology. In contrast, the transpose of a matrix~$M$ is denoted by~$M^T$.

We also use the Kronecker symbol~$\delta_{a,b}$, which is equal to~$1$ if~$a=b$ and equal to~$0$ if~$a \neq b$. It needs to be distinguished from the dual basis elements in the dual group ring introduced in Paragraph~\ref{DrinfDoub}, for which we use the notation~$\delta_{a}$, so that~$\delta_{a}(b) = \delta_{a,b}$.

We note that our treatment of the Yoneda product in Paragraph~\ref{YonProd} leads into territory that is no longer covered by the standard axioms of set theory put forward by E.~Zermelo and A.~Fraenkel, together with the axiom of choice, as it requires the formation of so-called big groups. To a lesser extent, this also concerns other aspects of category theory discussed here. Various ways to deal with these difficulties have been proposed. One of them is to adopt an additional set-theoretical axiom, namely the axiom of the existence of a universe, and this is the approach taken here. This approach is discussed in greater detail in~\cite[Chap.~I, \S~6, p.~21ff]{ML2}. In order to form the big groups just mentioned, we need to assume in addition that the category under consideration is small with respect to the universe in the sense of~\cite[loc.\ cit., p.~22]{ML2}.

The authors would like to thank Sarah Witherspoon for interesting discussions on the present material and Marc Hoyois for pointing out reference~\cite{Bg}.

While carrying out this research, the first and the third author were partially supported by the `Deutsche Forschungsgemeinschaft' under Germany's Excellence Strategy EXC 2121 `Quantum Universe' -- 390833306 and the Collaborative Research Center SFB 1624 `Higher Structures, Moduli Spaces and Integrability' -- 506632645, while the second and the fourth author were partially supported by NSERC grant RGPIN-2017-06543.

\setcounter{section}{3}
\section{Mapping class groups and fundamental groups} \label{Sec:MapClFund}
\subsection{The alternative R-matrix} \label{AltRMat}
In Paragraph~2.2, we have considered factorizable ribbon Hopf algebras. Any such Hopf algebra~$A$ comes, by definition, with an R-matrix~$R$. There is, however, always a second possible choice for such an R-matrix, namely $\tilde{R} \deq R'^{-1}$, the inverse of the element~$R'$ introduced in Paragraph~2.2 (cf.~\cite[Par.~2.1, p.~17]{SZ}). This alternative R-matrix leads to an alternative braiding, namely
\[\tilde{c}_{X,Y} \deq c_{Y,X}^{-1} \colon X \ot Y \to Y \ot X.\]
The monodromy matrix for this alternative R-matrix is
\[\tilde{Q} = \tilde{R}' \tilde{R} = R^{-1} R'^{-1} = Q^{-1}\]
and it is not difficult to see that~$\tilde{v} \deq v^{-1}$ is a ribbon element for this alternative R-matrix (cf.~\cite[Par.~4.4, p.~40]{SZ}). Note that~$A$ is also factorizable with respect to the alternative R-matrix (cf.~\cite[Par.~3.2, p.~26]{SZ}).

In Paragraph~2.4, we have seen that the dual space~$A^*$ of~$A$, endowed with the left coadjoint action, is a coend~$L$ for the category of $A$-modules. This coend depends, as an object of the category, on the tensor product functor and the duality, but not on the braiding, and therefore not on the R-matrix. However, the associated morphisms computed in Proposition~2.3 do depend on the R-matrix. The ones corresponding to~$\tilde{R}$, which we also denote by a tilde, can be expressed in terms of the original ones:
\begin{Proposition} \label{PropAltRMat}
\[\tilde{\fT} = \fT^{-1} \qquad \qquad \quad \tilde{\fN}' = \fN'^{-1} \qquad \qquad \quad 
\tilde{\fN} = \fN^{-1}\]

\end{Proposition}
\begin{proof}
\begin{parlist}
\item
These statements are direct consequences of Proposition~2.3: For $\varphi \in A^*$ and $a \in A$, we have
\[(\tilde{\fT} \circ \fT)(\varphi)(a) = \fT(\varphi)(a \tilde{v})
= \varphi(a \tilde{v} v) = \varphi(a)\]
and therefore $\tilde{\fT} \circ \fT = \id_L$. The equation
$\fT \circ \tilde{\fT} = \id_L$ can be derived analogously.

\item
For the second and the third statement, suppose that $\varphi, \psi \in A^*$ and $a,a' \in A$. We then have
\begin{align*}
(\tilde{\fN}' \circ \fN')&(\varphi \ot \psi)(a \ot a') =
\fN'(\varphi \ot \psi)(a \tilde{Q}^1 \ot S(\tilde{Q}^2) a') \\
&= \varphi(a \tilde{Q}^1 Q^1) \psi(S(Q^2) S(\tilde{Q}^2) a')
= \varphi(a \tilde{Q}^1 Q^1) \psi(S(\tilde{Q}^2 Q^2) a') \\
&= \varphi(a) \psi(a') = (\varphi \ot \psi)(a \ot a')
\end{align*}
and
\begin{align*}
(\tilde{\fN} \circ \fN)&(\varphi \ot \psi)(a \ot a') =
\fN(\varphi \ot \psi)(a \tilde{v}_{(1)} \ot S(\tilde{v}_{(2)}) a') \\
&= \varphi(a \tilde{v}_{(1)} v_{(1)}) \psi(S(v_{(2)}) S(\tilde{v}_{(2)}) a')
= \varphi(a \tilde{v}_{(1)} v_{(1)}) \psi(S(\tilde{v}_{(2)} v_{(2)}) a') \\
&= \varphi(a) \psi(a') = (\varphi \ot \psi)(a \ot a').
\end{align*}

This shows that $\tilde{\fN}' \circ \fN' = \id_{L \ot L}$ and
$\tilde{\fN} \circ \fN = \id_{L \ot L}$. The equations
$\fN' \circ \tilde{\fN}' = \id_{L \ot L}$ and $\fN \circ \tilde{\fN} = \id_{L \ot L}$
can be derived in a very similar way.
\qedhere
\end{parlist}
\end{proof}

The treatment of the morphism~$\fS$ requires some preparatory remarks. In Paragraph~2.4, we have used a right integral~$\rho \in A^*$ to define~$\fS$. Such a right integral is only unique up to a scalar, which did not matter, as we did only consider projective representations. In order to make the relation between~$\fS$ and~$\tilde{\fS}$ precise, we now use ribbon-normalized integrals (cf.~\cite[Def.~4.4, p.~39]{SZ}); i.e., we use a right integral that satisfies~$\rho(v) = 1$. For the definition of~$\tilde{\fS}$, we require analogously that~$\tilde{\rho}(\tilde{v}) = \tilde{\rho}(v^{-1}) = 1$. If we use these normalizations, the two morphisms are related as follows:
\begin{Lemma} \label{LemAltRMat}
$\tilde{\fS} = \fS^{-1}$
\end{Lemma}
\begin{proof}
As in Paragraph~2.8, we use the notation~$\hat{\fS}$ for the map that was denoted by~$\fS$ in~\cite[Sec.~3, p.~410]{LMSS1} or~\cite[Par.~4.1, p.~35]{SZ}, and accordingly we use the notation~$\hat{\tilde{\fS}}$ for the analogous map arising from the alternative R-matrix~$\tilde{R}$. We then know from~\cite[Par.~4.4, p.~40]{SZ} that~$\hat{\tilde{\fS}} = \hat{\fS}^{-1}$.

According to Lemma~2.4, $\fS$ and~$\hat{\fS}$ are conjugate via~$\bar{\iota}$. Applying this lemma to the alternative R-matrix, we obtain that $\tilde{\fS}$ and~$\hat{\tilde{\fS}}$ are conjugate via the corresponding map~$\tilde{\bar{\iota}}$. But~$\bar{\iota}$ and~$\tilde{\bar{\iota}}$ are proportional, so $\tilde{\fS}$ and~$\hat{\tilde{\fS}}$ are also conjugate via~$\bar{\iota}$. This implies the assertion.
\end{proof}

It should be noted that the proposition above is not new; rather, it is a variant of~\cite[Thm.~6.10, p.~323]{Ly} (cf.~also~\cite[Thm.~4.4, p.~523]{LM}).

\subsection{The Drinfel'd double of a finite group} \label{DrinfDoub}
Important examples for factorizable Hopf algebras are the Drinfel'd doubles of finite-dimensional Hopf algebras (cf.~\cite[Def.~IX.4.1, p.~213]{Ka}). In the case where the Hopf algebra is the group ring~$H \deq K[G]$ of a finite group~$G$, we denote its Drinfel'd double by~$D(G)$, or briefly by~$D$. If we denote the dual basis element of~$a \in G$ by $\delta_a \in K^G \deq H^*$, then the Drinfel'd double has a basis consisting of the elements $\delta_{a} \ot b$ for~$a, b \in G$. It follows from~\cite[Lem.~IX.4.2, p.~214]{Ka} that the product of two of these basis elements is given by
\[(\delta_a \ot b)(\delta_c \ot d) = \delta_{a, b c b^{-1}} \, \delta_a \ot b d,\] 
where $\delta_{a, b c b^{-1}}$ is the Kronecker symbol. The unit element of~$D$ is $1_{D} = \varepsilon \ot e$, where~$e$ is the unit element of~$G$ and $\varepsilon = \sum_{a\in G} \delta_a$ is the unit element of~$K^G$. In particular, we have  
$(\varepsilon \ot a)(\delta_b \ot e) = \delta_{aba^{-1}} \ot a$ (cf.~\cite[Eq.~(4.11), p.~219]{Ka}). By~\cite[Eq.~(4.2), p.~214]{Ka}, we have for the coproduct of~$D$ that
\[\Delta_{D}(\delta_a \ot b) = \sum_{c \in G}(\delta_c \ot b) \ot (\delta_{ac^{-1}} \ot b),\]
while we have $\varepsilon_{D}(\delta_a \ot b) = \delta_{a,e}$ for the counit. The antipode is given by \[S_{D}(\delta_a \ot b) = \delta_{b^{-1} a^{-1} b} \ot b^{-1}\]
(cf.~\cite[Par.~2.2, p.~17]{SZ}).

The Drinfel'd double~$D$ is quasitriangular with respect to the \mbox{R-matrix} 
\[R \deq \sum_{a \in G} (\varepsilon \ot a) \ot (\delta_a \ot e) 
= \sum_{a,b \in G} (\delta_b \ot a) \ot (\delta_a \ot e)\]
(cf.~\cite[Eq.~(4.12), p.~220]{Ka}), and it is factorizable with respect to this R-matrix (cf.~\cite[Prop.~3.3, p.~28]{SZ}). The monodromy matrix from Paragraph~2.2 has therefore the form
\begin{align*}
Q = \sum_{a,b \in G} (\delta_b \ot a) \ot (\varepsilon \ot b) (\delta_a \ot e) 
= \sum_{a,b \in G} (\delta_b \ot a) \ot (\delta_{bab^{-1}} \ot b).                                                                                                                                                                                                                                                                                                                    \end{align*}
Associated with every R-matrix is its Drinfel'd element (cf.~\cite[Eq.~(4.1), p.~180]{Ka}), which in this case is given by~$u_D = \sum_{a \in G} \delta_a \ot a^{-1}$, while its inverse is given by~\mbox{$u_D^{-1} = \sum_{a \in G} \delta_a \ot a$}
(cf.~\cite[Par.~2.2, p.~18]{SZ}), so that 
\[S_D(u_D^{-1}) = \sum_{a \in G} \delta_{a^{-1} a^{-1} a} \ot a^{-1} = u_D^{-1}. \]
Since~$S_D$ is an involution, $u_D$ is central, so that the formulas given in~\cite[Par.~2.1, p.~17]{SZ} imply that we can use~$u_D^{-1}$ as a ribbon element in the sense of Paragraph~2.2, which we will do in the sequel.

It follows from~\cite[Par.~2.3, p.~18]{SZ} that the function~$\rho_D \in D^*$ defined by 
\[\rho_D(\delta_{a} \ot b) \deq \delta_{b,e}\]
is a two-sided integral on~$D$. In other words, if we set $A \deq D$, $v \deq u_D^{-1}$, and $\rho \deq \rho_D$, the hypotheses made in Paragraph~2.2 and Paragraph~2.4 are satisfied. We have $\rho_D(u_D) = \rho_D(u_D^{-1}) = 1$, so that~$\rho_D$ is ribbon-normalized in the sense recalled in Paragraph~\ref{AltRMat}. According to~\cite[Lem.~4.4, p.~39]{SZ}, this implies that
$(\rho_D \ot \rho_D)(Q) = 1$. As also explained there, these equations have the consequence that the action of the modular group~$\SL(2,\Z)$ on~$Z(D)$ considered in Paragraph~2.8 is not only projective, but even linear. As we saw in Paragraph~2.8, this action is conjugate to the action 
on~$\bar{C}(D) \subset L = D^*$, so that this action is also linear. 

The alternative R-matrix considered in Paragraph~\ref{AltRMat} is in this case given by
\[\tilde{R} = \sum_{a,b \in G} (\delta_a \ot e) \ot (\delta_b \ot a^{-1})\]
with associated monodromy matrix
$\tilde{Q} =
\sum_{a, b \in G} (\delta_{aba^{-1}} \ot  a) \ot (\delta_{a^{-1}} \ot  b^{-1})$. As pointed out in Paragraph~\ref{AltRMat}, we can choose the Drinfel'd element~$u_D$ as our ribbon element for the alternative R-matrix in this case, and as just mentioned,
the two-sided integral~$\rho_D$ is also ribbon-normalized for this ribbon element.

\subsection{The endomorphisms~$\fS$ and~$\fT$} \label{EndomCoend}
We have discussed in Paragraph~2.4 that the coend~$L = D^*$ is a module over~$D$ with respect to the left coadjoint action. As in~\cite[Par.~3.3, p.~27]{SZ}, we can use the vector space isomorphism
$D^* \cong H \ot H^*$ given by
\[H \ot H^* \rightarrow D^*,~h \ot \varphi \mapsto 
(\varphi' \ot h' \mapsto \varphi'(h) \varphi(h'))\]
to transport this action to~$H \ot H^*$. The basis of~$D$ described in Paragraph~\ref{DrinfDoub} has a dual basis whose elements we denote by~$\uu(a,b)$, so that
\[\uu(a,b)(\delta_c \ot d) = \delta_{a,c} \delta_{b,d}\]
for all $a,b,c,d \in G$. The coadjoint action is then given on these basis elements by the following formulas:
\begin{Proposition} \label{ActCoend} 
For elements $a,b,h \in G$, we have
\[(\varepsilon \ot h).\uu(a,b) = \uu(hah^{-1}, hbh^{-1}) \quad \text{and} \quad
(\delta_h \ot e).\uu(a,b) = \delta_{h,[b^{-1}, a]} \; \uu(a,b),\]
where $[b^{-1}, a] = b^{-1} a b a^{-1}$ denotes the commutator of~$b^{-1}$ and~$a$.
\end{Proposition}
\begin{proof}
From the definition of the coadjoint action in Paragraph~2.4, we get
\begin{align*}
&((\varepsilon \ot h).\uu(a,b))(\delta_c \ot d) =
\uu(a,b)(S_D(\varepsilon \ot h)(\delta_c \ot d)(\varepsilon \ot h)) \\
& \quad = \uu(a,b)((\varepsilon \ot h^{-1})(\delta_c \ot d h))
= \uu(a,b)(\delta_{h^{-1}ch} \ot h^{-1} d h)
= \delta_c(hah^{-1})\delta_{hbh^{-1}}(d)
\end{align*}
and 
\begin{align*}
&((\delta_h \ot e).\uu(a,b))(\delta_c \ot d) =
\sum_{k \in G} \uu(a,b) (S_D(\delta_{k} \ot e)(\delta_c \ot d)(\delta_{hk^{-1}} \ot e)) \\
&\qquad = \sum_{k \in G} \uu(a,b)
((\delta_{k^{-1}} \ot e)(\delta_c \ot d)(\delta_{hk^{-1}} \ot e)) 
= \uu(a,b) ((\delta_c \ot d)(\delta_{hc} \ot e)) \\
&\qquad = \uu(a,b) (\delta_{c,dhcd^{-1}} \delta_c \ot d)
= \delta_{h,b^{-1}aba^{-1}} \, \uu(a,b)(\delta_c \ot d),
\end{align*}
which implies the assertions.
\end{proof}

In Paragraph~\ref{AltRMat}, we have recalled the endomorphisms~$\fT$ and~$\fS$ of~$L$. Under the isomorphism~$L \cong H \ot H^*$ described above, they take the following explicit form:
\begin{samepage}
\begin{Proposition} \label{FormST}
\begin{enumerate}
\item 
$\fT(\uu(a,b)) = \uu(a,{a^{-1} b})$

\item 
$\fS(\uu(a,b)) = \uu(a b^{-1} a^{-1},a)$
\end{enumerate}
\end{Proposition}
\end{samepage}
\begin{proof}
From Proposition~2.3 in Paragraph~2.4, we obtain
\begin{align*}
\fT(\uu(a,b))(\delta_c \ot d) &= \uu(a,b)((\delta_c \ot d) u_D^{-1})
= \sum_{h \in G} \uu(a,b)((\delta_c \ot d)(\delta_{h} \ot h)) \\
&= \sum_{h \in G} \delta_{c,dhd^{-1}} \uu(a,b)(\delta_c \ot dh).
\end{align*}
Since $c = dhd^{-1}$ if and only if $dh = cd$, this means that
\begin{align*}
\fT(\uu(a,b))(\delta_c \ot d) &= \uu(a,b)(\delta_c \ot cd)
= \delta_c(a) \delta_b(cd) = \delta_c(a) \delta_{a^{-1} b}(d),
\end{align*}
which yields the first assertion. 

For the second assertion, we get from the last formula in Paragraph~2.4 that
\begin{align*}
&\fS(\uu(a,b))(\delta_c \ot d)
= \uu(a,b)(Q^1) \, \rho_D(S_D(Q^2) (\delta_c \ot d)) \\
&\qquad = \sum_{h,k \in G} \uu(a,b)(\delta_{h} \ot k) \,
\rho_D(S_D(\delta_{h k h^{-1}} \ot h)(\delta_c \ot d)) \\
&\qquad = \rho_D(S_D(\delta_{a b a^{-1}} \ot a)(\delta_c \ot d)) 
= \rho_D((\delta_{b^{-1}} \ot a^{-1})(\delta_c \ot d)) \\
&\qquad = \delta_{b^{-1}, a^{-1} c a} \, \rho_D(\delta_{b^{-1}} \ot a^{-1} d)
= \delta_{a b^{-1} a^{-1}, c} \, \delta_{a^{-1} d, e}
\end{align*}
as required.
\end{proof}

We have seen in Paragraph~\ref{AltRMat} that the corresponding endomorphisms for the alternative R-matrix~$\tilde{R}$ are just the inverses of the endomorphisms for the original R-matrix, so that we have $\tilde{\fT} = \fT^{-1}$ and~$\tilde{\fS} = \fS^{-1}$. Explicitly, these endomorphisms are given as follows:
\begin{Cor} \label{InvST}
\begin{enumerate}
\item
$\fT^{-1}(\uu(c, d)) = \uu(c, cd)$

\item
$\fS^{-1}(\uu(c, d)) = \uu(d, d^{-1} c^{-1} d)$
\end{enumerate}
\end{Cor}
\begin{proof}
The first assertion is almost obvious. For the second assertion, suppose that
$c \deq a b^{-1} a^{-1}$ and $d \deq a$, so that
$\fS(\uu(a,b)) = \uu(c, d)$. We then have $c^{-1} = a b a^{-1}$ and consequently $b = d^{-1} c^{-1} d$, which gives the second assertion.
\end{proof}

We note that, written in terms of maps, the preceding computation shows that $\fS^{-1} \circ \fS = \id_L$, which is sufficient to establish the assertion in view of the finite-dimensionality of~$L$.

\subsection{Powers and commutators} \label{PowComm}
From Proposition~\ref{FormST}, we can compute frequently used powers of~$\fS$ with the help of commutators:
\begin{Proposition} \label{PowS}
For $a, b \in G$, we have
\begin{enumerate}
\item
$\fS^2(\uu(a,b)) = \uu([a, b^{-1}] a^{-1}, [a, b^{-1}] b^{-1})$

\item
$\fS^{-2}(\uu(c,d)) = \uu(c^{-1} [c, d^{-1}], d^{-1} [c, d^{-1}])$

\item
$\fS^4(\uu(a,b)) = (\varepsilon \ot [a, b^{-1}]). \uu(a,b)$

\item
$\fS^{-4}(\uu(c,d)) = (\varepsilon \ot [c, d^{-1}]^{-1}).\uu(c,d)$
\end{enumerate}
\end{Proposition}
\begin{proof}
\begin{pflist}
\item
We have
\begin{align*}
\fS^2(\uu(a,b)) &= \fS(\uu(a b^{-1} a^{-1}, a))
= \uu((a b^{-1} a^{-1}) a^{-1} (a b^{-1} a^{-1})^{-1}, a b^{-1} a^{-1}) \\
&= \uu((a b^{-1} a^{-1} b) a^{-1}, (a b^{-1} a^{-1} b) b^{-1}).
\end{align*}
This proves the first assertion.

\item
To prove the second assertion, we use the abbreviation
$h \deq [a, b^{-1}] = a b^{-1} a^{-1} b$, so that $h^{-1} a b^{-1} a^{-1} b = e$, which implies $a^{-1} b h^{-1} a b^{-1} = e$. Therefore, if we define
$c \deq h a^{-1}$ and $d \deq h b^{-1}$, we get
\[ [c, d^{-1}] = (ha^{-1}) (b h^{-1}) (a h^{-1}) (hb^{-1})
= h a^{-1} b h^{-1} a b^{-1} = h\]
and thus $a = c^{-1} h = c^{-1} [c, d^{-1}]$ as well as $b = d^{-1} h = d^{-1} [c, d^{-1}]$, which proves the second assertion.

\item
Using the facts just established, we get
\[\fS^4(\uu(a,b)) = \fS^2(\uu(h a^{-1}, h b^{-1})) = \uu(h a h^{-1}, h b h^{-1}),\]
which yields the third assertion by Proposition~\ref{ActCoend}.

\item
The fourth assertion follows from a very similar argument: With the abbreviation
$h \deq [a, b^{-1}]$ as before, we set $c \deq h a h^{-1}$ and $d \deq h b h^{-1}$ to get
\begin{align*}
&[c, d^{-1}] = h [a, b^{-1}] h^{-1} = h h h^{-1} = h = [a, b^{-1}]
\end{align*}
and therefore
\begin{align*}
\fS^{-4}(\uu(c,d)) = \uu(a,b) = \uu(h^{-1} c h,h^{-1} d h)
= \uu([c, d^{-1}]^{-1} c [c, d^{-1}],[c, d^{-1}]^{-1} d [c, d^{-1}])
\end{align*}
as asserted.
\qedhere
\end{pflist}
\end{proof}

We give two additional proofs of the third assertion, which are less direct and more complicated, but on the other hand connect it to two different topics that we have discussed before.
\begin{parlist}
\item
For the second proof of the third assertion, we now proceed in steps. As discussed at the end of Paragraph~\ref{DrinfDoub}, we have $(\rho_D \ot \rho_D)(Q) = 1$. Therefore a formula stated after Lemma~2.4 in Paragraph~2.8 states here that
\[\hS^4(\delta_c \ot d) = S_D(u_{D (1)}) (\delta_c \ot d) u_{D (2)}.\]
We then get
\begin{align*}
&S_D(u_{D (1)}) (\delta_c \ot d) u_{D (2)}
= \sum_{h,k \in G} S_D(\delta_{h} \ot k^{-1}) (\delta_c \ot d) (\delta_{k h^{-1}} \ot k^{-1})  \\
&\qquad = \sum_{h,k \in G} \delta_{c, d k h^{-1} d^{-1}} (\delta_{k h^{-1} k^{-1}} \ot k) (\delta_{c} \ot d k^{-1}) \\
&\qquad = \sum_{h,k \in G} \delta_{c, d k h^{-1} d^{-1}} \delta_{k h^{-1} k^{-1}, k c k^{-1}}
(\delta_{k h^{-1} k^{-1}} \ot k d k^{-1}) \\
&\qquad = \sum_{k \in G} \delta_{c, d k c d^{-1}} (\delta_{k c k^{-1}} \ot k d k^{-1}).
\end{align*}
Now we have $c = d k c d^{-1}$ if and only if $k = d^{-1} c d c^{-1} = [c,d^{-1}]^{-1}$, so that
\begin{align*}
\hS^4(\delta_c \ot d)
&= \delta_{[c,d^{-1}]^{-1} c [c,d^{-1}]} \ot [c,d^{-1}]^{-1} d [c,d^{-1}].
\end{align*}

\item
Because~$S_D$ is an involution, the maps~$\iota$ and~$\bar{\iota}$ considered in Paragraph~2.8 are equal. It therefore follows from Lemma~2.4 there, combined with~\cite[Prop.~4.1, p.~35]{SZ}, that~$\fS = \hat{\fS}^*$, the transpose of~$\hat{\fS}$, so that we have
\begin{align*}
(\fS^4(\uu(a,b))(\delta_c \ot d) &= \uu(a,b)(\hat{\fS}^4(\delta_c \ot d)) \\
&= \uu(a,b)(\delta_{[c,d^{-1}]^{-1} c [c,d^{-1}]} \ot [c,d^{-1}]^{-1} d [c,d^{-1}]).
\end{align*}
This expression is equal to~$1$ if
$a = [c,d^{-1}]^{-1} c [c,d^{-1}]$ and $b = [c,d^{-1}]^{-1} d [c,d^{-1}]$, and equal to~$0$ otherwise. As seen in the last step of the proof of the preceding proposition, this condition is equivalent to the condition that
$c = [a, b^{-1}] a [a, b^{-1}]^{-1}$ and $d = [a, b^{-1}] b [a, b^{-1}]^{-1}$.
This implies that
\begin{align*}
(\fS^4(\uu(a,b))(\delta_c \ot d)
&= \uu([a, b^{-1}] a [a, b^{-1}]^{-1}, [a, b^{-1}] b [a, b^{-1}]^{-1})(\delta_{c} \ot d)\\
&= ((\varepsilon \ot [a, b^{-1}]). \uu(a,b))(\delta_{c} \ot d)
\end{align*}
by Proposition~\ref{ActCoend}. This completes the second proof of the formula for~$\fS^4$.

\item
We still give a third proof, which is computationally simpler, but uses more of the categorical framework developed previously. If~$X$ is a module over~$D$, this framework yields an action of the mapping class group~$\Gamma_{1,1}$ on the space $\Hom_\CC(X,L)$, where \mbox{$L=D^*$} is the coend with the structures described in Paragraph~\ref{EndomCoend}. As we saw at the end of Paragraph~1.12, we have~$\Gamma_{1,1} \cong B_3$, and as we have
\mbox{$\rho_D(u_D) = \rho_D(u_D^{-1}) = 1$}, the discussion in Paragraph~2.8 shows that this action is not only projective, but even linear. According to the 2-chain relation mentioned in Paragraph~1.12, the inverse Dehn twist~$\fd_1^{-1}$ represents the same mapping class as~$\fs_1^4$. According to Paragraph~2.6, the inverse Dehn twist~$\fd_1^{-1}$ acts by precomposition with $\theta_{X}^{-1}$ on $\Hom_\CC(X,L)$, which is the same as postcomposition with $\theta_{L}^{-1}$, as the twist is natural. Also from Paragraph~2.6, we know that~$\fs_1^4$
acts by postcomposition with $\fS^{4}$. By the Yoneda lemma also mentioned at the end of that paragraph, this implies that~$\fs_1^4 = \theta_{L}^{-1}$.

According to Paragraph~2.2, the twist~$\theta_L$ is given by acting with the ribbon
element~\mbox{$v = u_D^{-1}$}. Therefore, we have by Proposition~\ref{ActCoend} that
\begin{align*}
\fS^{4}(\uu(a,b)) &= \theta_L^{-1}(\uu(a,b)) = u_D.\uu(a,b)
= \sum_{c \in G}(\delta_{c} \ot c^{-1}).\uu(a,b) \\
&= \sum_{c \in G}(\delta_{c} \ot e).\uu(c^{-1} a c, c^{-1} b c)
= \sum_{c \in G} \delta_{c, [c^{-1} b^{-1} c, c^{-1} a c]} \uu(c^{-1} a c,c^{-1} b c)\\
&= \uu([a, b^{-1}] a [a, b^{-1}]^{-1}, [a, b^{-1}] b [a, b^{-1}]^{-1})
= (\varepsilon \ot [a, b^{-1}]).\uu(a,b),
\end{align*}
where we have used that
$[b^{-1}, a]^{-1} = [a, b^{-1}]$. This finishes our third proof of the formula
for~$\fS^{4}$.
\end{parlist}

\begin{Cor} \label{CorPowS}
If $[G,G] \subset Z(G)$, there is a unique homomorphism from~$\SL(2,\Z)$ to~$\GL(D^*)$ that maps~$\fs$ to~$\fS$ and~$\ft$ to~$\fT$.
\end{Cor}
\begin{proof}
As we saw in Paragraph~1.12, the group~$\SL(2,\Z)$ is generated by the matrices~$\fs$ and~$\ft$ with the defining relations $\fs \ft \fs = \ft^{-1} \fs \ft^{-1}$ and~$\fs^4=1$. Since~\mbox{$\rho_D(u_D^{-1}) = 1$}, we know from Paragraph~2.8 that the first relation is always satisfied. The second relation follows from our hypothesis in view of the formula for~$\fS^4$ in Proposition~\ref{PowS}.
\end{proof}

\subsection{The endomorphisms~$\fA$ and~$\fB$} \label{EndomAB}
It turns out that the endomorphisms~$\fS$ and~$\fT$ and their inverses are actually conjugate. In order to see that, we introduce two other endomorphisms~$\fA$ and~$\fB$:
\begin{Definition} \label{DefAB}
We define~$\fA \colon L \to L$ and~$\fB \colon L \to L$ as the $K$-linear maps that take the values
\[\fA(\uu(a,b)) \deq \uu(b^{-1} a b, b^{-1}) \qquad \text{and} \qquad
\fB(\uu(a,b)) \deq \uu(a^{-1}, a b a^{-1})\]
on the basis elements.
\end{Definition}

The basic properties of these maps are the following:
\begin{Proposition} \label{Conjug}
Both~$\fA$ and~$\fB$ are involutions. They satisfy
\begin{enumerate}
\item
$\fT \circ \fA = \fA \circ \fT^{-1}$

\item
$\fS \circ \fA = \fA \circ \fS^{-1}$

\item
$\fT \circ \fB = \fB \circ \fT^{-1}$

\item
$\fS \circ \fB = \fB \circ \fS^{-1}$
\end{enumerate}
Moreover, we have~$\fA \circ \fB = \fS^2$ and~$\fB \circ \fA = \fS^{-2}$.
\end{Proposition}
\begin{proof}
\begin{parlist}
\item
The initial claim is almost obvious, as we have
\[\fA^2(\uu(a,b)) = \fA(\uu(b^{-1} a b, b^{-1})) = \uu(a,b)\]
and
$\fB^2(\uu(a,b)) = \uu(a^{-1}, a b a^{-1}) = \uu(a,b)$.

\item
From Proposition~\ref{FormST} and Corollary~\ref{InvST} above, we get that
\begin{align*}
\fT(\fA(\uu(a,b))) = \fT(\uu(b^{-1} a b, b^{-1})) =
\uu(b^{-1} a b, (b^{-1} a b)^{-1} b^{-1}) =
\uu(b^{-1} a b, b^{-1} a^{-1}),
\end{align*}
whereas
\begin{align*}
\fA(\fT^{-1}(\uu(a,b))) = \fA(\uu(a,ab))
= \uu((ab)^{-1} a (ab), (ab)^{-1})
= \uu(b^{-1} a b, b^{-1} a^{-1}).
\end{align*}
This proves the first equation.

\item
Similarly, we have for the second equation that
\begin{align*}
\fS(\fA(\uu(a,b))) &= \fS(\uu(b^{-1} a b, b^{-1}))
= \uu((b^{-1} a b) b (b^{-1} a b)^{-1}, b^{-1} a b)\\
&= \uu(b^{-1} a b a^{-1} b,b^{-1} a b)
\end{align*}
on the one hand and
\begin{align*}
\fA(\fS^{-1}(\uu(a,b))) &= \fA(\uu(b, b^{-1} a^{-1} b)) =
\uu((b^{-1} a^{-1} b)^{-1} b (b^{-1} a^{-1} b), b^{-1} a b) \\
&= \uu(b^{-1} a b a^{-1} b, b^{-1} a b)
\end{align*}
on the other hand. As these two expressions are equal, the second equation holds.

\item
For the third equation, we have
\begin{align*}
\fT(\fB(\uu(a,b))) = \fT(\uu(a^{-1}, a b a^{-1})) =
\uu(a^{-1}, a^2 b a^{-1})
\end{align*}
on the one hand and
$\fB(\fT^{-1}(\uu(a,b))) = \fB(\uu(a,ab)) = \uu(a^{-1}, a^2 b a^{-1})$
on the other hand, establishing the equation.

\item
Finally, for the fourth equation, we have
\begin{align*}
\fS(\fB(\uu(a,b))) = \fS(\uu(a^{-1}, a b a^{-1})) = \uu(b^{-1},a^{-1})
\end{align*}
on the one hand and
$\fB(\fS^{-1}(\uu(a,b))) = \fB(\uu(b,b^{-1} a^{-1} b)) = \uu(b^{-1},a^{-1})$
on the other hand, establishing the last equation.

\item
For the last two assertions, we have
\begin{align*}
(\fA \circ \fB)(\uu(a,b)) &= \fA(\uu(a^{-1}, a b a^{-1})) =
\uu((a b^{-1} a^{-1}) a^{-1} (a b a^{-1}), a b^{-1} a^{-1}) \\
&=\uu((a b^{-1} a^{-1} b) a^{-1}, (a b^{-1} a^{-1} b) b^{-1}) =
\uu([a,b^{-1}] a^{-1}, [a,b^{-1}] b^{-1}).
\end{align*}
In view of Proposition~\ref{PowS}, this shows that
$\fA \circ \fB = \fS^2$. Inverting this equation and using that~$\fA$ and~$\fB$ are involutions, we get that~$\fB \circ \fA = \fS^{-2}$.
\qedhere
\end{parlist}
\end{proof}

We note that, in the case where the elements~$a$ and~$b$ commute, the formulas for~$\tilde{\fT}$,~$\tilde{\fS}$, and~$\fA$ given in Corollary~\ref{InvST} and Definition~\ref{DefAB} reduce to
\[\tilde{\fT}(\uu(a,b)) = \uu(a,ab) \qquad
\tilde{\fS}(\uu(a,b)) = \uu(b,a^{-1}) \qquad
\fA(\uu(a,b)) = \uu(a,b^{-1}) \]
The matrices~$\ft$,~$\fs$, and~$\fa$ introduced in Paragraph~1.11 and~\cite[Par.~1.1, p.~7]{SZ} can therefore be considered as variants of these maps (cf.~also~\cite[Sec.~5, Eq.~(5.2), p.~8545]{KSSB}).

The endomorphisms considered here are compatible with commutators as follows:
\begin{Proposition} \label{TSABComm}
\begin{enumerate}
\item
If $\fT(\uu(a,b)) = \uu(c,d)$, then $[a,b^{-1}] = [c,d^{-1}]$.

\item
If $\fS(\uu(a,b)) = \uu(c,d)$, then $[a,b^{-1}] = [c,d^{-1}]$.

\item
If $\fA(\uu(a,b)) = \uu(c,d)$, then $[a,b^{-1}] = [c,d^{-1}]^{-1}$.

\item
If $\fB(\uu(a,b)) = \uu(c,d)$, then $[a,b^{-1}] = [c,d^{-1}]^{-1}$.
\end{enumerate}
\end{Proposition}
\begin{proof}
\begin{pflist}
\item
For the first assertion, we have
\[[c,d^{-1}] = [a, b^{-1} a] = a (b^{-1} a) a^{-1} (a^{-1} b) = a b^{-1} a^{-1} b
= [a,b^{-1}].\]

\item
For the second assertion, we have
\[[c,d^{-1}] = [a b^{-1} a^{-1}, a^{-1}] = (a b^{-1} a^{-1}) a^{-1} (a b a^{-1}) a
= a b^{-1} a^{-1} b = [a,b^{-1}].\]

\item
For the third assertion, we have
\[[c,d^{-1}] = [b^{-1} a b, b] = (b^{-1} a b) b (b^{-1} a^{-1} b) b^{-1}  =
b^{-1} a b a^{-1} = [a,b^{-1}]^{-1}.\]

\item
For the fourth assertion, we have
\[[c,d^{-1}] = [a^{-1}, a b^{-1} a^{-1}] = a^{-1} (a b^{-1} a^{-1}) a (a b a^{-1})
=  b^{-1} a b a^{-1}  = [a,b^{-1}]^{-1}\]
as asserted. \qedhere
\end{pflist}
\end{proof}

The above result can be interpreted in another way. First, we note that~$\fT$ is $D$-linear by construction and therefore satisfies
\[\fT((\delta_{h} \ot e).u(a,b)) = (\delta_{h} \ot e).\fT(u(a,b)).\]
If $\fT(\uu(a,b)) = \uu(c,d)$, this implies by Proposition~\ref{ActCoend} that
\[\delta_{h,[b^{-1},a]} \; \fT(u(a,b)) = \delta_{h,[d^{-1},c]} \; u(c,d),\]
which yields the first assertion in Proposition~\ref{TSABComm} above. The second assertion for~$\fS$ can be shown similarly.

In contrast, $\fA$ and~$\fB$ are not $D$-linear in general:
\begin{Cor} \label{ABLin}
$\fA$ and~$\fB$ are $H$-linear. Moreover, we have for $a,b,h \in G$ that
\begin{enumerate}
\item
$\fA((\delta_{h} \ot e).u(a,b)) = (\delta_{h^{-1}} \ot e).\fA(u(a,b))$

\item
$\fB((\delta_{h} \ot e).u(a,b)) = (\delta_{h^{-1}} \ot e).\fB(u(a,b))$
\end{enumerate}
\end{Cor}
\begin{proof}
By Proposition~\ref{ActCoend}, we have
\begin{align*}
\fA((\varepsilon \ot h&).u(a,b)) = \fA(u(h a h^{-1}, h b h^{-1}))
= u((h b^{-1} h^{-1}) (h a h^{-1}) (h b h^{-1}), h b^{-1} h^{-1}) \\
&= u(h b^{-1} a b h^{-1}, h b^{-1} h^{-1})
= (\varepsilon \ot h).u(b^{-1} a b, b^{-1})
= (\varepsilon \ot h).\fA(u(a,b))
\end{align*}
and the verification for~$\fB$ is similar. The two equations involving $\delta_{h} \ot e$
follow from Proposition~\ref{TSABComm} by a similar computation as the one that we carried out for~$\fT$ directly before this corollary.
\end{proof}

\subsection{The endomorphism~$\fN$} \label{EndomN}
In Paragraph~\ref{AltRMat}, we have also recalled the endomorphism~$\fN$ of~$L \ot L$. The isomorphism~$L \cong H \ot H^*$ described in Paragraph~\ref{EndomCoend} leads to a corresponding isomorphism between~$L \ot L$ and~$H \ot H^* \ot H \ot H^*$, and we proceed now to describe the explicit form of~$\fN$ as an endomorphism of~$H \ot H^* \ot H \ot H^*$. This description becomes clearer if we introduce the following auxiliary map:
\begin{Definition}
We define~$\n \colon H \ot H^* \ot H \ot H^* \to H$ as the $K$-linear map that takes the value
\[\n(\uu(a,b) \ot \uu(a',b')) \deq a'^{-1} b^{-1} a b \]
on the basis elements.
\end{Definition}

With the help of the function~$\n$, we can express the endomorphism~$\fN$ as follows:
\begin{Proposition} \label{FormN}
Suppose that $n \deq \n(\uu(a,b) \ot \uu(a',b'))$. Then we have
\[\fN(\uu(a,b) \ot \uu(a',b')) = \uu(a, b n^{-1}) \ot \uu(n a' n^{-1}, n b').\]
\end{Proposition}
\begin{proof}
From Proposition~2.3, we get
\begin{align*}
&\fN(\uu(a,b) \ot \uu(a',b'))((\delta_c \ot d) \ot (\delta_{c'} \ot d')) \\
&= \uu(a,b)((\delta_c \ot d) u_{D (1)}^{-1}) \;
\uu(a',b')(S_D(u_{D (2)}^{-1})(\delta_{c'} \ot d')) \\
&= \sum_{h,k \in G} \uu(a,b)((\delta_c \ot d) (\delta_{h} \ot k)) \;
\uu(a',b')(S_D(\delta_{k h^{-1}} \ot k)(\delta_{c'} \ot d'))  \\
&= \sum_{h,k \in G} \delta_{c,dhd^{-1}} \uu(a,b)(\delta_c \ot dk) \;
\uu(a',b')((\delta_{k^{-1} h} \ot k^{-1})(\delta_{c'} \ot d')).
\end{align*}
From the first factor in this expression, we see that only one term in this sum is nonzero, namely the summand for $h=d^{-1}cd$ and $k=d^{-1}b$, for which $k^{-1} h = b^{-1} c d$. Our expression is therefore equal to
\begin{align*}
&\uu(a,b)(\delta_c \ot b) \;
\uu(a',b')((\delta_{b^{-1} c d} \ot b^{-1} d)(\delta_{c'} \ot d')) \\
&\qquad = \delta_{c,a} \delta_{b^{-1} c d, b^{-1} d c' d^{-1} b}\;
\uu(a',b') (\delta_{b^{-1} c d} \ot b^{-1} d d')
= \delta_{c,a} \delta_{c d, d c' d^{-1} b} \delta_{a', b^{-1} c d} \delta_{b', b^{-1} d d'}.
\end{align*}
This expression is equal to~$1$ if and only if 
\[c = a, \qquad c d = d c' d^{-1} b, \qquad a' = b^{-1} c d, \qquad b' = b^{-1} d d' \]
and equal to~$0$ otherwise. Replacing~$c$ by~$a$, this condition takes the equivalent form
$c = a$, $a d = d c' d^{-1} b$, $a' = b^{-1} a d$, and $b' = b^{-1} d d'$. Solving for~$d$ in the third equation, it reads $d = a^{-1} b a'$, and if we substitute this form of~$d$ into the other equations, we obtain the equivalent condition
\[c = a, \qquad b a' = (a^{-1} b a') c' (a^{-1} b a')^{-1} b, \qquad d = a^{-1} b a', \qquad 
b' = b^{-1} (a^{-1} b a') d' \]
or alternatively 
\[c = a, \qquad a'^{-1} b^{-1} a b a' b^{-1} a^{-1} b a' = c', \qquad d = a^{-1} b a', \qquad 
a'^{-1} b^{-1} a b b' = d'. \]
Expressed in terms of~$n$, these equations read
\[c = a, \qquad n a' n^{-1} = c', \qquad d = b n^{-1}, \qquad n b' = d', \]
which establishes our assertion.
\end{proof}

We note that, in the abelian case, our assertion reduces to the formula
\[\fN(\uu(a,b) \ot \uu(a',b')) =
\uu(a,b n^{-1}) \ot \uu(a',n b') =
\uu(a,a^{-1} b a') \ot \uu(a', a'^{-1} b' a). \]

It turns out that the function~$\n$ is actually invariant under~$\fN$:
\begin{Lemma} \label{InvarN}
$\n \circ \fN = \n$
\end{Lemma}
\begin{proof}
If we set again
$n \deq \n(\uu(a,b) \ot \uu(a',b')) = a'^{-1} b^{-1} a b$
and
\[c \deq a \qquad d \deq b n^{-1} \qquad c' \deq n a' n^{-1} \qquad d' \deq n b', \]
we have by Proposition~\ref{FormN} that
$\fN(\uu(a,b) \ot \uu(a',b')) = \uu(c,d) \ot \uu(c',d')$. Therefore we get
\begin{align*}
(\n \circ \fN)(\uu(a,b) \ot \uu(a',b')) &= c'^{-1} d^{-1} c d
= (n a'^{-1} n^{-1}) (n b^{-1}) a (b n^{-1}) \\
&= n (a'^{-1} b^{-1} a b) n^{-1} = n
\end{align*}
as asserted.
\end{proof}

We have seen in Paragraph~\ref{AltRMat} that the endomorphism corresponding to~$\fN$ for the alternative R-matrix~$\tilde{R}$ is just~$\tilde{\fN} = \fN^{-1}$. With the help of the preceding lemma, this endomorphism is easy to compute:
\begin{Cor} \label{InvN}
Suppose that $n \deq \n(\uu(c,d) \ot \uu(c',d')) = c'^{-1} d^{-1} c d$. Then we have
\[\fN^{-1}(\uu(c,d) \ot \uu(c',d')) =
\uu(c,d n) \ot \uu(n^{-1} c' n, n^{-1} d').\]
\end{Cor}
\begin{proof}
Since~$\fN$ permutes the basis elements, there exist~$a,b,a',b' \in G$ such that
$\fN(\uu(a,b) \ot \uu(a',b')) = \uu(c,d) \ot \uu(c',d')$. In view of the preceding lemma and Proposition~\ref{FormN}, this means that
$\uu(c,d) \ot \uu(c',d') = \uu(a, b n^{-1}) \ot \uu(n a' n^{-1}, n b')$. The assertion then follows by solving this equation for~$a$, $b$, $a'$, and~$b'$.
\end{proof}

We note that, in the case where the elements~$c$, $d$, $c'$, and~$d'$ commute, the formula for~$\tilde{\fN}$ given above reduces to
\[\tilde{\fN}(\uu(c,d) \ot \uu(c',d')) =
\uu(c, c d c'^{-1}) \ot \uu(c', c^{-1} c' d').\]
In just the same way as it was discussed in Paragraph~\ref{EndomAB} for the other endomorphisms, the matrix~$\fn$ introduced in Paragraph~1.11 can be considered as a variant of this map.

The analogue of Proposition~\ref{TSABComm} for~$\fN$ is the following statement:
\begin{Proposition} \label{NComm}
Suppose that
\[\fN(\uu(a,b) \ot \uu(a',b')) = \uu(c,d) \ot \uu(c',d').\]
Then we have $[a,b^{-1}] [a',b'^{-1}] = [c,d^{-1}] [c',d'^{-1}]$.
\end{Proposition}
\begin{proof}
If we set
$n \deq \n(\uu(a,b) \ot \uu(a',b')) = a'^{-1} b^{-1} a b$,
we know from Proposition~\ref{FormN} that
$c = a$ and  $d = b n^{-1}$ and also that $c' = n a' n^{-1}$ and $d' = n b'$.
Therefore we have
\begin{align*}
[c,d^{-1}] [c',d'^{-1}] &= (c d^{-1} c^{-1} d) (c' d'^{-1} c'^{-1} d') \\
&= a (n b^{-1}) a^{-1} (b n^{-1}) (n a' n^{-1}) (b'^{-1} n^{-1}) (n a'^{-1} n^{-1}) (n b')
\\
&= a n (b^{-1} a^{-1} b a') n^{-1} b'^{-1} a'^{-1} b'
= a n n^{-1} n^{-1} b'^{-1} a'^{-1} b' \\
&= a n^{-1} b'^{-1} a'^{-1} b' = a (b^{-1} a^{-1} b a') b'^{-1} a'^{-1} b'
= [a,b^{-1}] [a',b'^{-1}]
\end{align*}
as asserted.
\end{proof}

As in the case of the other maps discussed at the end of Paragraph~\ref{EndomAB}, the preceding Proposition~\ref{NComm} can also be derived in a different way: By construction, $\fN$ is \mbox{$D$-linear}, so that we have
\[\fN((\delta_h \ot e).(\uu(a,b) \ot \uu(a',b'))) =
(\delta_h \ot e).\fN(\uu(a,b) \ot \uu(a',b')).\]
But as we will show in the proof of Lemma~\ref{LemBasis} below, we have
\[(\delta_h \ot e).(\uu(a,b) \ot \uu(a',b')) =
\delta_{[a,b^{-1}] [a',b'^{-1}], h^{-1}} \uu(a,b) \ot \uu(a',b').\]
This immediately yields the assertion.

\subsection{The natural transformation~$\fC$} \label{NatTrC}
We now introduce an important new structure element related to the R-matrix:
\begin{Definition} \label{DefC}
For an $H$-module~$V$, we define $\fC_V \colon L \ot V \to L \ot V$ as the linear map
\[\fC_V(\psi \ot v) = \sum_{h \in G} (\delta_h \ot e).\psi \ot h.v\]
\end{Definition}
In view of Proposition~\ref{ActCoend}, we then have in particular that
\[\fC_V(\uu(a, b) \ot v) = \uu(a, b) \ot [b^{-1}, a].v\]
for $a,b \in G$ and~$v \in V$, so that
$\fC_V^{-1}(\uu(a, b) \ot v) = \uu(a, b) \ot [a, b^{-1}].v$. The definition applies in particular to a $D$-module, because every $D$-module~$V$ becomes an \mbox{$H$-module} by restriction to~$H$, i.e., via $h.v \deq (\varepsilon \ot h).v$.

The first property of~$\fC_V$ is the following:
\begin{Lemma} \label{LinC}
$\fC_V$ is $H$-linear.
\end{Lemma}
\begin{proof}
For $a,b,c \in G$, we have by Proposition~\ref{ActCoend} that
\begin{align*}
\fC_V(c.(\uu(a, b) \ot v)) &= \fC_V(\uu(cac^{-1},cbc^{-1}) \ot c.v) \\
&= \uu(cac^{-1},cbc^{-1}) \ot [(cbc^{-1})^{-1}, cac^{-1}]c.v \\
&= \uu(cac^{-1},cbc^{-1}) \ot (c[b^{-1}, a] c^{-1})c.v \\
&= \uu(cac^{-1},cbc^{-1}) \ot c[b^{-1}, a].v
= c.\fC_V(\uu(a, b) \ot v)
\end{align*}
as required.
\end{proof}
If~$f \colon V \to W$ is an $H$-linear map between two \mbox{$H$-modules}, we clearly have
\[(\id_L \ot f) \circ \fC_V = \fC_W \circ(\id_L \ot f),\]
so that~$\fC$ is a natural transformation from a certain functor to itself, namely the functor that assigns to an $H$-module~$V$ the $H$-module~$L \ot V$. We record the following special case of this fact:
\begin{Cor} \label{CommC}
For two $H$-modules~$V$ and~$W$, we have
\[\fC_{V \ot L \ot W} \circ (\id_{L \ot V} \ot \fC_W) =
(\id_{L \ot V} \ot \fC_W) \circ \fC_{V \ot L \ot W}.\]
\end{Cor}
\begin{proof}
This is the naturality of~$\fC$ just discussed in the case where
$f= \id_{V} \ot \fC_W$.\qedhere
\end{proof}

The natural transformation~$\fC$ also satisfies commutation relations with endomorphisms of~$L$:
\begin{Proposition} \label{PropertC}
Suppose that~$V$ is an $H$-module. Then we have
\begin{enumerate}
\item
$(\fT \ot \id_V) \circ \fC_V = \fC_V \circ (\fT \ot \id_V)$

\item
$(\fS \ot \id_V) \circ \fC_V = \fC_V \circ (\fS \ot \id_V)$

\item
$(\fA \ot \id_V) \circ \fC_V^{-1} = \fC_V \circ (\fA \ot \id_V)$

\item
$(\fB \ot \id_V) \circ \fC_V^{-1} = \fC_V \circ (\fB \ot \id_V)$
\end{enumerate}
\end{Proposition}
\begin{proof}
This is a direct consequence of Proposition~\ref{TSABComm}. We show only the third statement; the proof of the other three is very similar. If
$\fA(u(a,b)) = u(c,d)$, we know from Proposition~\ref{TSABComm} that
$[c,d^{-1}] = [a,b^{-1}]^{-1} = [b^{-1},a]$. Therefore we get
\begin{align*}
(\fC_V^{-1} \circ (&\fA \ot \id_V))(\uu(a, b) \ot v) = \fC_V^{-1}(\uu(c,d) \ot v)
= \uu(c,d) \ot [c,d^{-1}].v \\
&= \uu(c,d) \ot [b^{-1},a]v = (\fA \ot \id_V)(\uu(a,b) \ot [b^{-1},a].v) \\
&= ((\fA \ot \id_V) \circ \fC_V)(\uu(a, b) \ot v)
\end{align*}
as asserted.
\end{proof}
We note that, in the case $V=L$, we also have
\[(\id_L \ot \fA) \circ \fC_L = \fC_L \circ(\id_L \ot \fA)\]
and similarly for~$\fT$,~$\fS$, and~$\fB$. This is a consequence of the naturality of~$\fC$ discussed above, as~$\fT$,~$\fS$,~$\fA$, and~$\fB$ are $H$-linear, as pointed out at the end of Paragraph~\ref{EndomAB}.

With the help of the natural transformation~$\fC$, we can now show that~$\fN$ and~$\fN^{-1}$ are conjugate. We will derive this fact from a more general formula that shows that
$\fN \ot \id_V$ and~$\fN^{-1} \ot \id_V$ are conjugate:
\begin{Proposition} \label{PropertN}
For any $H$-module~$V$, we have
\begin{align*}
&(\fB \ot \fB \ot \id_V)  \circ \fC_{L \ot V}^{-1} \circ (\id_L \ot \fC_V^{-1})
\circ (\fN \ot \id_V) \\
& \quad =
(\fN^{-1} \ot \id_V) \circ (\fB \ot \fB \ot \id_V)  \circ \fC_{L \ot V}^{-1}
\circ (\id_L \ot \fC_V^{-1}).
\end{align*}
\end{Proposition}
\begin{proof}
\begin{pflist}
\item
If we set $n \deq \nu(\uu(a,b) \ot \uu(a',b')) = a'^{-1} b^{-1} a b$, we have
\begin{align*}
&(\fC_{L \ot V}^{-1} \circ (\id_L \ot \fC_V^{-1}) \circ (\fN \ot \id_V))
(\uu(a,b) \ot \uu(a',b') \ot v) = \\
&(\fC_{L \ot V}^{-1} \circ (\id_L \ot \fC_V^{-1}))
(\uu(a, bn^{-1}) \ot \uu(na'n^{-1}, nb') \ot v) =  \\
&\fC_{L \ot V}^{-1}
(\uu(a, bn^{-1}) \ot \uu(na'n^{-1}, nb') \ot [na'n^{-1}, b'^{-1} n^{-1}].v) =  \\
&\uu(a, bn^{-1}) \ot
\uu([a, nb^{-1}]na'n^{-1}[a, nb^{-1}]^{-1}, [a, nb^{-1}]nb'[a, nb^{-1}]^{-1} \\
&\mspace{300mu} \ot [a, nb^{-1}][na'n^{-1}, b'^{-1} n^{-1}].v).
\end{align*}
By Proposition~\ref{NComm}, this expression is equal to
\begin{align*}
&\uu(a, bn^{-1}) \ot
\uu([a, nb^{-1}]na'n^{-1}[a, nb^{-1}]^{-1}, [a, nb^{-1}]nb'[a, nb^{-1}]^{-1}
\ot [a, b^{-1}][a', b'^{-1}].v).
\end{align*}
Since we have
$[a, nb^{-1}] n = a nb^{-1} a^{-1} b n^{-1} n = a nb^{-1} a^{-1} b = a a'^{-1}$,
this expression can be rewritten further as
\begin{align*}
&\uu(a, bn^{-1}) \ot
\uu((a a'^{-1}) a' (a' a^{-1}), (a a'^{-1}) b' (n a' a^{-1}))
\ot [a, b^{-1}][a', b'^{-1}].v = \\
&\uu(a, bn^{-1}) \ot \uu(a a' a^{-1}, a a'^{-1} b' n a' a^{-1})
\ot [a, b^{-1}][a', b'^{-1}].v.
\end{align*}
If we apply~$\fB$ to the first tensor factor in this expression, we obtain
\[\fB(\uu(a, bn^{-1})) = \uu(a^{-1}, abn^{-1}a^{-1}) = \uu(a^{-1}, b a' a^{-1}),\]
and for the second tensor factor we get similarly that
\begin{align*}
\fB(\uu(a a' a^{-1}, a a'^{-1} b' n a' a^{-1})) &=
\uu(a a'^{-1} a^{-1}, (a a' a^{-1}) a a'^{-1} b' n a' a^{-1} (a a'^{-1} a^{-1})) \\
&= \uu(a a'^{-1} a^{-1}, a b' n a^{-1})
= \uu(a a'^{-1} a^{-1}, a b' a'^{-1} b^{-1} a b a^{-1}).
\end{align*}
In total, we therefore have
\begin{align*}
&((\fB \ot \fB \ot \id_V)  \circ \fC_{L \ot V}^{-1} \circ (\id_L \ot \fC_V^{-1})
\circ (\fN \ot \id_V))(\uu(a,b) \ot \uu(a',b') \ot v) = \\
&\uu(a^{-1}, b a' a^{-1}) \ot \uu(a a'^{-1} a^{-1}, a b' a'^{-1} b^{-1} a b a^{-1})
\ot [a, b^{-1}][a', b'^{-1}].v.
\end{align*}

\item
On the other hand, if we define $c \deq [a,b^{-1}] = a b^{-1} a^{-1} b$, we have
\begin{align*}
&((\fB \ot \fB \ot \id_V)  \circ \fC_{L \ot V}^{-1}
\circ (\id_L \ot \fC_V^{-1}))(\uu(a,b) \ot \uu(a',b') \ot v) = \\
&((\fB \ot \fB \ot \id_V)  \circ \fC_{L \ot V}^{-1})
(\uu(a,b) \ot \uu(a',b') \ot [a', b'^{-1}].v) = \\
&(\fB \ot \fB \ot \id_V)
(\uu(a,b) \ot \uu(c a' c^{-1}, c b' c^{-1}) \ot c [a', b'^{-1}].v) = \\
&\uu(a^{-1}, a b a^{-1}) \ot \uu(c a'^{-1} c^{-1}, c a' b' a'^{-1} c^{-1})
 \ot c [a', b'^{-1}].v.
\end{align*}
We now set
\begin{align*}
n' &\deq \nu(\uu(a^{-1}, a b a^{-1}) \ot \uu(c a'^{-1} c^{-1}, c a' b' a'^{-1} c^{-1})) \\
&\mspace{5mu}
= c a' c^{-1} (a b^{-1} a^{-1}) a^{-1} (a b a^{-1})
= c a' c^{-1} (a b^{-1} a^{-1} b) a^{-1}
= c a' a^{-1},
\end{align*}
so that $n'^{-1} = a a'^{-1} c^{-1}$. Then we have by Corollary~\ref{InvN} that
\begin{align*}
&((\fN^{-1} \ot \id_V) \circ (\fB \ot \fB \ot \id_V)  \circ \fC_{L \ot V}^{-1}
\circ (\id_L \ot \fC_V^{-1}))(\uu(a,b) \ot \uu(a',b') \ot v) = \\
&\uu(a^{-1}, a b a^{-1} n') \ot
\uu(n'^{-1} c a'^{-1} c^{-1} n', n'^{-1} c a' b' a'^{-1} c^{-1})
 \ot c [a', b'^{-1}].v= \\
&\uu(a^{-1}, a b a^{-1} (c a' a^{-1})) \ot
\uu((a a'^{-1}) a'^{-1} (a' a^{-1}), (a a'^{-1}) a' b' a'^{-1} c^{-1})
 \ot c [a', b'^{-1}].v = \\
&\uu(a^{-1}, b a' a^{-1}) \ot \uu(a a'^{-1} a^{-1}, a b' a'^{-1} c^{-1})
 \ot c [a', b'^{-1}].v = \\
&\uu(a^{-1}, b a' a^{-1}) \ot \uu(a a'^{-1} a^{-1} ,a b' a'^{-1} (b^{-1} a b a^{-1}))
 \ot c [a', b'^{-1}].v
\end{align*}
as required to establish the assertion.
\qedhere
\end{pflist}
\end{proof}

We note that the two maps $\fC_{L \ot V}^{-1}$ and $\id_L \ot \fC_V^{-1}$ that appear in Proposition~\ref{PropertN} above actually commute by Corollary~\ref{CommC}.
We also note that in the case $V=K$, the base field considered as a trivial $H$-module, Proposition~\ref{PropertN} implies that
\[(\fB \ot \fB)  \circ \fC_L^{-1} \circ \fN =
\fN^{-1}  \circ (\fB \ot \fB) \circ \fC_L^{-1},\]
which shows that~$\fN$ and~$\fN^{-1}$ are conjugate.

\subsection{Representations of the Drinfel'd double} \label{RepDrinf} 
The modules of the Drinfel'd double can be described via centralizers, even if the characteristic of the base field is positive. We briefly review this description, variants of which can be found in several places in the literature, for example in~\cite[Cor.~2.3, p.~314]{W1}:
\begin{Proposition} \label{PropRepDrinf}
Suppose that $h \in G$ and that~$V$ is a representation of its centralizer $\Cent(h)$. Then the induced module
\[I(h, V) \deq K[G] \ot_{K[\Cent(h)]} V \]
becomes a module over~$D$ via
\[(\delta_a \ot b).(c \ot v) = \delta_{b^{-1}ab, chc^{-1}} \, bc \ot v, \]
so that in particular
$(\delta_a \ot e).(c \ot v) = \delta_{a, chc^{-1}} \, c \ot v$ and
$(\varepsilon \ot b).(c \ot v) = bc \ot v$.
\end{Proposition}
\begin{proof}
The modules over the Drinfel'd double can be described as left-right Yetter-Drinfel'd modules, also known as crossed bimodules (cf.~\cite[Thm.~IX.5.2, p.~221]{Ka}). Because~$K[G]$ is cocommutative, these are the same as left-left Yetter-Drinfel'd modules, so that the assertion follows from~\cite[Def.~1.4.15, p.~32]{HeS}.  
\end{proof}

For a different approach to this proposition from the viewpoint of Hopf algebra extensions, see~\cite[Prop.~7.4, p.~54]{KSZ} and the references given there. 

An arbitrary module of the Drinfel'd double can in fact be built from the special modules just introduced:

\begin{Proposition}
Suppose that $C_1,\ldots,C_r$ are the conjugacy classes of~$G$, and that $h_i \in C_i$ is a representative, where~$C_1 = \{e\}$ and therefore~$h_1=e$. If~$W$ is a module over the Drinfel'd double~$D$ of~$G$, then there are $\Cent(h_i)$-modules~$V_i$, for $i=1,\ldots,r$, so that
\[W \cong \bigoplus_{i=1}^r I(h_i, V_i).\]
The module~$W$ is projective if and only if all of the modules $V_1,\ldots,V_r$ are projective. 
Furthermore, $W$ is simple, or indecomposable, if and only if there is a unique~\mbox{$j \le r$} such that~$V_j$ is nonzero, and this unique module~$V_j$ is simple or indecomposable, respectively. 
\end{Proposition}
\begin{proof}
The elements $\delta_a \ot e$, for $a \in G$, form a complete system of orthogonal idempotents. Accordingly, $W$ is the direct sum of the subspaces $W_a$ defined as
\[W_a \deq \{w \in W \mid (\delta_a \ot e).w = w \}\]
and called the homogeneous components of degree~$a$. The formula for the product in the Drinfel'd double shows that $a.W_b = W_{aba^{-1}}$. In particular, $W_a$ is a module over~$\Cent(a)$, so that we can define $V_i \deq W_{h_i}$.

For~$i=1,\ldots,r$, we let $I_i$ be the subspace spanned by $\delta_a \ot b$ for all $a \in C_i$ and~$b \in G$. Then it follows again from the formula for the product in the Drinfel'd double that~$D$ is the direct sum of the ideals~$I_i$. This leads to a corresponding decomposition of~$W$ (cf.~\cite[Chap.~0, Exerc.~6, p.~17]{FD}), in which~$W$ is the direct sum of the modules $U_i \deq \sum_{a \in C_i} W_a$. These are $D$-submodules of~$W$ that can also be considered as modules over~$I_i$. By \cite[Chap.~I, Prop.~4.5, p.~24]{HiS}, $W$ is projective over~$D$ if and only if every~$U_i$ is projective over~$D$, which happens if and only if every~$U_i$ is projective over~$I_i$.

Now it is shown in~\cite[Lem.~1.4.16, p.~32]{HeS} that the map 
\[I(h_i, V_i) \to U_i,~a \ot w \mapsto a.w\]
is a $D$-linear isomorphism and that~$V_i$ is isomorphic to a direct summand of $I(h_i, V_i)$ restricted to~$\Cent(h_i)$. This implies the statement about the general decomposition and, after a short argument, also the statement about projectivity.

If~$W$ is indecomposable or even simple, there can be only one~$U_j$ that is nonzero. But a decomposition or submodule of~$V_j$ leads to a decomposition or submodule of 
$U_j \cong K[G] \ot_{K[\Cent(h_j)]} V_j$, and conversely a decomposition or submodule of~$U_j$ leads by another short argument to a decomposition or submodule of~$V_j = W_{h_j} \subset U_j$. This proves the remaining assertions.
\end{proof}

The case $h = e$ will play a special role in the sequel:
\begin{Cor} \label{TrivConjClass}
Every $K[G]$-module~$V$ becomes a module over~$D = D(G)$ by defining
\[(\delta_a \ot b).v \deq \delta_{a,e} \, b.v.\]
If~$V$ is projective as a module over~$K[G]$, it is also projective as a module over~$D$.
\end{Cor}
\begin{proof}
If $h = e$, we have $\Cent(h) = G$. With the help of the isomorphism 
\[I(e, V) = K[G] \ot_{K[G]} V \to V,~g \ot v \mapsto g.v\]
the assertions follow from the preceding propositions.
\end{proof}

The $D(G)$-module structure on~$V$ can of course be viewed as the pullback along the algebra homomorphism
\[D(G) \to K[G],~\delta_a \ot b \mapsto \delta_{a,e} \, b.\]
Pullback along this homomorphism defines a functor from the category of \mbox{$K[G]$-mod}\-ules to the category of $D(G)$-modules. This functor has a right adjoint, which assigns to a $D(G)$-module~$W$ the $K[G]$-module
\[W_e = \{w \in W \mid (\delta_e \ot e).w = w \}.\]
This adjunction can be deduced from the usual hom-tensor adjunction (cf.~\cite[Lem.~19.11, p.~225]{AF}) for the bimodule~$K[G]$, viewed as a left $D(G)$-module via the pullback just described and as a right $K[G]$-module via multiplication, as the map
\[\Hom_{D(G)}(K[G],W) \to W_e,~f \mapsto f(e)\]
is an isomorphism of $K[G]$-modules.

\pagebreak

We will need another general fact about the action of the Drinfel'd element:
\begin{Lemma} \label{ActDrinf}
Suppose that~$W$ is a $D(G)$-module and that~$w \in W_a$ for $a \in G$. Then we have $u_D.w = a^{-1}.w \in W_a$ and $u_D^{-1}.w = a.w \in W_a$.
\end{Lemma}
\begin{proof}
As already pointed out above, we have for $b \in G$ that~$b.w \in W_{bab^{-1}}$ and therefore
\[(\delta_{b^{-1}} \ot b).w = (\delta_{b^{-1}} \ot e).(b.w)
= \delta_{b^{-1},bab^{-1}} b.w = \delta_{b^{-1}, a} \, b.w\]
and similarly $(\delta_{b} \ot b).w = \delta_{b, a} \, b.w$. This implies that
\[u_D.w = \sum_{b \in G} (\delta_{b^{-1}} \ot b).w = a^{-1}.w\]
and similarly $u_D^{-1}.w= \sum_{b \in G} (\delta_{b} \ot b).w = a.w$. Of course, one equation follows from the other.
\end{proof}

This lemma is in fact a special case of a more general property of finite-dimensional Hopf algebras (cf.~\cite[Par.~3.9, p.~445]{So}).

\enlargethispage{5mm}
\vspace{4mm}

\subsection{A special module} \label{SpecMod}
As we have discussed at the end of Paragraph~2.6, the module
$W \deq L^{\ot g} = (D^*)^{\ot g}$ carries an action of the mapping class group~$\Gamma_{g,1}$. We will use the action that arises from the alternative R-matrix~$\tilde{R}$, because it leads to simpler formulas; we will show in Paragraph~\ref{StandRMat} that the  representation of~$\Gamma_{g,1}$ that arises from the standard \mbox{R-matrix} is isomorphic to the present one. Because our action is by~$D$-linear maps, it preserves the homogeneous components~$W_h$. Such a homogeneous component has the following basis:
\vspace{1mm}
\begin{Lemma} \label{LemBasis}
For $h \in G$, a basis of the vector space~$W_h$ consists of those decomposable tensors
\[\uu(a_1,b_1) \ot \dots \ot \uu(a_g,b_g)\]
for which the elements
$a_1, b_1, \ldots, a_g, b_g \in G$ satisfy $[a_1, b_1^{-1}] [a_2, b_2^{-1}] \cdots [a_g, b_g^{-1}] = h^{-1}$.
\end{Lemma}
\begin{proof}
The formula
$(\delta_h \ot e).\uu(a,b) = \delta_{h, [b^{-1}, a]} \; \uu(a,b)$
from Proposition~\ref{ActCoend} implies that
\begin{align*}
&(\delta_h \ot e).(\uu(a_1,b_1) \ot \dots \ot \uu(a_g,b_g)) \\
&\qquad =
\sum_{\substack{h_1,\ldots,h_g \in G \\ h_g \cdots h_2 h_1 = h}}
(\delta_{h_1} \ot e).\uu(a_1,b_1)  \ot \dots \ot (\delta_{h_g} \ot e).\uu(a_g,b_g) \\
&\qquad =
\sum_{\substack{h_1,\ldots,h_g \in G \\ h_g \dots h_2 h_1 = h}}
\delta_{h_1, [b_1^{-1}, a_1]} \; \uu(a_1,b_1)  \ot \dots \ot
\delta_{h_g, [b_g^{-1}, a_g]} \; \uu(a_g,b_g) \\
&\qquad =
\begin{cases}
\uu(a_1,b_1) \ot \dots \ot \uu(a_g,b_g) &: [b_g^{-1}, a_g] \cdots [b_1^{-1}, a_1]  = h \\
0 &:  [b_g^{-1}, a_g] \cdots [b_1^{-1}, a_1] \neq h
\end{cases}
\end{align*}
which implies our assertion by inversion, since $[b_i^{-1}, a_i]^{-1} = [a_i, b_i^{-1}]$.
\end{proof}

Now recall the capping homomorphism
$C_1 \colon \Gamma_{g,1} \to \Gamma_{g}(y) = \Gamma_{g,0}(y)$ from Paragraph~1.9, which arose by gluing a punctured disk, whose marked point was denoted by~$y$, in the place of the boundary component. We claim that the action of our mapping class group descends along the capping homomorphism:
\begin{Proposition} \label{DescCap}
The action of~$\Gamma_{g,1}$ on~$W_e$ descends along the capping homomorphism to an action of~$\Gamma_{g}(y)$.
\end{Proposition}
\begin{proof}
In view of the capping sequence from Proposition~1.8 in Paragraph~1.9, we have to show that~$\fd_1$ acts trivially. As we already pointed out at the end of Paragraph~2.6, $\fd_1$ acts by~$\theta_W$, in other words, by applying the ribbon element~$u_D$ for the R-matrix~$\tilde{R}$. But that this action is trivial on~$W_e$ follows immediately from Lemma~\ref{ActDrinf}.
\end{proof}

We now encounter again the small conceptual difficulty with the polygon model that we have already encountered in Paragraph~1.12, namely that the point~$y$ just introduced is different from the base point~$x$ that is the common image of all the vertices of the polygon, except for the start point and the end point of the edge labeled by~$\rho_1$. Almost as in Paragraph~1.12, we choose a diffeomorphism~$\phi \colon \Sigma_{g,0} \to \Sigma_{g,0}$ that is isotopic to the identity and satisfies~$\phi(y) = x$. We can assume that~$\phi$ restricts to the identity outside a closed embedded disk that contains both~$x$ and~$y$ in its interior and inside this disk just pushes~$y$ to~$x$. We then have the isomorphism
\[ \Gamma_{g,0}(y) \to \Gamma_{g,0}(x),~[\psi] \mapsto [\phi \circ \psi \circ \phi^{-1}]\]
that satisfies $F_x([\phi \circ \psi \circ \phi^{-1}]) = F_y([\psi])$. We can use this isomorphism to transport the action of~$\Gamma_{g}(y)$ on~$W_e$ to an action of~$\Gamma_{g}(x)$. For the maps introduced in Paragraph~1.7, we have by~\cite[Sec.~3.3, Fact~3.7, p.~73]{FM} that
\[[\ft'_i] = [\phi \circ \ft_i \circ \phi^{-1}], \qquad
[\fr''_i] = [\phi \circ \fr_i \circ \phi^{-1}], \qquad
[\fn''_i] = [\phi \circ \fn_i \circ \phi^{-1}],\]
and therefore also $[\fs'_i] = [\phi \circ \fs_i \circ \phi^{-1}]$, as follows from the discussion given there in the proof of Proposition~1.6.

From the Dehn-Lickorish theorem~1.4 in Paragraph~1.7, the mapping classes
$[\ft_i]$, $[\fr_i]$, and $[\fn_l]$, for $i=1,\dots,g$ and $l=1,\dots, g-1$, generate~$\Gamma_{g,1}$, and therefore also~$\Gamma_{g}(y)$ by Proposition~1.8 in Paragraph~1.9. Consequently, the mapping classes~$[\ft'_i]$, $[\fr''_i]$, and $[\fn''_l]$ generate~$\Gamma_{g}(x)$. Alternatively, we can also use $\fs'_i = \ft_i'^{-1} \fr_i''^{-1} \ft_i'^{-1}$, which was also defined in Paragraph~1.7, instead of~$\fr''_i$.

\subsection{The relation with the representation variety} \label{RelRepVar}
As we discussed in Paragraph~1.6, the group $\Gamma_{g}(x)$ acts on the fundamental group $\pi_1(\Sigma_{g},x)$. Therefore, this group also acts on the set
\[\Omega^g \deq \Hom(\pi_1(\Sigma_{g},x),G)\]
of group homomorphisms from the fundamental group to~$G$ by defining
\[([\psi].f)([\gamma]) \deq f([\psi^{-1} \circ \gamma])  \]
for $[\psi] \in \Gamma_{g}(x)$, $[\gamma] \in \pi_1(\Sigma_{g},x)$, and
$f \in \Omega^g$. The set~$\Omega^g$ is often called the representation variety of~$\pi_1(\Sigma_{g},x)$ in~$G$ (cf.~\cite[Sec.~2, p.~5]{GM}; \cite[p.~339]{PX}; \cite[Sec.~5, p.~5182]{Si}). By linear extension, we therefore obtain an action
of~$\Gamma_{g}(x)$ on the vector space~$M^g \deq K[\Omega^g]$ with basis~$\Omega^g$, which we call the linearized representation variety. We will show now that this $\Gamma_{g}(x)$-representation is isomorphic to the one considered in Paragraph~\ref{SpecMod}. For this, recall from Paragraph~1.2 that the fundamental group~$\pi_1(\Sigma_{g},x)$ is generated by the homotopy classes $[\alpha_1],[\beta_1],\ldots,[\alpha_g],[\beta_g]$ that satisfy the one defining relation
\[ [[\alpha_1], [\beta_1] ] \cdot [[\alpha_2], [\beta_2]] \cdots [[\alpha_g], [\beta_g]] = 1\]
in which the commutator brackets need to be distinguished from the homotopy class brackets.
\begin{Theorem} \label{ThmModTilde}
The map 
$\tilde{\Theta} \colon M^g \to W_e = (L^{\ot g})_e$ that maps a basis element
\mbox{$f \in \Hom(\pi_1(\Sigma_{g},x),G)$} of~$M^g$ to 
\[\tilde{\Theta}(f) \deq \uu(a_1,b_1^{-1}) \ot \dots \ot \uu(a_g,b_g^{-1}),\]
where $a_i \deq f([\alpha_i])$ and $b_i \deq f([\beta_i])$ for $i=1,\ldots,g$, is an isomorphism of the \mbox{$\Gamma_{g}(x)$-modules} $M^g$ and~$W_e$.
\end{Theorem}
\begin{proof}
\begin{pflist}
\item
In view of the defining relation of the fundamental group recalled above, Lemma~\ref{LemBasis} implies that~$\tilde{\Theta}$ is well-defined and bijective. As we also recalled above, the mapping classes~$[\ft'_i]$, $[\fs'_i]$, and $[\fn''_l]$ generate~$\Gamma_{g}(x)$, and of course their inverses generate this group, too. It is therefore sufficient to show the equivariance of~$\tilde{\Theta}$ for these generators. For~$[\ft'_i]$, we have by Proposition~1.6 in Paragraph~1.7 that
$([\ft_i'^{-1}].f)([\beta_i]) = f(\ft_i'([\beta_i])) = f([\beta_i \alpha_i]) = b_i a_i$ and therefore
\begin{align*}
\tilde{\Theta}([\ft_i'^{-1}].f) = \uu(a_1,b_1^{-1}) \ot \dots
\ot \uu(a_i,a_i^{-1} b_i^{-1}) \ot \dots \ot \uu(a_g,b_g^{-1}).
\end{align*}
On the other hand, as we are working with the alternative R-matrix~$\tilde{R}$ and~$\tilde{\fT} = \fT^{-1}$ by Proposition~\ref{PropAltRMat}, we get from the formulas in Paragraph~2.6 that
\begin{align*}
[\ft_i'^{-1}].\tilde{\Theta}(f) &=
(\id_{L^{\ot (i-1)}} \ot \tilde{\fT}^{-1} \ot \id_{L^{\ot (g-i)}}).\tilde{\Theta}(f) \\
&=
\uu(a_1,b_1^{-1}) \ot \dots \ot \fT(\uu(a_i,b_i^{-1})) \ot \dots \ot \uu(a_g,b_g^{-1}).
\end{align*}
Both expressions are equal by Proposition~\ref{FormST}.

\item
The computations for the generator~$[\fs'_i]$ are similar: We have by Proposition~1.6 that
\[([\fs_i'^{-1}].f)([\alpha_i]) = f(\fs_i'([\alpha_i])) = f([\alpha_i \beta_i \alpha_i^{-1}]) = a_i b_i a_i^{-1}\]
and
$([\fs_i'^{-1}].f)([\beta_i]) = f(\fs_i'([\beta_i])) = f([\alpha_i^{-1}]) = a_i^{-1}$, so that
\begin{align*}
\tilde{\Theta}([\fs_i'^{-1}].f) = \uu(a_1,b_1^{-1}) \ot \dots \ot
\uu(a_i b_i a_i^{-1},a_i) \ot \dots \ot \uu(a_g,b_g^{-1}).
\end{align*}
On the other hand, as~$\tilde{\fS} = \fS^{-1}$ by Lemma~\ref{LemAltRMat}, we get from the formulas in Paragraph~2.6 that
\begin{align*}
[\fs_i'^{-1}].\tilde{\Theta}(f) &=
(\id_{L^{\ot (i-1)}} \ot \tilde{\fS}^{-1} \ot \id_{L^{\ot (g-i)}}).\tilde{\Theta}(f) \\
&=
\uu(a_1,b_1^{-1}) \ot \dots \ot \fS(\uu(a_i,b_i^{-1})) \ot \dots \ot \uu(a_g,b_g^{-1}).
\end{align*}
Again, both expressions are equal by Proposition~\ref{FormST}.

\item
The last generator~$[\fn_i'']$, for $i=1,\dots, g-1$, is computationally the most complicated. Using the curve~$\mu_{i} = \alpha_{i+1}^{-1} \beta_i \alpha_i \beta_i^{-1}$ introduced in Paragraph~1.2, we have by Proposition~1.6 that
$([\fn_i''^{-1}].f)([\alpha_i]) = f(\fn_i''([\alpha_i])) = f([\alpha_i]) = a_i$
and
\[([\fn_i''^{-1}].f)([\beta_i]) = f(\fn_i''([\beta_i])) = f([\mu_i \beta_i])
= m_i b_i\]
with $m_i \deq f(\mu_i) = a_{i+1}^{-1} b_i a_i b_i^{-1}$. But in the case of this generator, we need to consider in addition that
$\fn''_i([\alpha_{i+1}]) = [\mu_i \alpha_{i+1} \mu_i^{-1}]$ and
$\fn''_i([\beta_{i+1}]) = [\beta_{i+1} \mu_i^{-1}]$, which implies that
\[([\fn_i''^{-1}].f)([\alpha_{i+1}]) = m_i a_{i+1} m_i^{-1}
\qquad \text{and} \qquad
([\fn_i''^{-1}].f)([\beta_{i+1}]) = b_{i+1} m_i^{-1}.\]
By inserting these equations, we get that $\tilde{\Theta}([\fn_i''^{-1}].f)$ is equal to
\begin{align*}
&\uu(a_1,b_1^{-1}) \ot \dots \ot \uu(a_i,b_i^{-1} m_i^{-1})
\ot \uu(m_i a_{i+1} m_i^{-1}, m_i b_{i+1}^{-1}) \ot \dots \ot \uu(a_g,b_g^{-1}).
\end{align*}
On the other hand, as~$\tilde{\fN} = \fN^{-1}$ by Proposition~\ref{PropAltRMat}, we get from the formulas in Paragraph~2.6 that
\begin{align*}
[\fn_i''^{-1}].\tilde{\Theta}(f) &=
(\id_{L^{\ot (i-1)}} \ot \tilde{\fN}^{-1} \ot \id_{L^{\ot (g-i-1)}}).\tilde{\Theta}(f) \\
&=
\uu(a_1,b_1^{-1}) \ot \dots \ot
\fN(\uu(a_i,b_i^{-1}) \ot \uu(a_{i+1},b_{i+1}^{-1})) \ot \dots \ot \uu(a_g,b_g^{-1}).
\end{align*}
Since $m_i = \nu(\uu(a_i,b_i^{-1}) \ot \uu(a_{i+1},b_{i+1}^{-1}))$, both expressions are equal by Proposition~\ref{FormN}.
\qedhere
\end{pflist}
\end{proof}

For $a \in G$ and $f \in \Hom(\pi_1(\Sigma_{g},x),G)$, the map
\[a.f \colon \pi_1(\Sigma_{g},x) \to G,~[\gamma] \mapsto a f([\gamma]) a^{-1}\]
is again a group homomorphism from the fundamental group~$\pi_1(\Sigma_{g},x)$ to~$G$. In this way, the adjoint action of~$G$ on itself leads to an action of~$G$
on the representation variety $\Omega^g = \Hom(\pi_1(\Sigma_{g},x),G)$ that can also be extended linearly to~$M^g$. Clearly, the action of~$G$ commutes with the action of~$\Gamma_{g}(x)$ on~$M^g$ introduced above.

On the other hand, we have seen in Paragraph~\ref{RepDrinf} that~$W_e$ is a $K[G]$-submodule of~$W = L^{\ot g}$. The map~$\tilde{\Theta}$ from Theorem~\ref{ThmModTilde} above is compatible with these \mbox{$G$-actions}:
\begin{Proposition} \label{GLin}
$\tilde{\Theta}$ is $K[G]$-linear.
\end{Proposition}
\begin{proof}
For \mbox{$f \in \Hom(\pi_1(\Sigma_{g},x),G)$}, we set
$a_i \deq f([\alpha_i])$ and $b_i \deq f([\beta_i])$ as in Theorem~\ref{ThmModTilde}. If~$a \in G$, we then have
\[\tilde{\Theta}(a.f) \deq \uu(a a_1 a^{-1}, a b_1^{-1} a^{-1}) \ot \dots
\ot \uu(a a_g a^{-1}, a b_g^{-1} a^{-1}).\]
But by Proposition~\ref{ActCoend}, this expression is equal to~$a.\tilde{\Theta}(f)$.
\end{proof}

\subsection{The endomorphisms~$\fG$ and~$\fU$} \label{EndomGU}
As already mentioned at the beginning of Paragraph~\ref{SpecMod}, there are two representations of the mapping class group~$\Gamma_{g,1}$ on the
space~$W \deq L^{\ot g} = (D^*)^{\ot g}$, one coming from the standard R-matrix of the Drinfel'd double and one coming from the alternative R-matrix~$\tilde{R}$. To show that these two representations are isomorphic, we need to introduce two new maps:
\begin{Definition} \label{DefGU}
For all $g \ge 2$, we define the map $\fG_g \colon L^{\ot g} \to L^{\ot g}$, starting for $g=2$ with $\fG_2 \deq \fC_L^{-1}$ and then setting recursively
\[\fG_g \deq \fC_{L^{\ot (g-1)}}^{-1} \circ (\id_L \ot \fG_{g-1}).\]
For $g=1$, we define $\fG_1 \deq \id_L$.
Furthermore, we define $\fU \deq \fB^{\ot g} \circ \fG_g$.
\end{Definition}
Stated differently, we have for $g \ge 2$ that
\[\fG_g \deq \fC_{L^{\ot (g-1)}}^{-1} \circ (\id_L \ot \fC_{L^{\ot (g-2)}}^{-1})
\circ (\id_{L^{\ot 2}} \ot \fC_{L^{\ot (g-3)}}^{-1})
\circ \cdots \circ (\id_{L^{\ot (g-2)}} \ot \fC_{L}^{-1}).\]
Of course, it is possible to express~$\fG_g$ in terms of elements:
\begin{Lemma}
Suppose that $a_1,b_1,a_2,b_2, \ldots, a_g,b_g \in G$. For $i=1,\dots,g$, define
$c_i \deq [a_1,b_1^{-1}] [a_2,b_2^{-1}] \cdots [a_{i-1},b_{i-1}^{-1}]$, so that $c_1=e$. Then we have
\[\fG_g(\uu(a_1,b_1) \ot \dots \ot \uu(a_g,b_g)) =
\uu(c_1 a_1 c_1^{-1}, c_1 b_1 c_1^{-1}) \ot \dots
\ot \uu(c_g a_g c_g^{-1}, c_g b_g c_g^{-1}).\]
\end{Lemma}
\begin{proof}
We argue by induction on~$g$. For $g=1$, the assertion is obvious. For $g=2$, the assertion follows from Proposition~\ref{ActCoend} and the formulas given after Definition~\ref{DefC}. For $g>2$, we define for $i \ge 2$ the element
$c'_i \deq  [a_2,b_2^{-1}] \cdots [a_{i-1},b_{i-1}^{-1}]$, so that $c'_2=e$, and therefore have $c_i = [a_1,b_1^{-1}] c'_i$. By the induction hypothesis, we have
\[\fG_{g-1}(\uu(a_2,b_2) \ot \dots \ot \uu(a_g,b_g)) =
\uu(c'_2 a_2 c'^{-1}_2, c'_2 b_2 c'^{-1}_2) \ot
\dots \ot \uu(c'_g a_g c'^{-1}_g, c'_g b_g c'^{-1}_g).\]
Since we have by Proposition~\ref{ActCoend} that
\begin{align*}
&\fC_{L^{\ot (g-1)}}^{-1}(\uu(a_1,b_1) \ot \uu(a_2,b_2) \ot \dots \ot \uu(a_g,b_g)) = \\
&\uu(a_1,b_1) \ot \uu([a_1,b_1^{-1}] a_2 [a_1,b_1^{-1}]^{-1}, [a_1,b_1^{-1}] b_2 [a_1,b_1^{-1}]^{-1}) \\
&\mspace{200mu}
\ot \dots \ot \uu([a_1,b_1^{-1}] a_g [a_1,b_1^{-1}]^{-1},[a_1,b_1^{-1}] b_g [a_1,b_1^{-1}]^{-1}),
\end{align*}
our assertion follows.
\end{proof}

We prove another lemma that shows how~$\fG_g$ and~$\fU$ behave with respect to commutators:
\begin{Lemma} \label{GUComm}
\begin{enumerate}
\item
Suppose that
$\fG_g(\uu(a_1,b_1) \ot \dots \ot \uu(a_g,b_g)) =
\uu(a'_1,b'_1) \ot \dots \ot \uu(a'_g,b'_g)$. Then we have
\[[a_1,b_1^{-1}][a_2,b_2^{-1}] \cdots [a_g,b_g^{-1}]
= [a'_g,b_g'^{-1}]  \cdots [a'_2,b_2'^{-1}][a'_1,b_1'^{-1}]. \]

\item
Suppose that
$\fU(\uu(a_1,b_1) \ot \dots \ot \uu(a_g,b_g)) =
\uu(a'_1,b'_1) \ot \dots \ot \uu(a'_g,b'_g)$. If
$h \deq [a_1,b_1^{-1}][a_2,b_2^{-1}] \cdots [a_g,b_g^{-1}]$, we have
$h^{-1} = [a'_1,b_1'^{-1}][a'_2,b_2'^{-1}] \cdots [a'_g,b_g'^{-1}]$.
\end{enumerate}
\end{Lemma}
\begin{proof}
\begin{pflist}
\item
The first assertion is correct for $g=1$. For $g=2$, we set $c \deq [a_1,b_1^{-1}]$ and then have
\[\fG_2(\uu(a_1,b_1) \ot \uu(a_2,b_2)) =
\uu(a_1,b_1) \ot \uu(c a_2 c^{-1}, c b_2 c^{-1}),\]
which implies
\begin{align*}
[a'_2,b_2'^{-1}] [a'_1,b_1'^{-1}] &= [c a_2 c^{-1},c b_2^{-1} c^{-1}] [a_1,b_1^{-1}] \\
&= c [a_2, b_2^{-1}] c^{-1} [a_1,b_1^{-1}] = c [a_2, b_2^{-1}]
= [a_1,b_1^{-1}] [a_2, b_2^{-1}].
\end{align*}
Arguing by induction, we use the notation
\[\fG_{g-1}(\uu(a_2,b_2) \ot \dots \ot \uu(a_g,b_g)) =
\uu(a''_2,b''_2) \ot \dots \ot \uu(a''_g,b''_g).\]
Again with the notation $c \deq [a_1,b_1^{-1}]$, we then have
\begin{align*}
&\fG_g(\uu(a_1,b_1) \ot \uu(a_2,b_2) \ot \dots \ot \uu(a_g,b_g)) = \\
&\uu(a_1,b_1) \ot \uu(ca''_2c^{-1},cb''_2c^{-1}) \ot \dots \ot \uu(ca''_gc^{-1},cb''_gc^{-1})
\end{align*}
and therefore by the induction hypothesis
\begin{align*}
[a'_g,b_g'^{-1}&] \cdots [a'_2,b_2'^{-1}][a'_1,b_1'^{-1}] =
[ca''_gc^{-1}, cb_g''^{-1}c^{-1}] \cdots [ca''_2c^{-1}, cb_2''^{-1}c^{-1}] [a_1,b_1^{-1}]
\\
&= (c [a''_g, b_g''^{-1}] \cdots [a''_2, b_2''^{-1}] c^{-1}) [a_1,b_1^{-1}]
= (c [a_2,b_2^{-1}] \cdots [a_g,b_g^{-1}] c^{-1}) [a_1,b_1^{-1}] \\
&= c [a_2,b_2^{-1}] \cdots [a_g,b_g^{-1}]
= [a_1,b_1^{-1}] [a_2,b_2^{-1}] \cdots [a_g,b_g^{-1}].
\end{align*}
This proves the first assertion.

\item
If $g=1$, we have $\fU = \fB$, so the second assertion follows from Proposition~\ref{TSABComm} in this case. For $g \ge 2$, we write
\[\fG_g(\uu(a_1,b_1) \ot \dots \ot \uu(a_g,b_g)) =
\uu(a''_1,b''_1) \ot \dots \ot \uu(a''_g,b''_g)\]
and then have $(a'_i,b'_i) = \fB(a''_i,b''_i)$ for $i=1,\dots,g$ by definition, which implies $[a'_i,b_i'^{-1}] = [a''_i,b_i''^{-1}]^{-1}$ by Proposition~\ref{TSABComm}. The first assertion now implies that
\begin{align*}
&[a'_1,b_1'^{-1}] [a'_2,b_2'^{-1}]\cdots [a'_g,b_g'^{-1}] =
[a''_1,b_1''^{-1}]^{-1} [a''_2,b_2''^{-1}]^{-1} \cdots [a''_g,b_g''^{-1}]^{-1} \\
&= ([a''_g,b_g''^{-1}] \cdots [a''_2,b_2''^{-1}] [a''_1,b_1''^{-1}])^{-1}
= ([a_1,b_1^{-1}][a_2,b_2^{-1}] \cdots [a_g,b_g^{-1}])^{-1} = h^{-1}
\end{align*}
as asserted.
\qedhere
\end{pflist}
\end{proof}

The second assertion of the preceding lemma can be restated as follows:
\begin{Cor} \label{CorDrinfU}
If $w \in W_h$, then $\fU(w) \in W_{h^{-1}}$. Consequently, we have
\[\fU(u_D^{-1}.w) = u_D.\fU(w).\]
\end{Cor}
\begin{proof}
By Lemma~\ref{LemBasis} in Paragraph~\ref{SpecMod}, the elements
$\uu(a_1,b_1) \ot \dots \ot \uu(a_g,b_g)$ with
$[a_1,b_1^{-1}][a_2,b_2^{-1}] \cdots [a_g,b_g^{-1}] = h^{-1}$ form a basis of~$W_h$. The second assertion of the preceding lemma shows that~$\fU$ maps these elements to~$W_{h^{-1}}$, proving our first claim. For the second claim, we get from Lemma~\ref{ActDrinf} in Paragraph~\ref{RepDrinf} that
\[\fU(u_D^{-1}.w) = \fU(h.w) = h.\fU(w) = u_D.\fU(w), \]
where we have used that~$\fU$ is $H$-linear. This in turn follows from the fact that~$\fB^{\ot g}$ is $H$-linear by Corollary~\ref{ABLin} in Paragraph~\ref{EndomAB} and that~$\fG_g$ is $H$-linear as a consequence of Lemma~\ref{LinC} in Paragraph~\ref{NatTrC}.
\end{proof}

We record the following two commutation relations for the map $\fG_g$:
\begin{Proposition} \label{PropCommABG}
\begin{enumerate}
\item
$\fA^{\ot g} \circ \fG_g = \fG_g^{-1} \circ \fA^{\ot g}$

\item
$\fB^{\ot g} \circ \fG_g = \fG_g^{-1} \circ \fB^{\ot g}$
\end{enumerate}
\end{Proposition}
\begin{proof}
For $g=1$, these relations are obvious. We now assume that $g \ge 2$ and then have for $j=0,\dots,g-2$ that
\begin{align*}
&\fA^{\ot g} \circ (\id_{L^{\ot j}} \ot \fC_{L^{\ot (g-1-j)}}^{-1}) =
(\fA^{\ot j} \ot \fA \ot \fA^{\ot (g-1-j)})
\circ (\id_{L^{\ot j}} \ot \fC_{L^{\ot (g-1-j)}}^{-1}) = \\
&(\id_{L^{\ot j}} \ot \fA \ot \id_{L^{\ot (g-1-j)}})
\circ (\id_{L^{\ot j}} \ot \fC_{L^{\ot (g-1-j)}}^{-1}) \circ
(\fA^{\ot j} \ot \id_{L} \ot \fA^{\ot (g-1-j)}) = \\
&(\id_{L^{\ot j}} \ot \fC_{L^{\ot (g-1-j)}})
\circ (\fA^{\ot j} \ot \fA \ot \fA^{\ot (g-1-j)}) =
(\id_{L^{\ot j}} \ot \fC_{L^{\ot (g-1-j)}}) \circ \fA^{\ot g},
\end{align*}
where the third equality follows from Proposition~\ref{PropertC} and the second follows from the naturality of~$\fC$ as in the remark directly following that proposition. We therefore have
\begin{align*}
\fA^{\ot g} \circ \fG_g &=
\fA^{\ot g} \circ \fC_{L^{\ot (g-1)}}^{-1} \circ (\id_L \ot \fC_{L^{\ot (g-2)}}^{-1})
\circ \cdots \circ (\id_{L^{\ot (g-2)}} \ot \fC_{L}^{-1}) \\
&= \fC_{L^{\ot (g-1)}} \circ (\id_L \ot \fC_{L^{\ot (g-2)}})
\circ \cdots \circ (\id_{L^{\ot (g-2)}} \ot \fC_{L}) \circ \fA^{\ot g}.
\end{align*}
Since the terms at the beginning all commute by Corollary~\ref{CommC}, we can reverse their order and rewrite this expression as
\begin{align*}
\fA^{\ot g} \circ \fG_g &= (\id_{L^{\ot (g-2)}} \ot \fC_{L}) \circ \cdots \circ (\id_L \ot \fC_{L^{\ot (g-2)}})
\circ \fC_{L^{\ot (g-1)}} \circ \fA^{\ot g} = \fG_g^{-1} \circ \fA^{\ot g}.
\end{align*}
This proves the first assertion about~$\fA$; the proof of the second assertion about~$\fB$ is identical.
\end{proof}

From this proposition, we see that $\fU$ is its own inverse:
\begin{Cor}
$\fU$ is an involution.
\end{Cor}
\begin{proof}
Since~$\fB$ is an involution by Proposition~\ref{Conjug}, the preceding proposition implies that
\[\fU \circ \fU = (\fB^{\ot g} \circ \fG_g) \circ (\fG_g^{-1} \circ \fB^{\ot g})
= \fB^{\ot g} \circ \fB^{\ot g} = \id_{L^{\ot g}}\]
as asserted.
\end{proof}

\subsection{The standard R-matrix} \label{StandRMat}
The map~$\fU$ introduced in Definition~\ref{DefGU} yields the required isomorphism:
\begin{Theorem} \label{RepIsom}
$\fU$ is an isomorphism from the $\Gamma_{g,1}$-action on~$L^{\ot g}$ coming from the standard R-matrix to the one coming from the alternative R-matrix~$\tilde{R}$.
\end{Theorem}
\begin{proof}
\begin{pflist}
\item
For the case $g=1$, we have discussed in Paragraph~1.12 that~$\Gamma_{1,1}$ is generated by~$[\ft_1]$ and~$[\fs_1]$, and from the discussion in Paragraph~2.6, we know that, in the representation coming from the standard R-matrix, these generators act by~$\fT$ and~$\fS$, respectively. For the representation coming from the alternative R-matrix~$\tilde{R}$, we get from Proposition~\ref{PropAltRMat} and Lemma~\ref{LemAltRMat} that these generators act by~$\fT^{-1}$ and~$\fS^{-1}$ instead. Since~$\fU = \fB$ if~$g=1$, our assertion follows from Proposition~\ref{Conjug} in this case.

\item
Suppose now that $g \ge 2$. We then know from the Dehn-Lickorish theorem~1.4 and the discussion preceding it that $\Gamma_{g,1}$ is generated by $[\ft_i]$, $[\fs_i]$, and $[\fn_l]$, for $i=1,\dots,g$ and \mbox{$l=1,\dots,g-1$}. Again from the discussion in Paragraph~2.6, we know that, in the representation coming from the standard R-matrix, these generators act by
\[\fT_i \deq \id_{L^{\ot (i-1)}} \ot \fT \ot \id_{L^{\ot (g-i)}}
\qquad \text{and} \qquad
\fS_i \deq \id_{L^{\ot (i-1)}} \ot \fS \ot \id_{L^{\ot (g-i)}}\]
for $i=1,\ldots,g$ and
$\fN_l \deq \id_{L^{\ot (l-1)}} \ot \fN \ot \id_{L^{\ot (g-l-1)}}$
for $l=1,\ldots,g-1$. For the representation coming from the alternative R-matrix~$\tilde{R}$, we get, as in the previous~case, from Proposition~\ref{PropAltRMat} and Lemma~\ref{LemAltRMat} that these generators act by~$\fT_i^{-1}$, $\fS_i^{-1}$, and~$\fN_l^{-1}$, respectively.
We therefore have to show that
$\fU \circ \fT_i = \fT_i^{-1} \circ \fU$ and
$\fU \circ \fS_i = \fS_i^{-1} \circ \fU$ for $i=1,\ldots,g$ as well as
$\fU \circ \fN_l = \fN_l^{-1} \circ \fU$
for $l=1,\ldots,g-1$.

\item
For $j=0,\dots,g-2$, we claim that~$\fT_i$ and~$\fS_i$ commute with
$\id_{L^{\ot j}} \ot \fC_{L^{\ot (g-1-j)}}^{-1}$. This is obvious if $i \le j$. If $i=j+1$, this follows from Proposition~\ref{PropertC}, and if $i>j+1$, this follows from the naturality of~$\fC$ as in the comment directly after Proposition~\ref{PropertC}. In total, we see that~$\fT_i$ and~$\fS_i$ commute with~$\fG_g$. The equations $\fU \circ \fT_i = \fT_i^{-1} \circ \fU$ and
$\fU \circ \fS_i = \fS_i^{-1} \circ \fU$ therefore follow from Proposition~\ref{Conjug}.

\item
The third equation $\fU \circ \fN_l = \fN_l^{-1} \circ \fU$ expands to
\begin{align*}
&\fB^{\ot g} \circ \fC_{L^{\ot (g-1)}}^{-1} \circ (\id_L \ot \fC_{L^{\ot (g-2)}}^{-1})
\circ (\id_{L^{\ot 2}} \ot \fC_{L^{\ot (g-3)}}^{-1})
\circ \cdots \circ (\id_{L^{\ot (g-2)}} \ot \fC_{L}^{-1}) \circ \fN_l= \\
&\fN_l^{-1} \circ \fB^{\ot g} \circ \fC_{L^{\ot (g-1)}}^{-1} \circ (\id_L \ot \fC_{L^{\ot (g-2)}}^{-1})
\circ (\id_{L^{\ot 2}} \ot \fC_{L^{\ot (g-3)}}^{-1})
\circ \cdots \circ (\id_{L^{\ot (g-2)}} \ot \fC_{L}^{-1}).
\end{align*}
Here we can cancel the last terms
$\id_{L^{\ot j}} \ot \fC_{L^{\ot (g-1-j)}}^{-1}$
for $j=l+1,\dots,g-2$, because they commute with~$\fN_l$. We are left with the condition
\begin{align*}
&\fB^{\ot g} \circ \fC_{L^{\ot (g-1)}}^{-1} \circ (\id_L \ot \fC_{L^{\ot (g-2)}}^{-1})
\circ (\id_{L^{\ot 2}} \ot \fC_{L^{\ot (g-3)}}^{-1})
\circ \cdots \circ (\id_{L^{\ot l}} \ot \fC_{L^{\ot (g-1-l)}}^{-1}) \circ \fN_l= \\
&\fN_l^{-1} \circ \fB^{\ot g} \circ \fC_{L^{\ot (g-1)}}^{-1} \circ (\id_L \ot \fC_{L^{\ot (g-2)}}^{-1})
\circ (\id_{L^{\ot 2}} \ot \fC_{L^{\ot (g-3)}}^{-1})
\circ \cdots \circ (\id_{L^{\ot l}} \ot \fC_{L^{\ot (g-1-l)}}^{-1}).
\end{align*}
Furthermore, also the first terms
$\id_{L^{\ot j}} \ot \fC_{L^{\ot (g-1-j)}}^{-1}$
for $j=0,\dots,l-2$ commute with~$\fN_l$ by the naturality of~$\fC$, so that the expression reduces to
\begin{align*}
&\fB^{\ot g} \circ (\id_{L^{\ot (l-1)}} \ot \fC_{L^{\ot (g-l)}}^{-1}) \circ (\id_{L^{\ot l}} \ot \fC_{L^{\ot (g-1-l)}}^{-1}) \circ \fN_l= \\
&\fN_l^{-1} \circ \fB^{\ot g} \circ (\id_{L^{\ot (l-1)}} \ot \fC_{L^{\ot (g-l)}}^{-1}) \circ (\id_{L^{\ot l}} \ot \fC_{L^{\ot (g-1-l)}}^{-1}),
\end{align*}
where we have also used that the terms $\id_{L^{\ot j}} \ot \fC_{L^{\ot (g-1-j)}}^{-1}$ commute with each other, a fact already used in the proof of Proposition~\ref{PropCommABG}. Inserting the definition of~$\fN_l$, this equation reads
\begin{align*}
&\fB^{\ot g} \circ (\id_{L^{\ot (l-1)}} \ot \fC_{L^{\ot (g-l)}}^{-1}) \circ (\id_{L^{\ot l}} \ot \fC_{L^{\ot (g-1-l)}}^{-1}) \circ (\id_{L^{\ot (l-1)}} \ot \fN \ot \id_{L^{\ot (g-l-1)}}) = \\
&(\id_{L^{\ot (l-1)}} \ot \fN^{-1} \ot \id_{L^{\ot (g-l-1)}}) \circ \fB^{\ot g} \circ (\id_{L^{\ot (l-1)}} \ot \fC_{L^{\ot (g-l)}}^{-1})
\circ (\id_{L^{\ot l}} \ot \fC_{L^{\ot (g-1-l)}}^{-1}),
\end{align*}
which follows from
\begin{align*}
&\fB^{\ot (g-l+1)} \circ \fC_{L^{\ot (g-l)}}^{-1} \circ (\id_{L} \ot \fC_{L^{\ot (g-1-l)}}^{-1}) \circ (\fN \ot \id_{L^{\ot (g-l-1)}}) = \\
&(\fN^{-1} \ot \id_{L^{\ot (g-l-1)}}) \circ \fB^{\ot (g-l+1)}
\circ \fC_{L^{\ot (g-l)}}^{-1} \circ (\id_{L} \ot \fC_{L^{\ot (g-1-l)}}^{-1}),
\end{align*}
which is in turn equivalent to
\begin{align*}
&(\fB \ot \fB \ot \id_{L^{\ot (g-l-1)}}) \circ \fC_{L^{\ot (g-l)}}^{-1} \circ (\id_{L} \ot \fC_{L^{\ot (g-1-l)}}^{-1}) \circ (\fN \ot \id_{L^{\ot (g-l-1)}}) = \\
&(\fN^{-1} \ot \id_{L^{\ot (g-l-1)}}) \circ (\fB \ot \fB \ot \id_{L^{\ot (g-l-1)}})
\circ \fC_{L^{\ot (g-l)}}^{-1} \circ (\id_{L} \ot \fC_{L^{\ot (g-1-l)}}^{-1}).
\end{align*}
But this is exactly Proposition~\ref{PropertN} for the $H$-module~$V=L^{\ot (g-l-1)}$.
\qedhere
\end{pflist}
\end{proof}

We note that it is a consequence of the Dehn-Lickorish theorem~1.4 that the mapping class~$[\fd_1]$ of the Dehn twist~$\fd_1$ is a word in the generators considered in the preceding proof. As discussed at the end of Paragraph~2.6, $\fd_1$ acts by the ribbon twist~$\theta_W$, which in our case is given by the action of the ribbon element, as we explained at the end of Paragraph~2.2. As explained in Paragraph~\ref{DrinfDoub}, the inverse Drinfel'd element~$u_D^{-1}$ is a ribbon element for the standard R-matrix~$R$, while~$u_D$ itself is a ribbon element for the alternative R-matrix~$\tilde{R}$, so that it follows from Theorem~\ref{RepIsom} above that $\fU(u_D^{-1}.w) = u_D.\fU(w)$ for all $w \in W = L^{\ot g}$. This gives another proof of this fact, which we have established in Corollary~\ref{CorDrinfU} directly.

After these preparations, we can give the following analogue of Theorem~\ref{ThmModTilde} for the standard R-matrix:
\begin{Cor} \label{CorMod}
The map
$\Theta \colon M^g \to W_e = (L^{\ot g})_e$ that maps a basis element
\mbox{$f \in \Hom(\pi_1(\Sigma_{g},x),G)$} of~$M^g$ to
\[\Theta(f) \deq \fU(\uu(a_1,b_1^{-1}) \ot \dots \ot \uu(a_g,b_g^{-1}))\]
where $a_i \deq f([\alpha_i])$ and $b_i \deq f([\beta_i])$ for $i=1,\ldots,g$, is an isomorphism of the \mbox{$\Gamma_{g}(x)$-modules} $M^g$ and~$W_e$.
\end{Cor}
\begin{proof}
By Lemma~\ref{ActDrinf}, the inverse Drinfel'd element~$u_D^{-1}$ acts trivially on~$W_e$, so that the $\Gamma_{g,1}$-representation on~$W_e$ coming from the standard R-matrix descends to an action of~$\Gamma_{g}(x)$ for the same reason as the one coming from the alternative \mbox{R-matrix} does, which was discussed in Paragraph~\ref{SpecMod}. This follows of course also from Theorem~\ref{RepIsom}.
It then follows from Corollary~\ref{CorDrinfU} that~$\fU$ maps~$W_e$ to itself and therefore induces an isomorphism between these two $\Gamma_{g}(x)$-representations by Theorem~\ref{RepIsom}. Composing it with the isomorphism~$\tilde{\Theta}$ from Theorem~\ref{ThmModTilde}, we get the assertion.
\end{proof}

It may be noted that the discussion in this paragraph is not interesting in the case~\mbox{$g=0$}, because the group $\Gamma_{0,1}$ is trivial by the Alexander lemma, as already mentioned after the statement of the Dehn-Lickorish theorem~1.4.

\section{Derived block spaces and group cohomology} \label{Sec:DerBlCoh}
\subsection{Yoneda products} \label{YonProd}
We return for a moment to a more general situation and assume that~$\CC$ is a \mbox{$K$-linear} abelian category, where~$K$ is, as usual, our base field. As we stated at the beginning of Paragraph~3.1, there are different approaches to the definition of the Ext-functors. The approach primarily used in~\cite{Mi} is originally due to~N.~Yoneda (cf.~\cite{Y}). For two objects~$X$ and~$Y$ of~$\CC$ and~$m \ge 1$, we denote, following the notation in~\cite[Chap.~IV, Sec.~9, p.~148ff]{HiS}, by~$\Yext^m(X,Y)$ the equivalence classes of $m$-extensions, i.e., exact sequences~$E$ of the form
\[E \colon 0 \to Y \to Z_{m-1} \to \dots \to Z_1 \to Z_0 \to X \to 0\]
under an equivalence relation specified in~\cite[Chap.~VII, \S~3, p.~172]{Mi}.
$\Yext^m(X,Y)$ is a (big) abelian group under the Baer sum (cf.~\cite[Chap.~VII, Thm.~3.3, p.~174]{Mi}). In fact, $\Yext^m(X,Y)$ is a bifunctor that is covariant in~$Y$, contravariant in~$X$, and additive in both variables (cf.~\cite[Chap.~VII, \S~3, p.~175]{Mi}). Furthermore, there is an operation
\[\Yext^m(X,Y) \times \Yext^l(W,X) \to \Yext^{m+l}(W,Y)\]
that is called the Yoneda product and is defined as follows: If
\[F \colon 0 \to X \to Z'_{l-1} \to \dots \to Z'_1 \to Z'_0 \to W \to 0\]
represents an element in~$\Yext^l(W,X)$, we can splice it together with the extension above to obtain the sequence
\[EF \colon 0 \to Y \to Z_{m-1} \to \dots \to Z_1 \to Z_0 \to Z'_{l-1} \to \dots \to Z'_1 \to Z'_0 \to W \to 0\]
that represents the required element in~$\Yext^{m+l}(W,Y)$ (cf.~\cite[p.~171]{Mi}).

In our situation,~$\Yext^m(X,Y)$ is not only a (big) abelian group, but rather a (big) vector space over our base field~$K$. In order to introduce this vector space structure, we need the following lemma, for which we unfortunately do not know a reference:
\begin{Lemma} \label{YonedaNat}
Suppose that $\nu$ is a natural transformation from the
identity functor of~$\CC$ to itself. Then we have
\[\Yext^m(\nu_X,Y) = \Yext^m(X,\nu_Y)\]
for all objects~$X$ and~$Y$ of~$\CC$.
\end{Lemma}
\begin{proof}
In the notation of~\cite[Chap.~VII, \S~1, p.~162f]{Mi}, our assertion is that, for an extension~$E$ as above, the extensions~$\nu_Y E$ and~$ E \nu_X$ are equivalent. We first consider the case~$m=1$. Suppose that
\[E \colon 0 \xlongrightarrow{} Y \xlongrightarrow{f} Z \xlongrightarrow{g} X
\xlongrightarrow{} 0\]
is a representative of an equivalence class in~$\Yext^1(X,Y)$. By definition, its image under~$\Yext^1(\nu_X,Y)$ is represented by the top line in the diagram
\begin{center}
\begin{tikzcd}[column sep = large] {}
0 \arrow{r}{} &
Y \arrow{d}{\id_Y} \arrow{r}{f'} &
V \arrow{d}{h'} \arrow{r}{g'} &
X \arrow{d}{\nu_X} \arrow{r}{} &
0 &\\
0 \arrow{r}{} &
Y \arrow{r}{f} &
Z \arrow{r}{g} &
X \arrow{r}{} &
0
\end{tikzcd}
\end{center}
where the right square is a pullback diagram and~$f'$ is determined by the universal property of the pullback with respect to the morphism~$f$ and the zero morphism. On the other hand, its image under~$\Yext^1(X,\nu_Y)$ is represented by the bottom line in the diagram
\begin{center}
\begin{tikzcd}[column sep = large] {}
0 \arrow{r}{} &
Y \arrow{d}{\nu_Y} \arrow{r}{f} &
Z \arrow{d}{h''} \arrow{r}{g} &
X \arrow{d}{\id_X} \arrow{r}{} &
0 &\\
0 \arrow{r}{} &
Y \arrow{r}{f''} &
W \arrow{r}{g''} &
X \arrow{r}{} &
0
\end{tikzcd}
\end{center}
where the left square is a pushout diagram and~$g''$ is determined by the universal property of the pushout with respect to the morphism~$g$ and the zero morphism.

Because we have $g \circ \nu_Z = \nu_X \circ g$ by naturality, the universal property of the pullback implies that there is a morphism~\mbox{$h \colon Z \to V$} that is uniquely determined by the property that
$h' \circ h = \nu_Z$ and $g' \circ h = g$. By the uniqueness property of the pullback, we then get $f' \circ \nu_Y = h \circ f$, as we have
\[h' \circ f' \circ \nu_Y = f \circ \nu_Y = \nu_Z \circ f = h' \circ h \circ f\]
and
\[g' \circ f' \circ \nu_Y = 0 = g \circ f = g' \circ h \circ f.\]
The universal property of the pushout therefore yields the existence of a morphism~\mbox{$k \colon W \to V$} with the property that $k \circ f'' = f'$ and $k \circ h'' = h$. This morphism makes the diagram
\begin{center}
\begin{tikzcd}[column sep = large] {}
0 \arrow{r}{} &
Y \arrow{d}{\id_Y} \arrow{r}{f''} &
W \arrow{d}{k} \arrow{r}{g''} &
X \arrow{d}{\id_X} \arrow{r}{} &
0 &\\
0 \arrow{r}{} &
Y \arrow{r}{f'} &
V \arrow{r}{g'} &
X \arrow{r}{} &
0
\end{tikzcd}
\end{center}
commutative. This is obvious for the left square and follows for the right square from the universal property of the pushout, as we have
\[g' \circ k \circ f'' = g' \circ f' = 0 = g'' \circ f''\]
and
\[g' \circ k \circ h'' = g' \circ h = g = g'' \circ h''\]
as required. This proves our assertion in the case~$m=1$.

For the case~$m>1$, we know from~\cite[p.~171]{Mi} that an extension~$E$ can be written as the Yoneda product $E = E_m E_{m-1} \dots E_2 E_1$ of~$m$ short exact sequences. If the first of these is the short exact sequence
\[E_m \colon 0 \xlongrightarrow{} Y \xlongrightarrow{f} Z_{m-1} \xlongrightarrow{g} X_{m-1}
\xlongrightarrow{} 0\]
it follows from the case~$m=1$ that the extension $\nu_Y E = (\nu_Y E_m) E_{m-1} \dots E_2 E_1$ is equivalent to the extension $(E_m \nu_{X_{m-1}}) E_{m-1} \dots E_2 E_1$. Now~\cite[Chap.~VII, \S~3, Eq.~(3), p.~171]{Mi} implies that this extension is equivalent to
$E_m (\nu_{X_{m-1}} E_{m-1}) \dots E_2 E_1$. Applying the case~$m=1$ again and continuing in this way, we finally see that the extensions~$\nu_Y E$ and~$E \nu_X$ are equivalent.
\end{proof}

If $\lambda \in K$ is a scalar, the morphisms $\nu_X \deq \lambda \id_X$ constitute a natural transformation from the identity functor to itself, and by applying it to extensions as indicated above,~$\Yext^m(X,Y)$ becomes a vector space over~$K$ (cf.~\cite[Sec.~5.1, p.~68]{He}). It follows easily from the preceding lemma and \cite[Chap.~VII, Lem.~3.2, p.~173]{Mi} that the Yoneda product is bilinear with respect to this vector space structure.

\vspace{1mm}

Therefore, if we define~$\Yext^0(X,Y) \deq \Hom_\CC(X,Y)$, introduce the direct sum
\[\Yext(X,Y) \deq \bigoplus_{m=0}^\infty \Yext^m(X,Y),\]
and also extend the definition of the Yoneda product to degree zero by applying morphisms to extensions by functoriality as indicated above, then~$\Yext(X,X)$
becomes an associative algebra that is graded over the nonnegative integers~$\N_0$, and $\Yext(X,Y)$ becomes a graded bimodule over~$\Yext(Y,Y)$ on the left and~$\Yext(X,X)$ on the right (cf.~\cite[Chap.~VII, \S~3, p.~175]{Mi}). As a consequence, the canonical projection from~$\Yext(X,X)$ to~$\Yext^0(X,X) = \Hom_\CC(X,X)$ is an algebra homomorphism, which we call the augmentation of~$\Yext(X,X)$. In particular, if~$\CC$ satisfies the assumptions stated in Paragraph~2.1, we can take for~$X$ the unit object~$\1$ of the category and get an algebra homomorphism from~$\Yext(\1,\1)$ to~$\Hom_\CC(\1,\1) \cong K$, which we can use to view any vector space over~$K$ as a module over~$\Yext(\1,\1)$ by pullback along this augmentation. We will refer to $\Yext(\1,\1)$-modules that arise in this way as trivial modules.

\vspace{1mm}

We have also assumed in Paragraph~2.1 that~$\CC$ has enough projectives, and it is shown in~\cite[Chap.~VII, \S~7, p.~182ff]{Mi} that the groups~$\Yext^m(X,Y)$ and the groups~$\Ext^m(X,Y)$ discussed in Paragraph~3.1 are then naturally isomorphic. For a chosen projective resolution
\[\cdots \xlongrightarrow{d_3} P_2 \xlongrightarrow{d_2} P_1 \xlongrightarrow{d_1} P_0
\xlongrightarrow{\xi} X\]
of~$X$, this isomorphism associates with an element of~$\Yext^m(X,Y)$ that is represented by the extension
\vspace{0.5mm}
\[E \colon 0 \to Y \to Z_{m-1} \to \dots \to Z_1 \to Z_0 \to X \to 0\]
the cohomology class of the morphism~$\phi_m \in \Hom_{\CC}(P_m,Y)$ that arises by completing the extension to a positive chain complex by adding zeros in the higher degrees and\hfill
\pagebreak

then applying the comparison theorem already mentioned in Paragraph~3.1 to obtain the diagram
\begin{center}
\begin{tikzcd}[column sep = small] {}
\cdots \arrow{r}{} &
P_{m+2} \arrow{d}{\phi_{m+2}} \arrow{r}{d_{m+2}} &
P_{m+1} \arrow{d}{\phi_{m+1}} \arrow{r}{d_{m+1}} &
P_m \arrow{d}{\phi_m} \arrow{r}{d_m} &
P_{m-1} \arrow{d}{\phi_{m-1}} \arrow{r}{} &
\cdots \arrow{r}{} &
P_1 \arrow{d}{\phi_1} \arrow{r}{d_1} &
P_0 \arrow{d}{\phi_0} \arrow{r}{\xi} &
X \arrow{d}{\id_X} \arrow{r}{} &
0 &\\
\cdots \arrow{r}{} &
0 \arrow{r}{} &
0 \arrow{r}{} &
Y \arrow{r}{} &
Z_{m-1} \arrow{r}{} &
\cdots \arrow{r}{} &
Z_1 \arrow{r}{} &
Z_0 \arrow{r}{} &
X \arrow{r}{} &
0
\end{tikzcd}
\end{center}
(cf.~also~\cite[Sec.~2.6, p.~39]{Be}).
Note that the diagram shows that
\[d_{m+1}^*(\phi_m) = \phi_m \circ d_{m+1} = 0,\]
so that~$\phi_m$ is a cocycle. With the help of this isomorphism, the Yoneda product can be carried over to the groups~$\Ext^m(X,Y)$. As a consequence, the direct sum
\[\Ext(X,Y) \deq \bigoplus_{m=0}^\infty \Ext^m(X,Y)\]
is a graded bimodule over the graded algebra~$\Ext(Y,Y)$ on the left and the graded algebra~$\Ext(X,X)$ on the right with respect to the induced Yoneda product.

We note that the preceding lemma has an analogue for~$\Ext^m(X,Y)$ that is substantially easier to prove:
\begin{Lemma} \label{ExtNat}
Suppose that $\nu$ is a natural transformation from the
identity functor of~$\CC$ to itself. Then we have
\[\Ext^m(\nu_X,Y) = \Ext^m(X,\nu_Y)\]
for all objects~$X$ and~$Y$ of~$\CC$.
\end{Lemma}
\begin{proof}
For a projective resolution of~$X$ as above, naturality implies that~$\nu_X$ lifts to a chain map, so that the diagram
\begin{center}
\begin{tikzcd}[column sep = large] {}
\cdots \arrow{r}{} &
P_2 \arrow{d}{\nu_{P_2}} \arrow{r}{d_2} &
P_1 \arrow{d}{\nu_{P_1}} \arrow{r}{d_1} &
P_0 \arrow{d}{\nu_{P_0}} \arrow{r}{\xi} &
X \arrow{d}{\nu_X} \arrow{r}{} &
0 &\\
\cdots \arrow{r}{} &
P_2 \arrow{r}{d_2} &
P_1 \arrow{r}{d_1} &
P_0 \arrow{r}{\xi} &
X \arrow{r}{} &
0
\end{tikzcd}
\end{center}
commutes. Therefore the morphism $\Ext^m(\nu_X,Y)$ is induced by precomposition with~$\nu_{P_m}$ on the space~$\Hom_\CC(P_m,Y)$, while $\Ext^m(X,\nu_Y)$ is induced by postcomposition with~$\nu_{Y}$ on this space. By naturality, these two morphisms agree.
\end{proof}

The Yoneda product is closely related to the so-called cup product; in fact, the two products agree in special cases. Although we will not need the cup product in the sequel, we discuss the relation of the two products in Paragraph~\ref{CupProd} of the appendix, because the majority of the related literature uses the language of cup products instead of Yoneda products.

\subsection{Reduction to group cohomology} \label{RedGroupCohom}
We now return to the case of the Drinfel'd double of a finite group that we introduced in Paragraph~\ref{DrinfDoub}. In Paragraph~\ref{RepDrinf}, we have discussed the $K[G]$-module~$W_e$ that is associated with an arbitrary $D(G)$-module~$W$. The arising cohomology groups of the Drinfel'd double can be computed via the group cohomology of~$W_e$, a fact recorded for the case~$W=K$ in~\cite[Lem.~1.3, p.~311]{W1} (cf.~also~\cite[Chap.~IV, Sec.~12, Eq.~(12.1), p.~163]{HiS}):
\begin{Proposition} \label{GroupCohom}
$\Ext_D^m(K, W) \cong H^m(G, W_e)$
\end{Proposition}
\begin{proof}
We choose a projective resolution of the trivial $K[G]$-module~$K$:
\[K \longleftarrow P_1 \longleftarrow P_2 \longleftarrow \cdots \]
From Corollary~\ref{TrivConjClass}, we know that we can turn~$P_m$ into a projective module over~$D$ via the action $(\delta_a \ot b).v \deq \delta_{a,e} \, b.v$, so that we obtain a projective resolution of the trivial module~$K$ over~$D$. In view of the functor adjunction discussed after that corollary, we have
\[\Hom_{D}(P_m, W) \cong \Hom_G(P_m, W_e).\]
But the cohomology of the latter cochain complex is just group cohomology of~$G$ with coefficients in the $G$-module $W_e$. It should be noted in this context that the cohomology groups~$\Ext_{K[G]}^m(K, W_e)$ and~$\Ext_{\Z[G]}^m(\Z, W_e)$ are isomorphic (cf.~\cite[Sec.~1.1, p.~3]{E}).
\end{proof}

As we have discussed in Paragraph~\ref{YonProd}, $\Ext_D^m(K, W)$ and $\Yext_D^m(K, W)$ are also isomorphic. It is immediate from the description there that, if
\[0 \to W \to Z_{m-1} \to \dots \to Z_1 \to Z_0 \to K \to 0\]
represents an element in $\Yext_D^m(K, W)$, the corresponding element in
$\Yext_G^m(K, W_e)$ is
\[0 \to W_e \to (Z_{m-1})_e \to \dots \to (Z_1)_e \to (Z_0)_e \to K \to 0,\]
where we have used that $K_e=K$. In view of this fact, the isomorphisms described in the preceding proposition combine to an algebra isomorphism between the graded algebra~$\Ext_D(K, K)$ and the cohomology ring
\[H(G, K) \deq \bigoplus_{m=0}^\infty H^m(G, K)\]
endowed with the Yoneda product, which is, as will be explained in Paragraph~\ref{CupProd} of the appendix, equal to the cup product in this case. Using this fact, the right $\Ext_D(K, K)$-module $\Ext_D(K, W)$ then becomes isomorphic to the right $H(G, K)$-module $H(G, W_e) \deq \bigoplus_{m=0}^\infty H^m(G, W_e)$.

In particular, this discussion applies to the module $W \deq L^{\ot g}$ considered in Paragraph~\ref{SpecMod}. As recalled there, this module carries an action of the mapping class group~$\Gamma_{g,1}$ that preserves the homogeneous components, and therefore especially~$W_e$. Since this action is by~$D$-linear maps, it is also by $K[G]$-linear maps. We recall in addition from Paragraph~\ref{RelRepVar} and Paragraph~\ref{StandRMat} that this $K[G]$-module is isomorphic to the linearized representation variety~$M^g$ via~$\tilde{\Theta}$ and~$\Theta$, respectively.

In Paragraph~3.2, we have shown that the action of~$\Gamma_{g,1}$ descends to an action of~\mbox{$\Gamma_{g} = \Gamma_{g,0}$} in cohomology, i.e., on the derived block spaces
$\TFT^m(\Sigma_{g}) = \Ext_D^m(K, L^{\ot g})$ introduced there. As a consequence, this must also hold for the corresponding cohomology groups of~$G$:
\begin{Cor} \label{Group}
For all $m \ge 0$, we have $\TFT^m(\Sigma_{g}) \cong H^m(G, M^g)$ as $\Gamma_{g}$-modules.
\end{Cor}

The derived block spaces in the various degrees can be combined together into a graded vector space by forming the direct sum
\[\TFT^\infty(\Sigma_g) \deq \bigoplus_{m=0}^\infty \TFT^m(\Sigma_g) \]
that is by the preceding corollary isomorphic to~$H(G, M^g)$. For $g=0$, we have $L^{\ot g} \cong K$ and therefore $\TFT^\infty(\Sigma_0) \cong H(G, K)$. The considerations in Paragraph~\ref{YonProd} therefore imply that~$\TFT^\infty(\Sigma_g)$ is a right module over $\TFT^\infty(\Sigma_0)$. The action of the mapping class group is linear with respect to this module structure:

\begin{Proposition} \label{MapClassYon}
$\Gamma_{g}$ acts on $\TFT^\infty(\Sigma_g)$ by~$\TFT^\infty(\Sigma_0)$-linear maps.
\end{Proposition}
\begin{proof}
Suppose that
\[E \colon 0 \to W \to Z_{m-1} \to \dots \to Z_1 \to Z_0 \to K \to 0\]
represents an element in~$\Yext_D^m(K,W)$ and that
\[F \colon 0 \to K \to Z'_{l-1} \to \dots \to Z'_1 \to Z'_0 \to K \to 0\]
represents an element in~$\Yext_D^l(K,K)$, where~$m \ge 1$ and~$l \ge 1$. The action of an element of~$\Gamma_{g}$ is induced by a $D$-linear map~$f \colon W \to W$. By writing~$E$ as a product of short exact sequences as in the proof of Lemma~\ref{YonedaNat}, we get, as already used there, that $f(EF) = (fE)F$, which is the assertion in this case. If~$m=0$, we
have~$\Yext_D^m(K,W) = \Hom_D(K,W)$, and if~$e \in \Hom_D(K,W)$ and~$l \ge 1$, we have by~\cite[Chap.~VII, \S~3, Eq.~(2), p.~171]{Mi} that~$f(eF) = (fe)F$, which is the assertion in this case. The case~$m \ge 1$ and~$l = 0$ follows from the very same reference, and the case~$m=l=0$ follows from the associativity of composition.

We note that a different proof of this proposition can be given by viewing the Yoneda product as a special case of the cup product and describing the latter through projective resolutions. In this picture, the proof of the assertion is even easier, as~$f$ acts only on the coefficients~$W = L^{\ot g}$. This viewpoint is explained in greater detail in Paragraph~\ref{CupProd} and Paragraph~\ref{CupProdRes} of the appendix.
\end{proof}

As we will see in Paragraph~\ref{Abel}, the $\TFT^\infty(\Sigma_0)$-module structure on~$\TFT^\infty(\Sigma_g)$ is particularly easy in the case of abelian groups.

Let us mention a connection of Corollary~\ref{Group} to geometry. As shown in
\cite[Chap.~VI, Prop.~3.1, p.~191]{HiS}, the group $H^0(G, M^g)$ is isomorphic to the group of invariant elements of~$M^g$. Now an element $m = \sum \zeta_f f \in M^g$, where the sum is over all group homomorphisms~$f \in \Omega^g = \Hom(\pi_1(\Sigma_{g},x),G)$ and only finitely many coefficients $\zeta_f \in K$ are nonzero, is invariant if and only if $\zeta_{a.f} = \zeta_f$ for all $a \in G$. Therefore, this space is spanned by the characteristic functions on the orbits of the \mbox{$G$-action} on the representation variety~$\Omega^g$ with respect to the action described at the end of Paragraph~\ref{RelRepVar}. The space of these orbits is known as the character variety of~$\pi_1(\Sigma_{g},x)$ in~$G$ (cf.~\cite[Sec.~11, p.~5191]{Si}). Since~$G$ is finite, every such orbit in the character variety corresponds to an isomorphism class of principal $G$-bundles over~$\Sigma_{g}$ (cf.~\cite[Exerc.~14.3, p.~196]{F}; \cite[\S~13.9, p.~66]{St}). In this way, we can interpret the ordinary block space~$\TFT^0(\Sigma_{g}) = \TFT(\Sigma_{g})$ as the linear span of the isomorphism classes of principal $G$-bundles over~$\Sigma_{g}$. A similar point in the case~$g=1$ is made in~\cite[Prop.~3.3, p.~3705]{SW}. This connection suggests that the derived block spaces for the Drinfel'd double of a finite group might play a role in a potential derived version of Dijkgraaf-Witten theory over fields of finite characteristic~(cf.~\cite{DW},~\cite{FQ}).

\section{Decomposition of the representation} \label{Sec:DecRep}
\subsection{The general case} \label{GenCase}
In Paragraph~\ref{RelRepVar}, we have introduced the representation variety $\Omega^g$ and the linearized representation variety~$M^g$ with~$\Omega^g = \Hom(\pi_1(\Sigma_g,x),G)$ as a basis, as we just recalled. For a subgroup~$U \le G$ of~$G$, we now refine this definition and introduce the set
\[\Omega^{g}_U \deq \{f \in \Omega^g \mid \im(f) = U \}\]
as well as the subspace~$M^g_U \deq K[\Omega^{g}_U]$ of~$M^g$. It is immediate from this definition that~$\Omega^g$ is the disjoint union of its subsets~$\Omega^{g}_U$ for the various subgroups~$U$ of~$G$, and consequently
\[M^g =\bigoplus_{U \le G} M^g_U\]
is the direct sum of its subspaces~$M^g_U$. The set~$\Omega^{g}_U$, and consequently the vector space~$M^g_U$, are invariant under the action of~$\Gamma_{g}(x)$ described in Paragraph~\ref{RelRepVar}, but not under the action of~$G$ also described there, unless~$U$ is normal. To remedy that, we make still another definition: For a conjugacy class~$O$ of subgroups, i.e., the orbit
\[O = \{a U a^{-1} \mid a \in G \}\]
of a given subgroup~$U$ of~$G$ under the conjugation action of~$G$, we define
\[\Omega^{g}(O) \deq \bigcup_{V \in O} \Omega^{g}_V
\qquad \text{and} \qquad
M^g(O) \deq \bigoplus_{V \in O} M^g_V.\]
If $O_1,\dots,O_r$ are the distinct conjugacy classes of subgroups of~$G$, we have
\[\Omega^{g} = \Omega^{g}(O_1) \cup \dots \cup \Omega^{g}(O_r)\]
and consequently
$M^{g} = M^{g}(O_1) \oplus \dots \oplus M^{g}(O_r)$.

We can describe these sets in a way that separates the action of~$G$ and the action of~$\Gamma_{g}(x)$ more clearly. For this, suppose that $O = \{a U a^{-1} \mid a \in G \}$ is the orbit of the subgroup~$U$ of~$G$ under conjugation. The centralizer
\[C \deq \{a \in G \mid ab = ba \text{ for all } b \in U\}\]
of~$U$ is a normal subgroup of the normalizer
\[N \deq \{a \in G \mid aba^{-1} \in U \text{ for all } b \in U\}\]
of~$U$, and we denote the arising quotient group by~$Z \deq N/C$. Note that~$C$ is the kernel of the natural homomorphism~$N \to \Aut(U)$ that is given by conjugation, and therefore this homomorphism induces a group homomorphism~$Z \to \Aut(U)$. If $f \in \Hom(\pi_1(\Sigma_g,x),G)$ has image~$U$, we can compose~$f$ with conjugation by an element from~$N$ and get in this way an action of~$Z$ on~$\Omega_U^{g}$. On the other hand, for a left coset $aC \in G/C$ of~$C$ and an element~$b \in N$, we have that
\[aCb = (ab) (b^{-1} C b) = abC\]
is again a left coset, so that we get a right action of~$N$ on~$G/C$. Because~$C$ obviously acts trivially on~$G/C$, this action also descends to an action of~$Z$.

The Cartesian product $G/C \times \Omega_U^{g}$ carries an action of~$G$ via the standard action of~$G$ on the coset space~$G/C$ and an action of~$\Gamma_{g}(x)$ from its action
on~$\Omega_U^{g}$. It also carries a left action of~$Z$ coming from the right action of~$Z$ on~$G/C$ and its left action on~$\Omega_U^{g}$ via
\[z.(aC, f) \deq (aC.z^{-1}, z.f)\]
for $z \in Z$, $aC \in G/C$, and~$f \in \Omega_U^{g}$. All three actions commute. In analogy with the construction of associated bundles (cf.~\cite[Sec.~I-5, p.~54f]{KN}), we denote the set of orbits of the $Z$-action by~$G/C \times_Z \Omega_U^{g}$. Because the actions commute, this quotient space inherits both a $G$-action and a~$\Gamma_{g}(x)$-action from the Cartesian product. This space is isomorphic to~$\Omega^g(O)$:
\begin{Proposition} \label{PropDecomp}
The map
\[\varphi \colon G/C \times \Omega_U^{g} \to \Omega^g(O),~(aC,f) \mapsto
([\gamma] \mapsto af([\gamma])a^{-1})\]
descends to a bijection between~$G/C \times_Z \Omega_U^{g}$ and~$\Omega^g(O)$ that is equivariant with respect to both the $G$-action and the~$\Gamma_{g}(x)$-action.
\end{Proposition}
\begin{proof}
It follows from the construction that~$\varphi$ is surjective and that the definition of~$\varphi$ does not depend on the representative~$a$ of the coset~$aC$. The $G$-equivariance of~$\varphi$ follows from the computation
\begin{align*}
\varphi(a.(bC),f)([\gamma]) &= \varphi(abC,f)([\gamma]) = ab f([\gamma]) b^{-1}a^{-1} \\
&= a (\varphi(bC,f)([\gamma])) a^{-1} = (a.\varphi(bC,f))([\gamma])
\end{align*}
for $[\gamma] \in \pi_1(\Sigma_g,x)$. If $[\psi] \in \Gamma_{g}(x)$, we have by the definition from Paragraph~\ref{RelRepVar} that
\begin{align*}
\varphi(aC, [\psi].f)([\gamma]) &= a f([\psi^{-1} \circ \gamma]) a^{-1}
= \varphi(aC,f)([\psi^{-1} \circ \gamma])
= ([\psi].\varphi(aC,f))([\gamma]),
\end{align*}
which shows the $\Gamma_{g}(x)$-equivariance. If~$b \in N$ and~$z \deq bC \in Z = N/C$, we have seen above that $z.(aC, f) = (ab^{-1}C, z.f)$ and therefore have
\begin{align*}
\varphi(z.(aC, f))([\gamma]) &= ab^{-1} (z.f)([\gamma]) (ab^{-1})^{-1}
= ab^{-1} (b f([\gamma]) b^{-1}) (ab^{-1})^{-1} \\
&= a f([\gamma]) a^{-1} = \varphi(aC, f)([\gamma]),
\end{align*}
which shows that~$\varphi$ induces a map
$\bar{\varphi}\colon G/C \times_Z \Omega_U^{g}\to O^g(O)$ that is still surjective and equivariant, but in addition is now injective: If $\varphi(a C,f) = \varphi(a' C,f')$,
then $a'^{-1} a f([\gamma]) a^{-1} a' = f'([\gamma])$ for all $[\gamma]\in \pi_1(\Sigma_g,x)$. Since
$\im(f) = \im(f') = U$, this implies that $a'^{-1} a \in N$.
If $z \deq a'^{-1} a C \in Z = N/C$, this equation also shows that~$f' = z.f$ and therefore
$(a' C,f') = z.(a C,f)$.
\end{proof}

Because~$\Omega^g(O)$ is a basis of the space~$M^g(O)$, the bijection found in the proposition above yields a vector space isomorphism between $K[G/C \times_Z \Omega_U^{g}]$ and $M^g(O)$. Now the space
\mbox{$K[G/C \times_Z \Omega_U^{g}]$} is isomorphic to $K[G/C] \ot_{K[Z]} M_U^{g}$, where the right \mbox{$K[Z]$-module} structure on~$K[G/C]$ and the left $K[Z]$-module structure on~$M_U^{g}$ are the linear extensions of the respective~$Z$-actions on~$G/C$ and on~$\Omega_U^{g}$. This can be deduced from the commutative diagram
\begin{center}
\begin{tikzcd}[column sep = large] {}
K[G/C \times \Omega_U^{g}] \arrow{d}{} \arrow{r}{} &
K[G/C] \ot_K M_U^{g} \arrow{d}{}  \\
K[G/C \times_Z \Omega_U^{g}] \arrow{r}{} &
K[G/C] \ot_{K[Z]} M_U^{g}
\end{tikzcd}
\end{center}
by observing that $K[G/C \times_Z \Omega_U^{g}]$ satisfies the necessary universal property of the tensor product over~$K[Z]$ with regard to balancing. The preceding proposition therefore also yields a vector space isomorphism
\[K[G/C] \ot_{K[Z]} M_U^{g} \to M^g(O)\]
that is still equivariant with respect to both the $G$-action and the~$\Gamma_{g}(x)$-action. This isomorphism descends to cohomology:
\begin{Theorem} \label{ThmDecomp}
$H^m(G,M^g(O)) \cong H^m(G,K[G/C]) \ot_{K[Z]} M_U^{g}$
\end{Theorem}
\begin{proof}
\begin{parlist}
\item
Suppose that $f_1,\dots,f_s \in \Omega_U^{g}$ is a system of representatives for the orbits of the $Z$-action on~$\Omega_U^{g}$ described above. If $a \in N$ satisfies~$(aC).f_i = f_i$, we have~$a f_i([\gamma]) a^{-1} = f_i([\gamma])$ for all $[\gamma] \in \pi_1(\Sigma_g,x)$, and  therefore~$a b a^{-1} = b$ for all~$b \in U$, so that~$a \in C$. This shows that~$\Omega_U^{g}$
is the disjoint union of~$s$ orbits of~$Z$, each of which is isomorphic to~$Z$. Consequently,
$M_U^{g}$ is a free $K[Z]$-module with basis~$f_1,\dots,f_s$.

\item
For a projective resolution
\[K \longleftarrow P_1 \longleftarrow P_2 \longleftarrow \cdots \]
of the trivial $G$-module~$K$, we then get from~\cite[Chap.~II, \S~4.2, Prop.~2, p.~269]{Bo} an isomorphism of cochain complexes
\[\Hom_{K[G]}(P_m, K[G/C]) \ot_{K[Z]} M_U^{g} \cong
\Hom_{K[G]}(P_m, K[G/C] \ot_{K[Z]} M_U^{g})\]
that combines with the isomorphism above to an isomorphism
\[\Hom_{K[G]}(P_m, K[G/C]) \ot_{K[Z]} M_U^{g} \cong
\Hom_{K[G]}(P_m, M^g(O)).\]
Because $M_U^{g}$ is free of finite rank, this isomorphism induces an isomorphism
\[\Ext^m_{K[G]}(K, K[G/C]) \ot_{K[Z]} M_U^{g} \cong
\Ext^m_{K[G]}(K, M^g(O))\]
as asserted. This last isomorphism can be viewed as a special instance of the universal coefficient theorem (cf.~\cite[Sec.~3.4, p.~30]{E}; \cite[Thm.~7.55, p.~448]{Ro}), but is in fact more directly a consequence of the additivity of homology and cohomology (cf.~\cite[Prop.~6.8, p.~330]{Ro}): As we have shown in the first step that $M_U^{g} \cong K[Z]^s$, we get
\begin{align*}
\Hom_{K[G]}(P_m, K[G/C]) \ot_{K[Z]} M_U^{g} \cong \Hom_{K[G]}(P_m, K[G/C])^s
\end{align*}
and passing to cohomology in the cochain complex on the left means doing that component-wise in the cochain complex on the right, so that the result on the right~is
\begin{align*}
\Ext^m_{K[G]}(K, K[G/C])^s \cong \Ext^m_{K[G]}(K, K[G/C]) \ot_{K[Z]} M_U^{g}
\end{align*}
as required.
\qedhere
\end{parlist}
\end{proof}

With the help of the Eckmann-Shapiro lemma (cf.~\cite[Prop.~4.1.3, p.~36]{E}; \cite[Lem.~6.3.2, p.~171]{We}), we can simplify the formula in the preceding theorem further:
\begin{Cor} \label{CorDecomp}
$H^m(G,M^g(O)) \cong H^m(C,K) \ot_{K[Z]} M_U^{g}$
\end{Cor}
\begin{proof}
The $K[G]$-module~$K[G/C]$ appearing in the theorem can be viewed as induced from the trivial $K[C]$-module structure, so that $K[G/C] \cong K[G] \ot_{K[C]} K$. In the case of finite groups, induced modules are isomorphic to coinduced modules (cf.~\cite[Prop.~4.1.2, p.~36]{E}; \cite[Lem.~6.3.4, p.~172]{We}), which yields in our situation that
\[K[G/C] \cong \Hom_{K[C]}(K[G],K).\]
Another instance of the hom-tensor adjunction already mentioned in Paragraph~\ref{RepDrinf} yields
\[\Hom_{K[G]}(P_m, \Hom_{K[C]}(K[G],K)) \cong \Hom_{K[C]}(P_m, K),\]
where~$P_m$ denotes a module from the projective resolution that appears in the proof of the theorem above. Since these modules are also projective over~$K[C]$ (cf.~\cite[p.~35]{E}; \cite[p.~172]{We}), we obtain that
\[H^m(G,K[G/C]) \cong H^m(C,K).\]
The right $K[Z]$-module structure on $H^m(G,K[G/C])$ then induces a corresponding structure on
$H^m(C,K)$.
\end{proof}

In summary, the preceding results yield the following approach to the computation
of the~$\Gamma_g$-module~$H^m(G,M^g)$: If $O_1,\dots,O_r$ are the distinct conjugacy classes of subgroups of~$G$, we choose a representative~$U_i \in O_i$ for all $i=1,\dots,r$. If~$N_i$ is its normalizer and~$C_i$ is its centralizer, we form the factor group~$Z_i \deq N_i/C_i$ and obtain that
\[H^m(G,M^g) \cong \bigoplus_{i=1}^r H^m(C_i,K) \ot_{K[Z_i]} M_{U_i}^{g}. \]
Here, the mapping class group acts only on the tensor factors~$M_{U_i}^{g}$.

\subsection{The abelian case} \label{Abel}
In the case where the subgroup~$U$ of~$G$ is contained in the center of~$G$, its orbit~$O$ defined in the previous paragraph consists only of~$U$ itself, so that we have
\mbox{$\Omega^g(O) = \Omega^g_U$} and~$M^g(O) = M^g_U$. For both the normalizer~$N$ and the centralizer~$C$ of~$U$ introduced there, we then have~$N=C=G$, so that the
quotient group~$Z \deq N/C$ consists only of one element and we have~$K[Z] \cong K$. The formulas in both Theorem~\ref{ThmDecomp} and Corollary~\ref{CorDecomp} therefore reduce to
\[H^m(G,M_U^{g}) \cong H^m(G,K) \ot_{K} M_U^{g}.\]
If~$G$ is an abelian group, obviously every subgroup is central, and by summing over the various subgroups~$U$ and additionally reversing the tensor factors, we obtain
$H^m(G,M^{g}) \cong M^{g} \ot_{K} H^m(G,K)$. In view of Corollary~\ref{Group}, this equation can be written in the form
\[\TFT^\infty(\Sigma_g) \cong \TFT^0(\Sigma_g) \ot_K \TFT^\infty(\Sigma_0),\]
where we have used that the $G$-action on~$M^{g}$ is trivial in the abelian case. By construction, this isomorphism is $\Gamma_g$-equivariant if
$\TFT^0(\Sigma_g) \ot_K \TFT^\infty(\Sigma_0)$ is endowed with the $\Gamma_g$-action coming from the left tensor factor. However, this isomorphism is also compatible with the Yoneda product:
\begin{Proposition} \label{AbelFreeGen}
Suppose that~$G$ is a finite abelian group. Then we have
\[\TFT^\infty(\Sigma_g) \cong \TFT^0(\Sigma_g) \ot_K \TFT^\infty(\Sigma_0)\]
as graded right modules over~$\TFT^\infty(\Sigma_0)$.
\end{Proposition}
\begin{proof}
In the proof of Theorem~\ref{ThmDecomp}, we have seen that there is an isomorphism
\[\Hom_{K[G]}(P_m, K[G/C]) \ot_{K[Z]} M_U^{g} \cong \Hom_{K[G]}(P_m, M^g(O))\]
that in our situation upon reversing the tensor factors on the left yields
\[M_U^{g} \ot_{K} \Hom_{K[G]}(P_m, K)  \cong \Hom_{K[G]}(P_m, M_U^{g})\]
and, after summing over the various subgroups~$U$, induces the required isomorphism in cohomology. But if we view~$M_U^{g}$ as the space~$\Hom_K(K,M_U^{g})$ by evaluating a linear map from~$K$ to~$M_U^{g}$ at the unit element of~$K$, this isomorphism becomes the isomorphism
\[\Hom_{K[G]}(K,M_U^{g}) \ot_K \Hom_{K[G]}(P_m,K) \cong \Hom_{K[G]}(P_m,M_U^{g})\]
given by composition. Here, we have used again the fact that the module structure on~$M_U^{g}$ is trivial, so that $\Hom_{K[G]}(K,M_U^{g}) = \Hom_{K}(K,M_U^{g})$. Consequently the isomorphism
\[\Hom_{K[G]}(K,M_U^{g}) \ot_K \Ext^m_{K[G]}(K,K) \to \Ext^m_{K[G]}(K,M_U^{g})\]
induced in cohomology just arises from the fact that~$\Ext^m$ is a covariant functor in the second argument. As discussed in Paragraph~\ref{YonProd}, this isomorphism becomes the Yoneda product
\[\Yext^0_{K[G]}(K,M_U^{g}) \ot_K \Yext^m_{K[G]}(K,K) \to \Yext^m_{K[G]}(K,M_U^{g})\]
in the Yoneda picture of the functor, as we also noted in the proof of Proposition~\ref{MapClassYon}. In other words, this isomorphism is just the restriction of the module structure map to $\Yext^0_{K[G]}(K,M_U^{g}) \ot_K \Yext^m_{K[G]}(K,K)$, and the module structure map is clearly $\TFT^\infty(\Sigma_0)$-linear.
\end{proof}

The previous proposition shows that, in the case of an abelian group, $\TFT^\infty(\Sigma_g)$ is a free $\TFT^\infty(\Sigma_0)$-module that is generated by~$\TFT^0(\Sigma_g)$ as a~$\TFT^\infty(\Sigma_0)$-module. As we will see in Paragraph~\ref{YonChar3}, $\TFT^\infty(\Sigma_g)$ is not always generated in degree zero in the nonabelian case.

If~$G$ is abelian, a group homomorphism $f \colon \pi_1(\Sigma_g,x) \to G$ clearly factors over the commutator factor group. As we have discussed in Paragraph~1.11, the commutator factor group of the fundamental group~$\pi_1(\Sigma_g,x)$ is isomorphic to the first singular homology group~$H_1(\Sigma_g, \Z)$. We have also discussed there that this group is free abelian of rank~$2g$ with the homology classes of the curves $\alpha_1, \beta_1, \alpha_2, \beta_2,  \ldots, \alpha_g, \beta_g$ as a basis and that the algebraic intersection number induces a symplectic form on~$H_1(\Sigma_g, \Z)$, whose fundamental matrix with respect to this basis is
\[J \deq \begin{pmatrix}
0 & -1 &  &  & \\
1 & 0 &  &  & \\
& & \ddots & &\\
& & & 0 & -1\\
& & & 1 & 0\\
\end{pmatrix}.\]
The action of an element~$[\psi] \in \Gamma_g$ is therefore in this basis represented by a matrix~$A \in \Sp(2g, \Z)$, i.e., satisfies $A^T J A = J$.

As a consequence, the map
\[\check{\Theta} \colon \Omega^g \to G^{2g},~f \mapsto v \deq (a_1,b_1,\dots,a_g,b_g)\]
where $a_i \deq f([\alpha_i])$ and $b_i \deq f([\beta_i])$, which is in general only injective, is bijective in the abelian case. This map is similar, but not equal, to the mappings~$\tilde{\Theta}$ and~$\Theta$ discussed in Paragraph~\ref{RelRepVar} and Paragraph~\ref{StandRMat}. The definition of the action of an element~$[\psi] \in \Gamma_g(x)$ on the representation variety~$\Omega^g$ in Paragraph~\ref{RelRepVar} then yields, after a short computation, the equivariance property
\[\check{\Theta}([\psi].f) = (A^T)^{-1} \check{\Theta}(f)\]
for $f \in \Omega^g$. Here, matrices in~$\Sp(2g, \Z)$ act on~$G^{2g}$ as matrices act on column vectors; it is sometimes more convenient to write~$G$ additively instead of multiplicatively when working with this action. Note that $(A^T)^{-1} = J A J^{-1}$ by the defining property of a symplectic matrix given above.
The bijection~$\check{\Theta}$ between the representation variety~$\Omega^g$ and~$G^{2g}$ clearly yields a vector space isomorphism between the linearized representation variety~$M^g$ and~$K[G^{2g}]$, the free vector spaces that have the respective sets as bases.

\subsection{The cyclic case} \label{Cycl}
An important special case of the preceding considerations is the case where $G=\Z_n$ is a cyclic group of order~$n$, written additively. Obviously, the action of~$\Sp(2g, \Z)$ on column vectors with entries in~$\Z_n$ factors over~$\Sp(2g, \Z_n)$. In this context, it is worth pointing out that the canonical map from~$\Sp(2g, \Z)$ to~$\Sp(2g, \Z_n)$ is surjective (cf.~\cite[Thm.~1, p.~85]{NS}; \cite[Thm.~VII.21, p.~132]{N}). The following proposition, whose proof can be found in~\cite[Thm.~3.5, p.~18]{FF}, describes the orbits of the action of~$\Sp(2g, \Z_n)$
on~$\Z_n^{2g}$ via greatest common divisors:
\begin{Proposition} \label{SymplOrb}
Suppose that $v = (\bar{k}_1, \bar{l}_1,\ldots, \bar{k}_g, \bar{l}_g)$
and $w = (\bar{k}'_1, \bar{l}'_1,\ldots, \bar{k}'_g, \bar{l}'_g)$ are two vectors in~$\Z_n^{2g}$. Then the following assertions are equivalent:
\begin{enumerate}
\item
There exists $A \in \Sp(2g, \Z_n)$ so that $w = Av$.

\item
$\gcd(k_1, l_1,\ldots, k_g, l_g,n) = \gcd(k'_1, l'_1,\ldots, k'_g, l'_g,n)$
\end{enumerate}
\end{Proposition}

In other words, if we define  the set
\[\Omega^{g,n,m} \deq
\{(\bar{k}_1, \bar{l}_1,\ldots, \bar{k}_g, \bar{l}_g) \in \Z_n^{2g}
\mid m = \gcd(k_1, l_1,\ldots, k_g, l_g,n) \}\]
for every divisor~$m$ of~$n$, then the sets $\Omega^{g,n,m}$ are exactly the different
$\Sp(2g, \Z_n)$-orbits, so that we have the partition
\[\Z_n^{2g} = \bigcup_{m|n} \Omega^{g,n,m}\]
into disjoint sets. Here, the notation has been chosen to be parallel to the
notation~$\Omega_U^{g}$ introduced in Paragraph~\ref{GenCase}: If~$U$ is the subgroup of~$\Z_n$ that is generated by~$\bar{m}$ and has order~$n/m$, then~$\Omega^{g,n,m}$ is the image of~$\Omega_U^{g}$ under the bijection~$\check{\Theta}$ between~$\Omega^{g}$ and~$\Z_n^{2g}$ discussed in Paragraph~\ref{Abel}. The decomposition of a basis into disjoint sets leads to a direct sum decomposition of the associated vector space, so that we have
\[K[\Z_n^{2g}] = \bigoplus_{m|n} K[\Omega^{g,n,m}].\]
Clearly, the map
\[\Omega^{g,n/m,1} \to \Omega^{g,n,m},~(\bar{k}_1, \bar{l}_1,\ldots, \bar{k}_g, \bar{l}_g)
\mapsto (m\bar{k}_1, m\bar{l}_1,\ldots, m\bar{k}_g, m\bar{l}_g)\]
is bijective, and it is also $\Sp(2g, \Z_n)$-equivariant if the action of~$\Sp(2g, \Z_n)$ on the left-hand side is defined as the pullback of the action of~$\Sp(2g, \Z_{n/m})$ along the canonical map. We can therefore concentrate on the action of~$\Sp(2g, \Z_n)$ on
$\Omega^{g,n} \deq \Omega^{g,n,1}$.

The decomposition above can be refined: The group of units~$\Z_n^\times$ acts on~$\Omega^{g,n}$ by multiplication, and therefore by linear extension on~$K[\Omega^{g,n}]$. We denote the set of orbits of~$\Omega^{g,n}$ under this action by $\P^{2g-1}_n \deq \Omega^{g,n}/\Z_n^\times$. This notation is motivated by the fact that, in the case where~$n$ is prime, $\Z_n$ is a field
and~$\Omega^{g,n}$ is the set of nonzero vectors in~$\Z_n^{2g}$. The
quotient~$\Omega^{g,n}/\Z_n^\times$ is therefore the set of nonzero vectors modulo nonzero scalar multiples, which corresponds to the set of one-dimensional subspaces of~$\Z_n^{2g}$, i.e., the corresponding projective space.

We now assume that~$K$ contains a primitive $\phi(n)$-th root of unity, where
\mbox{$\phi(n) \deq |\Z_n^\times|$} denotes Euler's totient function. Then the character group
$\hat{\Z}^\times_n \deq \Hom(\Z_n^\times, K^\times)$ contains also $\phi(n)$ elements; these elements are called Dirichlet characters. Because in this case the characteristic of the base field cannot divide~$\phi(n)$, Maschke's theorem implies that the group algebra~$K[\Z_n^\times]$ is semisimple, and so its action on~$K[\Omega^{g,n}]$ decomposes into eigenspaces
\[K_\chi[\P^{2g-1}_n] \deq \{v \in K[\Omega^{g,n}]
\mid \bar{k} v = \chi(\bar{k}) v \; \text{for all} \; \bar{k} \in \Z_n^\times\}\]
that are indexed by Dirichlet characters~$\chi \in \hat{\Z}^\times_n$. Note that these eigenspaces are clearly invariant under the $\Sp(2g, \Z_n)$-action. The notation is chosen in order to make a connection with associated line bundles: Since~$\Z_n^\times$ acts freely
on~$\Omega^{g,n}$, it is a principal fiber bundle over the orbit space with respect to the projection map
\[\pi \colon \Omega^{g,n} \to \P^{2g-1}_n.\]
For a one-dimensional representation given by the
Dirichlet character~$\chi \colon \Z_n^\times \to K^\times$, we can therefore form the associated line bundle, a construction that was already mentioned in Paragraph~\ref{GenCase}. This line bundle is the space of orbits of the action of~$\Z_n^\times$ on the
Cartesian product~$K \times \Omega^{g,n}$ under the action
\[\bar{k}.(\lambda,v) \deq (\lambda \chi(\bar{k}^{-1}),\bar{k} v),\]
where $\bar{k}^{-1} \in \Z_n^\times$ denotes the multiplicative inverse of~$\bar{k}$. We denote the individual orbit by~$[\lambda,v]$, the entire orbit space, i.e., the associated line bundle, by~$K \times_{\chi} \Omega^{g,n}$, and the space of its sections by~$\Gamma[K \times_{\chi} \Omega^{g,n}]$. Such a section $\sigma \colon \P^{2g-1}_n \to K \times_{\chi} \Omega^{g,n}$ assigns to the orbit~$\pi(v)$ of~$v \in \Omega^{g,n}$ an element $[\lambda,w]$ with the property that~$w$ lies in the same orbit as~$v$. Because the action of~$\Z_n^\times$ on~$\Omega^{g,n}$ is free, there is in fact a unique representative in~$[\lambda,w]$ whose second component is~$v$. Consequently, there is a function $\tau \colon \Omega^{g,n} \to K$ with the property that
\[\sigma(\pi(v)) = [\tau(v),v].\] For $\bar{k} \in \Z_n^\times$, we have that $\sigma(\pi(v)) = \sigma(\pi(\bar{k} v))$ and therefore
\[[\tau(v),v] = [\tau(\bar{k} v),\bar{k} v] = [\chi(\bar{k}) \tau(\bar{k} v), v],\]
so that $\tau(v) = \chi(\bar{k}) \tau(\bar{k} v)$ or $\tau(\bar{k} v) = \chi(\bar{k}^{-1}) \tau(v)$. Conversely, every function~$\tau$ that satisfies the last equation yields a section of our line bundle via \mbox{$\sigma(\pi(v)) \deq [\tau(v),v]$}. This correspondence allows us to relate the sections of our line bundle to the eigen\-spaces introduced before:
\begin{Lemma}
The map
\[\Gamma[K \times_{\chi} \Omega^{g,n}] \to K_\chi[\P^{2g-1}_n],~\sigma
\mapsto \sum_{v \in \Omega^{g,n}} \tau(v) v\]
is bijective.
\end{Lemma}
\begin{proof}
We have
\[\sum_{v \in \Omega^{g,n}} \tau(v) \bar{k} v =
\sum_{v \in \Omega^{g,n}} \tau(\bar{k}^{-1} v) v = \sum_{v \in \Omega^{g,n}} \chi(\bar{k}) \tau(v) v,\]
so that the image of~$\sigma$ is an eigenvector with eigenvalue $\chi(\bar{k})$, which shows that the map is well-defined. On the other hand, any vector in~$K[\Omega^{g,n}]$ can be expanded in the form $\sum_{v \in \Omega^{g,n}} \tau(v) v$ with coefficients~$\tau(v) \in K$, and the condition that the vector is an eigenvector yields the condition $\tau(\bar{k} v) = \chi(\bar{k}^{-1}) \tau(v)$ for the coefficients.
\end{proof}

We note that in the case where the Dirichlet character~$\chi$ is constantly equal to~$1$, the condition $\tau(\bar{k} v) = \chi(\bar{k}^{-1}) \tau(v)$ states that~$\tau$ is constant on the orbits of the~\mbox{$\Z_n^\times$-action}, so that in this case the $\Sp(2g, \Z_n)$-modules~$K_\chi[\P^{2g-1}_n]$ and~$K[\P^{2g-1}_n]$ are isomorphic, where $K[\P^{2g-1}_n]$ carries the permutation representation that arises from the action
of~$\Sp(2g, \Z_n)$ on~$\Omega^{g,n}$ and its quotient
\mbox{$\P^{2g-1}_n = \Omega^{g,n}/\Z_n^\times$}.

In the general case, we can collect in the sum $\sum_{v \in \Omega^{g,n}} \tau(v) v$ the terms that belong to the same $\Z_n^\times$-orbit, and find
\[\sum_{\bar{k} \in \Z_n^\times} \tau(\bar{k} v) \bar{k} v =
\sum_{\bar{k} \in \Z_n^\times} \chi(\bar{k}^{-1}) \tau(v) \bar{k} v =
\tau(v) \sum_{\bar{k} \in \Z_n^\times} \chi(\bar{k}^{-1}) \bar{k} v.\]
This shows that $K_\chi[\P^{2g-1}_n]$ has a basis consisting of the vectors
$\sum_{\bar{k} \in \Z_n^\times} \chi(\bar{k}^{-1}) \bar{k} v$, where~$v$ runs over a system of representatives for the orbits of~$\Omega^{g,n}$ under the action of~$\Z_n^\times$. Changing this representative changes the corresponding basis vector by a scalar multiple. If we act on this vector with a matrix~$A \in \Sp(2g,\Z_n)$, it is mapped to the analogous vector in the orbit of~$Av$. However, as~$Av$ may not be the representative that we have chosen in that orbit, the image may differ from the corresponding vector by a scalar multiple. We therefore see that the action of~$A$ on~$K_\chi[\P^{2g-1}_n]$ is represented with respect to this basis by a monomial matrix.

This representation is in fact an induced representation. To see this, we suppose that
$v = (1,0,0,\dots,0) \in \Z_n^{2g}$ is the first canonical basis vector, and introduce the subgroup
\[P \deq \{A \in \Sp(2g,\Z_n) \mid
A v = \bar{k} v \; \text{for some} \; \bar{k} \in \Z_n^\times\}\]
of symplectic matrices whose first column is a multiple of~$v$. The assignment
\[P \to \Z_n^\times,~A \mapsto \bar{k}\]
for $A v = \bar{k} v$ that assigns to~$A \in P$ its first entry is clearly a group homomorphism. It is surjective, because the matrix
$A \deq \diag(\bar{k}, \bar{k}^{-1}, 1, 1, \ldots, 1)$ is symplectic for
all~$\bar{k} \in \Z_n^\times$. A given Dirichlet character $\chi \colon \Z_n^\times \to K^\times$ can be composed with the homomorphism just constructed to obtain a group homomorphism from~$P$ to~$K^\times$. In this way, the base field~$K$ becomes a~$P$-module, and we can form the induced module
$K[\Sp(2g,\Z_n)] \ot_{K[P]} K$. It is isomorphic to the eigenspace under consideration:
\begin{Proposition}
If $x \deq \sum_{\bar{k} \in \Z_n^\times} \chi(\bar{k}^{-1}) \bar{k} v$, the map
\[K[\Sp(2g,\Z_n)] \ot_{K[P]} K \to K_\chi[\P^{2g-1}_n],~h \ot \lambda \mapsto \lambda h.x\]
is an isomorphism of $\Sp(2g,\Z_n)$-representations.
\end{Proposition}
\begin{proof}
If~$A \in P$ is a symplectic matrix with $A v = \bar{k} v$, we have by linearity that also
$A x = \bar{k} x$. But since $x \in K_\chi[\P^{2g-1}_n]$, we have $\bar{k} x = \chi(\bar{k}) x$.
From this fact we see that the stated map is well-defined, as it is $K[P]$-balanced. Since the action of~$\Sp(2g,\Z_n)$ on~$\Omega^{g,n}$ is transitive by Proposition~\ref{SymplOrb}, it follows from the discussion preceding this proposition that our map is surjective. Now
$K[\Sp(2g,\Z_n)]$ is a free right \mbox{$K[P]$-module} with a basis consisting of a set of left coset representatives. Therefore we have
\[\dim(K[\Sp(2g,\Z_n)] \ot_{K[P]} K) = [\Sp(2g,\Z_n):P].\]
On the other hand, we have seen above that
$\dim(K_\chi[\P^{2g-1}_n]) = |\P^{2g-1}_n|$. But by definition, $P$ is the stabilizer of the orbit~$\pi(v)$ of~$v$ in~$\P^{2g-1}_n$. So the dimensions of both sides are equal by the orbit-stabilizer theorem, and consequently the map is bijective. It is obvious that it is
$\Sp(2g,\Z_n)$-equivariant.
\end{proof}

\subsection{Cyclic subgroups} \label{CyclSub}
We now return to the situation of Paragraph~\ref{GenCase} and assume that the subgroup~$U$ of~$G$ considered there is cyclic, so that $U = \langle h \rangle$ for a generator~$h$, whose order we denote by~$n$. In comparison with the previous paragraph, a small difficulty arises from the fact that here we use multiplicative instead of additive notation for the group operation. To address this difficulty, we denote the image of~$\Omega^g_U$ and~$\Omega^g(O)$ under~$\check{\Theta}$ by~$\check{\Omega}^g_U$ and~$\check{\Omega}^g(O)$, respectively. We then have the bijection
\[\omega_{g,n} \colon \Omega^{g,n} \to \check{\Omega}^g_U,~(\bar{k}_1, \bar{l}_1,\ldots, \bar{k}_g, \bar{l}_g)
\mapsto (h^{k_1}, h^{l_1}, \ldots, h^{k_g}, h^{l_g})\]
and a related bijection for~$\Omega^g(O)$ will be discussed below. Because $\omega_{g,n}$ is derived from the group isomorphism~$\Z_n \to U,~\bar{k} \mapsto h^k$, it is $\Sp(2g, \Z)$-equivariant by construction.

An element $a \in N$ determines an automorphism of~$U$ that maps~$u \in U$ to $a u a^{-1}$. Because~$\Aut(U)$ is isomorphic to~$\Z_n^\times$ (cf.~\cite[2.2.2, p.~50]{KS}), there exists
$\bar{k} \in \Z_n^\times$ so that~$a u a^{-1} = u^k$ for all~$u \in U$. Clearly, this exponent depends only on the coset~$aC \in Z = N/C$, so that we have a group homomorphism
\[\rho \colon Z \to \Z_n^\times,~z \mapsto \rho(z)\]
with $\bar{k} = \rho(aC)$ if~$a h a^{-1} = h^k$. We can use this group homomorphism to express a certain equivariance property:
\begin{Proposition} \label{RhoEquiv}
For $f \in \Omega_U^g$ and $z \in Z$, we have
\[(\omega_{g,n}^{-1} \circ \check{\Theta})(z.f) =
\rho(z) (\omega_{g,n}^{-1} \circ \check{\Theta})(f).\]
\end{Proposition}
\begin{proof}
For $\bar{k} = \rho(z)$ and~$z = aC$, suppose that $(\omega_{g,n}^{-1} \circ \check{\Theta})(f) = (\bar{k}_1,\bar{l}_1,\dots,\bar{k}_g,\bar{l}_g)$, which means that
\mbox{$h^{k_i} \deq f([\alpha_i])$} and $h^{l_i} \deq f([\beta_i])$. We then have by definition that
\[(z.f)([\alpha_i]) = a f([\alpha_i]) a^{-1} = a h^{k_i} a^{-1} = h^{k k_i}\]
and similarly $(z.f)([\beta_i]) = h^{k l_i}$. This establishes our assertion.
\end{proof}

The group homomorphism~$\rho$ can also be used to introduce a $Z$-action on the conjugacy class
\[I \deq \{a h a^{-1} \mid a \in G\}\]
of~$h$: If $\bar{k} = \rho(bC)$, we define the action of~$z=bC$ as
\[I \to I,~c \mapsto c^k.\]
Note that~$c^k \in I$, since $c = a h a^{-1}$ implies that
$c^k = a h^k a^{-1} = a b h b^{-1} a^{-1}$. This action of~$Z$ on~$I$ can be viewed as both a left and a right action, because~$\Z_n^\times$ is abelian.
\begin{Lemma} \label{Equiv}
The map
\[G/C \to I,~aC \mapsto a h a^{-1}\]
is $Z$-equivariant.
\end{Lemma}
\begin{proof}
For a coset $z = bC \in Z$, suppose that $\rho(z) = \bar{k}$. As we have seen in Paragraph~\ref{GenCase}, we then have
$(aC).z = abC$, and this coset is mapped to
\[a b h b^{-1} a^{-1} = a h^k a^{-1} = (a h a^{-1})^k = (a h a^{-1}).z\]
as required.
\end{proof}

The maps appearing above can be arranged into the commutative diagram
\begin{center}
\begin{tikzcd}[column sep = large] {}
N/C \arrow{d}{} \arrow{r}{} &
G/C \arrow{d}{} \arrow{r}{} &
G/N \arrow{d}{}  \\
U \cap I \arrow{r}{} &
I \arrow[swap]{r}{b \, \mapsto \, \langle b \rangle} &
O
\end{tikzcd}
\end{center}
of sets, in which the first two vertical maps are given by $aC \mapsto a h a^{-1}$ and the third vertical map is given by $aN \mapsto a h a^{-1}$. The rows of this diagram are exact in the sense that the composition in both rows is constant, and the image of the first map is the preimage of the respective distinguished points, which is~$N$ for the first row and~$U$ for the second row.

This lemma and the preceding proposition can be used to turn the bijection
between~$\Omega^g(O)$ and~$G/C \times_Z \Omega^g_U$ established in Proposition~\ref{PropDecomp} into a bijection
between~$\check{\Omega}^g(O)$ and~$I \times_Z \Omega^{g,n}$ in such a way that the diagram
\begin{center}
\begin{tikzcd}[column sep = large] {}
G/C \times_Z \Omega^g_{U} \arrow{d}{} \arrow{r}{} &
\Omega^g(O) \arrow{d}{\check{\Theta}}  \\
I \times_Z \Omega^{g,n} \arrow{r}{} &
\check{\Omega}^g(O)
\end{tikzcd}
\end{center}
commutes. Here, the map from Proposition~\ref{PropDecomp} appears on top, while the left vertical map is induced by the map
\[G/C \times \Omega^g_{U} \to I \times \Omega^{g,n},~(a, f) \mapsto
(a h a^{-1}, (\omega_{g,n}^{-1} \circ \check{\Theta})(f)),\]
which descends to a bijection between~$G/C \times_Z \Omega^g_{U}$
and~$I \times_Z \Omega^{g,n}$ in view of the equivariance properties established in Proposition~\ref{RhoEquiv} and Lemma~\ref{Equiv}.
The lower horizontal map is induced by
\[I \times \Omega^{g,n} \to \check{\Omega}^g(O),~(c,(\bar{k}_1, \bar{l}_1,\ldots, \bar{k}_g, \bar{l}_g))
\mapsto (c^{k_1}, c^{l_1}, \ldots, c^{k_g}, c^{l_g}).\]
By linear extension of the arising bijection between~$\Omega^g(O)$ and~$I \times_Z \Omega^{g,n}$, we obtain an isomorphism
\[M^g(O) \cong K[I] \otimes_{K[Z]} K[\Omega^{g,n}].\]
This isomorphism can be used to derive the following variant of Theorem~\ref{ThmDecomp} for cyclic subgroups:
\begin{Cor} \label{DecompCohCyl}
$H^m(G,M^g(O)) \cong H^m(G,K[I]) \ot_{K[Z]} K[\Omega^{g,n}]$
\end{Cor}

As we have discussed in Paragraph~\ref{Cycl}, we can decompose~$K[\Omega^{g,n}]$ into the eigen\-spaces~$K_\chi[\P^{2g-1}_n]$ if we assume that~$K$ contains a primitive $\phi(n)$-th root of unity. We want to decompose~$K[I]$ in a similar way. By construction, the action of~$Z$ on~$I$ is the pullback of the action of the subgroup~$\rho(Z)$ of~$\Z_n^\times$ on~$K[I]$, so that a decomposition into the eigenspaces for the action of this subgroup is also a decomposition into eigenspaces for the action of~$Z$. The group algebra~$K[\rho(Z)]$ of this subgroup is a subalgebra of
$K[\Z_n^\times]$, which is semisimple and splits completely, as already mentioned in Paragraph~\ref{Cycl}. The same therefore also holds for this subalgebra, and any group homomorphism from~$\rho(Z)$ to~$K^\times$ is the restriction of a Dirichlet character in~$\hat{\Z}_n^\times$. As a consequence, $K[I]$ decomposes completely into eigenspaces of~$Z$, which we denote, in analogy with Paragraph~\ref{Cycl}, by
\begin{align*}
K_\eta[I/Z] \deq
\{x \in K[I]\mid x.z = \eta(z) x \; \text{for all} \;  z \in Z\}
\end{align*}
for $\eta \in \hat{Z} \deq \Hom(Z,K^\times)$. By combining this decomposition with the one from Paragraph~\ref{Cycl}, we obtain
\begin{align*}
K[I] \ot_{K[Z]} K[\Omega^{g,n}]
= \bigoplus_{\eta \in \hat{Z}} \bigoplus_{\chi \in \hat{\Z}_n^\times}
K_\eta[I/Z] \ot_{K[Z]} K_\chi[\P_n^{2g-1}].
\end{align*}
Now we have for $x \in K_\eta[I/Z]$ and~$v \in K_\chi[\P_n^{2g-1}]$ that
\[\eta(z) \; x \ot_{K[Z]} v = (x.z) \ot_{K[Z]} v = x \ot_{K[Z]} \rho(z) v =
\chi(\rho(z)) \; x \ot_{K[Z]} v,\]
so that $x \ot_{K[Z]} v = 0$ unless $\eta = \chi \circ \rho$. On the other hand, if this condition holds, we have for the tensor product over the base field that
\[(x.z) \ot_K v = \eta(z) \; x \ot_K v =
\chi(\rho(z)) \; x \ot_K v  = x \ot_K \rho(z) v, \]
so that the tensor products over~$K[Z]$ and over~$K$ are isomorphic. Therefore the above decomposition reduces to
\begin{align*}
K[I] \ot_{K[Z]} K[\Omega^{g,n}]
= \bigoplus_{\chi \in \hat{\Z}_n^\times}
K_{\chi \circ \rho}[I/Z] \ot K_\chi[\P_n^{2g-1}].
\end{align*}
In the case where~$K$ contains a primitive $\phi(n)$-th root of unity, we therefore get the following decomposition of our cohomology group:
\begin{Cor} \label{DecompCohEig}
$\displaystyle
H^m(G,M^g(O)) \cong
\bigoplus_{\chi \in \hat{\Z}_n^\times}
H^m(G,K_{\chi \circ \rho}[I/Z]) \ot K_\chi[\P_n^{2g-1}]$
\end{Cor}
\begin{proof}
This follows directly from the preceding discussion. That taking cohomology commutes with taking the tensor product follows from a simple version of the universal coefficient theorem, as in the proof of Theorem~\ref{ThmDecomp}. Alternatively, this can be deduced from Corollary~\ref{DecompCohCyl}, as we have
\[H^m(G,K[I]) \cong
\bigoplus_{\eta \in \hat{Z}} H^m(G,K_{\eta}[I/Z])\]
because the cohomology functor is additive. If we take the tensor product with
$K[\Omega^{g,n}] = \bigoplus_{\chi \in \hat{\Z}_n^\times} K_\chi[\P_n^{2g-1}]$ over~$K[Z]$ on both sides, the same discussion as above shows that many terms in the double sum vanish, and the remaining terms form, as before, the sum that appears on the right-hand side of our assertion.
\end{proof}

\subsection{The symmetric group on three letters} \label{SymmThree}
We now apply the decomposition method explained in Paragraph~\ref{CyclSub} to the smallest nonabelian group and assume that~$G=S_3$, the symmetric group on three letters. We also assume that the genus is~$g=1$. The group~$S_3$ contains six elements, namely the identity mapping as unit element~$e$, three transposition~$t_1$, $t_2$, and~$t_3$ with the property that~$i$ is the unique fixed point of~$t_i$, and two 3-cycles denoted by~$s \deq (1 2 3)$
and~$s^{-1} = (1 3 2)$. There are three conjugacy classes, namely~$I_1 \deq \{e\}$,
$I_2 \deq \{t_1, t_2, t_3\}$, and~$I_3 \deq \{s,s^{-1}\}$.

As we have seen in Paragraph~1.12, the fundamental group of the torus~$\Sigma_{1}$ is free abelian with the homotopy classes of the paths~$\alpha_1$ and~$\beta_1$ as a basis. The image of a group homomorphism from $\pi_1(\Sigma_{1},x)$ to~$S_3$ is therefore an abelian subgroup of~$S_3$. There are five such abelian subgroups, namely the trivial group~$T \deq \{e\}$ consisting only of the identity element, three subgroups~$U_i \deq \{e,t_i\}$ of order~$2$, and the alternating
group~\mbox{$A_3 = \{e, s , s^{-1}\}$} of order~$3$. The groups~$T$ and~$A_3$ are normal, while the three groups of order~$2$ are conjugate. This means that we have three conjugacy classes of abelian subgroups, namely~$O_1 \deq \{T\}$,~$O_2 \deq \{U_1, U_2, U_3\}$,
and~$O_3 \deq \{A_3\}$. As discussed in Paragraph~\ref{GenCase}, this leads to the decomposition
\[M^1 = M^1(O_1) \oplus M^1(O_2) \oplus M^1(O_3).\]
As already stated in Paragraph~\ref{Abel}, the map
\[\check{\Theta} \colon \Omega^1 \to S_3^2,~f \mapsto v \deq (a,b)\]
where $a \deq f([\alpha_1])$ and $b \deq f([\beta_1])$, is injective. Its image consists of the pairs of commuting elements in~$S_3$, and it is equivariant if we define the \mbox{$S_3$-action} on pairs via $c.(a,b) \deq (cac^{-1}, cbc^{-1})$ for $a,b,c \in S_3$. As in Paragraph~\ref{Abel}, we have for $f \in \Omega^1$ that
\[\check{\Theta}([\psi].f) = (A^T)^{-1} \check{\Theta}(f)\]
if~$[\psi] \in \Gamma_1(x)$ corresponds to the matrix~\mbox{$A \in \Sp(2, \Z)$}, where we recall from Paragraph~1.12 that $\Sp(2, \Z) = \SL(2, \Z)$ and that this correspondence is bijective.

The map~$\check{\Theta}$ takes the unique element in~$\Omega^1_T$ to the pair~$(e,e)$. The images of the other groups are parametrized by the maps
\[\omega_{1,2} \colon \Omega^{1,2} \to \check{\Omega}^1_{U_i} = \{(e,t_i), (t_i,e), (t_i, t_i)\},~(\bar{k},\bar{l}) \mapsto (t_i^k, t_i^l) \]
in the case of~$U_i$ and
\[\omega_{1,3} \colon \Omega^{1,3} \to \check{\Omega}^1_{A_3},~(\bar{k},\bar{l}) \mapsto (s^k,s^l) \]
in the case of~$A_3$. Since the action of~$\SL(2, \Z)$ on~$\Omega^{1,2}$ and~$\Omega^{1,3}$ is transitive (cf.~\cite[Kap.~II, Satz~6.14, p.~183]{Hu}), we see that also the action of~$\Gamma_1(x)$ on~$\Omega^1_{T}$, $\Omega^1_{U_i}$, and~$\Omega^1_{A_3}$ is transitive.

We now discuss the form that our decomposition takes for the various subgroups. It is obvious that in the case~$U=T$, both the action of the mapping class
group~$\Gamma_1(x)$ and the action of~$S_3$ are trivial. If~$U=U_i$, we have that both the centralizer~$C$ and the normalizer~$N$ are equal to~$U_i$, so that~$Z=N/C$ is trivial. From the formula stated before Corollary~\ref{DecompCohCyl} in Paragraph~\ref{CyclSub}, we therefore get that~$M^1(O_2)$ decomposes completely into a tensor product
\[M^1(O_2) \cong K[I_2] \otimes_{K} K[\Omega^{1,2}]\]
over the base field~$K$. This map is $S_3$-equivariant if~$S_3$ acts on $I_2$ by conjugation, and it is also \mbox{$\SL(2, \Z)$-equi}\-variant for the $\SL(2, \Z)$-action coming from the right factor~$\Omega^{1,2}$.

In the case where~$U = A_3$, we have~$N = S_3$ and~$C = A_3$, so that \mbox{$Z=N/C \cong \Z_2$}.
The homomorphism~$\rho \colon Z \to \Z_3^\times$ maps the unique nonidentity element to~$-\bar{1}$, and in the quotient~$I_3 \times_Z \Omega^{1,3}$ of the Cartesian product, the action of~$z \in Z$ on the right factor~$\Omega^{1,3}$ is given via multiplication by~$\rho(z)$. On the other hand, the unique nonidentity element of~$Z$ interchanges~$s$ and~$s^{-1}$ when acting on the left factor~$I_3$.

If the characteristic of the base field~$K$ is different from~$2 = \phi(3)$, then~$-1_K$ is a primitive second root of unity, and we can apply the decomposition
\begin{align*}
K[I_3] \ot_{K[Z]} K[\Omega^{1,3}]
= \bigoplus_{\chi \in \hat{\Z}_3^\times}
K_{\chi \circ \rho}[I_3/Z] \ot K_\chi[\P_3^{1}]
\end{align*}
stated before Corollary~\ref{DecompCohEig}. We have $\hat{\Z}_3^\times = \{\chi_1, \chi_2\}$, where the first Dirichlet character~$\chi_1$ is constantly equal to~$1_K$ and the second Dirichlet character~$\chi_2$ satisfies~$\chi_2(-\bar{1}) = -1_K$. The eigenspace
$K_{\chi_1 \circ \rho}[I_3/Z]$ is spanned by~$s + s^{-1}$, while the eigenspace
$K_{\chi_2 \circ \rho}[I_3/Z]$ is spanned by~$s - s^{-1}$. On the first space, $S_3$ operates trivially, while on the second, it operates via the sign representation. We therefore get the decomposition
\begin{align*}
K[I_3] \ot_{K[Z]} K[\Omega^{1,3}]
= K[s + s^{-1}] \ot K_{\chi_1}[\P_3^{1}]
\oplus K[s - s^{-1}] \ot K_{\chi_2}[\P_3^{1}].
\end{align*}

If the characteristic of the base field~$K$ is equal to~$2$, we cannot apply this decomposition. In this case, the $S_3$-module~$K[I_3]$ is indecomposable, as the unique nontrivial submodule is spanned by~$s + s^{-1}$. The set $\Omega^{1,3}$ consists of the eight nonzero vectors in~$\Z_3^2$, and we can choose four vectors~$v_1$, $v_2$, $v_3$, and~$v_4$ so that
\[\Omega^{1,3} = \{v_1, v_2, v_3, v_4, -v_1, -v_2, -v_3, -v_4\}.\]
Then the set $I_3 \times_Z \Omega^{1,3}$ consists of the eight equivalence classes
\[\overline{(s,v_i)} = \{(s,v_i), (s^{-1},-v_i)\} \qquad \text{and} \qquad
\overline{(s^{-1},v_i)} = \{(s^{-1},v_i), (s,-v_i)\}\]
for $i=1,2,3,4$. So the mapping
\[I_3 \times \{v_1, v_2, v_3, v_4\} \to I_3 \times_Z \Omega^{1,3},\]
which is the composition of the inclusion into $I_3 \times \Omega^{1,3}$ and the canonical quotient map, is bijective and $S_3$-equivariant. This implies that
\[K[I_3] \ot_{K[Z]} K[\Omega^{1,3}] \cong K[I_3]^4\]
as $S_3$-modules.

\section{Computation of the cohomology groups} \label{Sec:CompCoh}
\subsection{Cohomology in degree zero} \label{DegZero}
In this section, we reach our goal to determine explicitly the derived block spaces for the Drinfel'd double of~$S_3$ in genus~$g=1$, i.e., for the torus, and their module structure over the mapping class group~$\Gamma_1 \cong \SL(2,\Z)$. Recall that, in view \linebreak of Corollary~\ref{Group},
the derived block spaces are isomorphic to the cohomology \linebreak groups~$H^m(S_3, M^1)$, and we will now use the decomposition described in Paragraph~\ref{SymmThree} to compute these groups. Another important tool will be the periodic resolution for the cyclic group, which is reviewed in Paragraph~\ref{PerResCycl} of the appendix.

In degree zero, we find by construction the non-derived block space, i.e., the space $\Ext_D^0(K, L) \cong \Hom_D(K, L)$, which is by Corollary~\ref{Group} isomorphic to~$H^0(S_3, M^1)$. As already mentioned toward the end of Paragraph~\ref{RedGroupCohom}, this cohomology group is given by the fixed points of the action, so that the decomposition described in Paragraph~\ref{SymmThree} yields
\[(M^1)^{S_3} = M^1(O_1)^{S_3} \oplus M^1(O_2)^{S_3} \oplus M^1(O_3)^{S_3}.\]
Because the action of~$S_3$ commutes with the action of the mapping class group, the modular group~$\SL(2,\Z)$ still acts on the fixed points, and this action is trivial in the case of~$M^1(O_1)^{S_3} \cong K$. In the case of the second summand, the decomposition
\[M^1(O_2) \cong K[I_2] \otimes_{K} K[\Omega^{1,2}]\]
found in Paragraph~\ref{SymmThree} shows that $M^1(O_2)^{S_3} \cong K[\Omega^{1,2}]$ as $\SL(2,\Z)$-modules.
Note that $\Omega^{1,2} \cong \P^1_2$, since~$\Z_2^\times$ is trivial. In the case of the third summand, the discussion at the end of Paragraph~\ref{SymmThree} implies that the space of fixed points is four-dimensional, and that the four vectors
\[\overline{(s,v_i)} + \overline{(s^{-1},v_i)} =
\overline{(s^{-1},-v_i)} + \overline{(s,-v_i)} \]
form a basis of this space. This shows that $M^1(O_3)^{S_3}$ is isomorphic to~$K[\P^1_3]$ as an $\SL(2,\Z)$-module.

We note that non-derived block spaces have been investigated by many authors (cf.~for example \cite{FF}, \cite{KSSB}, \cite{SZ}). As we have explained in Paragraph~3.4, in our case this space is isomorphic to the zeroth Hochschild cohomology group~$HH^0(D,D)$ of the Drinfel'd double~$D$, which is in turn equal to the center~$Z(D)$ of~$D$ (cf.~\cite[Sec.~1, Examp.~1.1, p.~403]{LMSS1}; \cite[Sec.~1.2, p.~10]{W2}).

\subsection{Cohomology in characteristic~$2$} \label{Char2}
If the characteristic of the base field~$K$ is different from~$2$ or~$3$, the group algebra~$K[S_3]$ is semisimple by Maschke's theorem, and consequently the higher cohomology groups~$H^m(S_3, M^1)$ with $m>0$ vanish. This is also a consequence of the fact that these cohomology groups are annihilated by multiplication with both the group order and the characteristic of the base field (cf.~\cite[Chap.~2, \S~7, Cor.~2, p.~210]{Su}). So we only need to treat these two characteristics, and we begin with the case that the characteristic of~$K$ is~$2$. To deal with this case, we will need two closely related lemmata:
\begin{Lemma} \label{ResChar2}
Suppose that~$W$ is a $K[S_3]$-module. Then the restriction map
\[\Res^{S_3}_{U_i} \colon H^m(S_3, W) \to H^m(U_i, W)\]
is a $K$-linear isomorphism for all $i=1,2,3$ and all~$m > 0$.
\end{Lemma}
\begin{proof}
This follows from~\cite[Chap.~2, \S~7, (7.28), p.~211]{Su}. Note that there group cohomology is defined over the integers~$\Z$, while we here work over our base field~$K$ of characteristic~$2$, which in particular forces the cohomology to be 2-torsion. The connection between the two approaches is explained in~\cite[Sec.~1.1, p.~3]{E}, as we already mentioned in the proof of Proposition~\ref{GroupCohom}.
\end{proof}

To put this lemma to use, we will need an additional fact. For a $K[S_3]$-module~$W$, we consider the space~$W^{A_3}$ of invariants under the action of the alternating group~$A_3$. There is a natural inclusion mapping~$\iota \colon W^{A_3} \to W$ as well as a projection mapping
\[\pi \colon W \to W^{A_3},~w \mapsto \frac{1}{3} (w + s.w + s^{-1}.w).\]
The fact that~$A_3$ is normal in~$S_3$ implies that~$\pi$ is $S_3$-equivariant, so that in particular its image~$W^{A_3}$ is an $S_3$-submodule. These two modules have isomorphic cohomology groups:
\begin{Lemma} \label{InclChar2}
The map $\iota_* : H^m(S_3,W^{A_3}) \to H^m(S_3,W)$ is bijective.
\end{Lemma}
\begin{proof}
\begin{pflist}
\item
As just recalled in Paragraph~\ref{DegZero}, we have $H^m(S_3,W) = W^{S_3}$ if~$m=0$, and taking invariants in~$W^{A_3}$ yields the same result, so the assertion holds. We will now consider the case~$m>0$. In view of Lemma~\ref{ResChar2}, it is sufficient to show that the analogous map
\mbox{$\iota_* : H^m(U_i,W^{A_3}) \to H^m(U_i,W)$} is bijective. Since $\pi \circ \iota = \id$, we have
\[W = \im(\iota) \oplus \ker(\pi).\]
To prove our assertion, we need to establish that $H^m(U_i,\ker(\pi)) = 0$.

\item
We first treat the case where~$K$ contains a primitive third root of unity~$\zeta$. We then have
$(\zeta - 1)(\zeta^2 + \zeta + 1) = \zeta^3 - 1 = 0$, which implies
$\zeta^2 + \zeta + 1 = 0$. Since $s^3 = e$, the action of~$s$ on~$W$ satisfies a polynomial with distinct roots, and therefore~$W$~decomposes into the eigenspaces
\[W_j \deq \{w \in W \mid s.w = \zeta^{j} w\}\]
for $j=0,1,2$. From the equation derived above, we get that $W_0 = \im(\iota)$ and $W_1 \oplus W_2 = \ker(\pi)$. If $w \in W_1$, we get from the relation $s t_i = t_i s^{-1}$ that
\[s t_i.w = t_i s^{-1}.w = \zeta^{-1} t_i.w,\]
so that $t_i.w \in W_2$. A very similar computation shows that $t_i.w \in W_1$ if
$w \in W_2$.

We now choose a basis $w_1,\ldots,w_n$ of~$W_1$. Then $t_i.w_1,\ldots,t_i.w_n$ is a basis of~$W_2$.
The spaces $P_k := \Span_K(w_k,t_i.w_k)$ are subspaces that are invariant under both~$t_i$ and~$s$ and are therefore two-dimensional~$K[S_3]$-submodules with
$$W_1 \oplus W_2 = P_1 \oplus \dots \oplus P_n.$$
All of them are isomorphic to $P_1$ under the isomorphism that maps~$w_1$ to~$w_k$ and
$t_i.w_1$ to~$t_i.w_k$. The restriction of~$P_k$ to~$K[U_i]$ is actually isomorphic to~$K[U_i]$ under the isomorphism that maps~$1$ to~$w_k$ and~$t_i$ to~$ t_i.w_k$. We therefore have that
\[H^m(U_i,W_1 \oplus W_2) \cong H^m(U_i,P_1) \oplus \dots \oplus H^m(U_i,P_n) = 0\]
by Corollary~\ref{CohomReg}, as required.

\item
If~$K$ does not contain a primitive third root of unity, there is a finite field extension~$K' \supset K$ that does contain a primitive third root of unity
(cf.~\cite[Chap.~I, \S~5, Thm.~6, p.~32]{J}), and we can consider the $S_3$-module~$W \ot_K K'$.
From general principles of multilinear algebra (cf.~\cite[Sec.~1.19, p.~24]{Gr2}), it follows that
\[\ker(\pi \ot \id_{K'}) = \ker(\pi) \ot K'.\]
Using the same principles, the description of the cohomology groups given in Proposition~\ref{CohomCycl} shows that
\[H^m(U_i,\ker(\pi) \ot K') = H^m(U_i,\ker(\pi)) \ot K'.\]
This fact is also a consequence of the simple version of the universal coefficient theorem already used in the proofs of Theorem~\ref{ThmDecomp} and Corollary~\ref{DecompCohEig}. But by the preceding step, the left-hand side of this equation vanishes, which establishes our claim that~$H^m(U_i,\ker(\pi)) = 0$.
\qedhere
\end{pflist}
\end{proof}

We can now compute the arising mapping class group representations on the derived block spaces. In this case, there are no representations in higher degree that did not already appear in degree zero, which we considered in Paragraph~\ref{DegZero}:
\begin{Theorem} \label{ModGroupChar2}
We have
\[\TFT^m(\Sigma_{1}) \cong
\begin{cases}
K \oplus K[\P^1_2] \oplus K[\P^1_3] &: m = 0\\
K \oplus K[\P^1_2] &: m \neq 0
\end{cases}\]
as $\SL(2,\Z)$-modules.
\end{Theorem}
\begin{proof}
\begin{pflist}
\item
From Corollary~\ref{Group}, we know that $\TFT^m(\Sigma_{1}) \cong H^m(S_3, M^1)$, and in view of the discussion in Paragraph~\ref{SymmThree}, we have
\[H^m(S_3, M^1) \cong H^m(S_3, M^1(O_1)) \oplus H^m(S_3, M^1(O_2))
\oplus H^m(S_3, M^1(O_3)).\]
We have already treated the case $m=0$ in Paragraph~\ref{DegZero}, so let us now assume that~\mbox{$m>0$}. By Lemma~\ref{ResChar2}, we have
\[H^m(S_3, M^1(O_1)) \cong H^m(U_i, M^1(O_1)) \cong H^m(\Z_2, K),\]
and this is isomorphic to~$K$ by the discussion following Proposition~\ref{CohomCycl}. As already stated in Paragraph~\ref{DegZero}, the $\SL(2,\Z)$-module structure on this space is trivial.

\item
For the second summand, we have seen in Paragraph~\ref{SymmThree} that
\[M^1(O_2) \cong K[I_2] \otimes_{K} K[\Omega^{1,2}],\]
where $\Omega^{1,2} \cong \P^1_2$, as already stated in Paragraph~\ref{DegZero}. On the right-hand side, the~$S_3$-action is on the first factor~$K[I_2]$ and the $\SL(2,\Z)$-action is on the second factor~$K[\P^1_2]$.

The space $K[I_2]^{A_3}$ is spanned by~$t_1 + t_2 + t_3$ and is therefore isomorphic to~$K$ as an $S_3$-module. By Lemma~\ref{ResChar2}, Lemma~\ref{InclChar2}, and the fact that $H^m(U_i, K) \cong K$ just observed, we therefore get
\[H^m(S_3, K[I_2]) \cong H^m(S_3, K[I_2]^{A_3}) \cong H^m(S_3, K)
\cong H^m(U_i, K) \cong K.\]
Now Corollary~\ref{DecompCohCyl} yields
\[H^m(S_3, M^1(O_2)) \cong H^m(S_3, K[I_2]) \otimes_{K} K[\P^1_2] \cong K[\P^1_2]\]
because~$Z$ is trivial in this case.
By construction, this isomorphism is $\SL(2, \Z)$-equi\-variant.

\item
For the third summand, we have seen at the end of Paragraph~\ref{SymmThree} that~$M^1(O_3)$ is isomorphic to~$K[I_3]^4$ as an $S_3$-module. But if we restrict the action of~$S_3$ on~$K[I_3]$ to any~$U_i$, the resulting module is isomorphic to~$K[U_i]$ under the $K$-linear bijection that maps~$1$ to~$s$ and~$t_i$ to~$s^{-1}$. By Lemma~\ref{ResChar2} and Corollary~\ref{CohomReg}, we therefore get
\[H^m(S_3, M^1(O_3)) \cong H^m(S_3, K[I_3]^4) \cong H^m(U_i, K[I_3])^4
\cong H^m(U_i, K[U_i])^4 = 0\]
as required.
\qedhere
\end{pflist}
\end{proof}

\subsection{The Yoneda module structure in characteristic~$2$} \label{YonChar2}
We have explained in Paragraph~\ref{RedGroupCohom} that~$\TFT^\infty(\Sigma_1)$ is a right module over $\TFT^\infty(\Sigma_0)$ via the Yoneda product, and we have also explained there that this module structure reduces in our case to the right module structure of~$H(S_3, M^1)$ over~$H(S_3, K)$. Staying in the case where~$K$ has characteristic~$2$, we now make this module structure explicit.

We first need to understand the ring structure of~$H(S_3, K)$. For this, let~$V$ be a two-dimensional vector space over~$K$ with basis~$b_1,b_2$, and consider the unique $K$-linear map~$f \colon V \to V$ with the property that
\[f(b_1) = b_1 \qquad \text{and} \qquad f(b_2) = b_1 + b_2.\]
This map has order~$2$, and we can therefore introduce a $\Z_2$-module structure on~$V$ by requiring that the unique nonzero element acts via~$f$. Since~$S_3/A_3 \cong \Z_2$, we can pull this module structure back along the quotient map to get an $S_3$-module structure on~$V$ in which~$s$ acts as the identity and all~$t_i$ act via~$f$.

The unique $K$-linear map from~$K$ to~$V$ that maps~$1_K$ to~$b_1$ is clearly $S_3$-equivariant. Similarly, the unique $K$-linear map from~$V$ to~$K$ that maps~$b_1$ to zero and~$b_2$ to~$1_K$ is $S_3$-equivariant. We therefore have an extension
\[E \colon 0 \xlongrightarrow{} K \xlongrightarrow{} V \xlongrightarrow{} K
\xlongrightarrow{} 0\]
that represents an element in~$\Yext_{S_3}^1(K,K) \cong H^1(S_3, K)$. This element is not zero: Since the action on~$V$ is not trivial, we do not have $V \cong K \oplus K$, so that the extension does not split (cf.~\cite[Chap.~III, Sec.~2, p.~93]{HiS}).

If~$K[x]$ denotes the polynomial ring in one indeterminate~$x$, the universal property of the polynomial ring yields that there is a unique algebra homomorphism from~$K[x]$ to~$H(S_3, K)$ that maps~$x$ to the class of~$E$. It turns out that this homomorphism is an isomorphism:
\begin{Lemma} \label{IsomChar2}
The algebra homomorphism~$K[x] \to H(S_3, K)$ that maps~$x$ to the class of~$E$ is bijective.
\end{Lemma}
\begin{proof}
It follows from Lemma~\ref{ResChar2} that the restriction map
\[\Res^{S_3}_{U_i} \colon H(S_3, K) \to H(U_i, K)\]
is bijective for all $i=1,2,3$, and it is clearly an algebra homomorphism with respect to the Yoneda product. It is therefore sufficient to show that the composition of our algebra homomorphism with the restriction mapping is bijective. But since~$U_i \cong \Z_2$, the restriction of~$E$ to~$U_i$ is essentially the original $\Z_2$-module structure that we pulled back along the quotient map to~$S_3/A_3 \cong \Z_2$ to construct the $S_3$-module structure. Now the required property for $\Z_2$-modules is established in \cite[Prop.~3.5.5, p.~67]{Be} or \cite[Sec.~3.2, p.~26]{E}.
\end{proof}

We now turn to the right module structure of~$H(S_3, M^1)$ over~$H(S_3, K)$. The decomposition $M^1 = M^1(O_1) \oplus M^1(O_2) \oplus M^1(O_3)$ from Paragraph~\ref{SymmThree} leads to the decomposition
\[H(S_3, M^1) = H(S_3, M^1(O_1)) \oplus H(S_3, M^1(O_2)) \oplus H(S_3, M^1(O_3))\]
into submodules for the action of~$H(S_3, K)$. Their structure is described in the following theorem:
\begin{Theorem} \label{ModStrucChar2}
The $H(S_3, K)$-module~$H(S_3, M^1)$ is generated by~$H^0(S_3, M^1)$. For the summands, we have
\begin{enumerate}
\item
$H(S_3, M^1(O_1))$ is free of rank~$1$ over~$H(S_3, K)$. The module action induces an $\SL(2, \Z)$-equivariant isomorphism
\[H^0(S_3, M^1(O_1)) \ot H(S_3, K) \to H(S_3, M^1(O_1)).\]

\item
$H(S_3, M^1(O_2))$ is free of rank~$3$ over~$H(S_3, K)$. The module action induces an $\SL(2, \Z)$-equivariant isomorphism
\[H^0(S_3, M^1(O_2)) \ot H(S_3, K) \to H(S_3, M^1(O_2)).\]

\item
$H(S_3, M^1(O_3)) = H^0(S_3, M^1(O_3)) \cong K[\P^1_3]$ is a trivial $H(S_3, K)$-module.
\end{enumerate}
\end{Theorem}
\begin{proof}
\begin{pflist}
\item
Recall that we have $H^0(S_3, M) = \Yext_{S_3}^0(K,M) = \Hom_{S_3}(K,M)$ for any $S_3$-module~$M$, and by the definition reviewed in Paragraph~\ref{YonProd}, the module action
\[H^0(S_3, M) \ot H(S_3, K) \to H(S_3, M)\]
given by the Yoneda product maps an element~$f \ot \bar{F}$, for~$f \in \Hom_{S_3}(K,M)$ and \mbox{$\bar{F} \in \Yext_{S_3}^m(K,K)$}, to $\Yext_{S_3}^m(K,f)(\bar{F})$, which is the class of~$fF$ in the notation used in that paragraph. This immediately implies our bijectivity assertion in the case where~$M=M^1(O_1)$, because~$M^1(O_1) \cong K$. The $\SL(2, \Z)$-equivariance follows from Proposition~\ref{MapClassYon}, or alternatively from the fact that the $\SL(2, \Z)$-module structure on both sides is trivial.

\item
For the second assertion, the $\SL(2, \Z)$-equivariance follows again from Proposition~\ref{MapClassYon}. To establish bijectivity, we consider the square
\begin{center}
\begin{tikzcd}[column sep = 105 pt] {}
\Hom_{S_3}(K, M^1(O_2)^{A_3}) \ot \Yext^m(K, K) \arrow{r}{f \ot \bar{F} \mapsto \Yext_{S_3}^m(K,f)(\bar{F})}
\arrow{d}{\iota_* \ot \id} &
\Yext^m(K, M^1(O_2)^{A_3}) \arrow{d}{\iota_*} \\
\Hom_{S_3}(K, M^1(O_2)) \ot \Yext^m(K, K) \arrow{r}{f \ot \bar{F} \mapsto \Yext_{S_3}^m(K,f)(\bar{F})} &
\Yext^m(K, M^1(O_2))
\end{tikzcd}
\end{center}
The square is commutative by the functoriality of~$\Yext^m$. The vertical maps are bijective by Lemma~\ref{InclChar2}. To prove our assertion, which is that the lower map is bijective, it is therefore sufficient to show that the upper map is bijective.

In the proof of Theorem~\ref{ModGroupChar2}, we have seen that~$M^1(O_2)^{A_3}$ is trivial as an $S_3$-module, and is therefore isomorphic to~$K^3$. The assertion then becomes that
\[\Hom_{S_3}(K, K^3) \ot H^m(S_3, K) \to H^m(S_3, K^3),~f \ot \bar{F}
\mapsto \Yext_{S_3}^m(K,f)(\bar{F})\]
is bijective. But this follows from the fact, mentioned in Paragraph~\ref{YonProd}, that~$\Yext^m$ is an additive functor.

\item
For the third assertion, we have already seen in the proof of Theorem~\ref{ModGroupChar2} that $H^m(S_3, M^1(O_3)) = 0$ if~$m>0$, so that
$H(S_3, M^1(O_3)) = H^0(S_3, M^1(O_3))$. In Paragraph~\ref{DegZero}, we have seen that
$H^0(S_3, M^1(O_3)) = M^1(O_3)^{S_3} \cong K[\P^1_3]$ as an $\SL(2, \Z)$-module, and we have also observed there that this space is four-dimensional over~$K$. For the module structure over~$H(S_3, K) = \Yext_{S_3}(K,K)$, note that $H(S_3, M^1(O_3))$ is a graded module over~$H(S_3, K)$ that has only zero terms in higher degree, which implies that $H^m(S_3, K)$ must act as zero if~$m>0$. This implies that the module structure is trivial in the sense introduced in Paragraph~\ref{YonProd}.
\qedhere
\end{pflist}
\end{proof}

We note that this theorem shows in particular that, in the nonabelian case,~$\TFT^\infty(\Sigma_1)$ is not always free over~$\TFT^\infty(\Sigma_0)$, in contrast to the abelian case treated in Proposition~\ref{AbelFreeGen}.

\subsection{Cohomology in characteristic~$3$} \label{Char3}
We now turn to the case where the base field~$K$ has characteristic~$3$. The symmetric group~$S_3$ contains the alternating group~$A_3$ as a normal subgroup and therefore acts on~$A_3$ by conjugation. As explained in Paragraph~\ref{PerResCycl} of the appendix, this action leads to a corresponding action on the cohomology groups of~$A_3$. The action of~$A_3$ on its cohomology groups is trivial (cf.~\cite[Prop.~4.1.1, p.~36]{E}), so that the action of~$S_3$ factors over the quotient~$S_3/A_3 \cong \Z_2$. We can therefore talk about the fixed points of this action of~$S_3$, or equivalently of the quotient~$S_3/A_3$. These fixed points appear in the following analogue of Lemma~\ref{ResChar2} for the case of characteristic~$3$:
\begin{Lemma} \label{ResChar3}
Suppose that~$W$ is a $K[S_3]$-module. Then the restriction map
\[\Res^{S_3}_{A_3} \colon H^m(S_3, W) \to H^m(A_3, W)^{S_3/A_3}\]
is a $K$-linear isomorphism.
\end{Lemma}
\begin{proof}
In view of the points already made in the proof of Lemma~\ref{ResChar2}, this follows from~\cite[Chap.~III, Thm.~(10.3), p.~84]{Br}. We note that this result, as well as Lemma~\ref{ResChar2}, can also be deduced from the Lyndon-Hochschild-Serre spectral sequence (cf.~\cite[Chap.~III, \S~2, Thm.~1, p.~127]{HoS}; \cite[Chap.~XI, \S~10, Cor.~10.3, p.~354]{ML1}). Furthermore, we note that the assertion is also correct in the case~$m=0$, because the cohomology groups in degree zero are given via fixed points, as recalled in Paragraph~\ref{DegZero}, and we have
$(W^{A_3})^{S_3/A_3} = W^{S_3}$.
\end{proof}

Recall from Paragraph~\ref{SymmThree} that the
character group $\hat{\Z}_3^\times = \Hom(\Z_3^\times, K^\times)$ of the group
$\Z_3^\times = \{\bar{1}, -\bar{1}\}$ consists of two Dirichlet characters~$\chi_1$ and~$\chi_2$, where the first one is the trivial character and the second one is given by
$\chi_2(\bar{1}) = 1_K$ and~$\chi_2(-\bar{1}) = -1_K$.
The $\SL(2, \Z)$-representation~$K_{\chi_2}[\P^1_3]$ already discussed there appears in the following analogue of Theorem~\ref{ModGroupChar2} for fields of characteristic~$3$:
\begin{Theorem} \label{ModGroupChar3}
We have $\TFT^0(\Sigma_{1}) \cong K \oplus K[\P^1_2] \oplus K[\P^1_3]$. For~$m > 0$, we have
\[\TFT^m(\Sigma_{1}) \cong
\begin{cases}
K \oplus K[\P^1_3] &: m \equiv 0 \pmod{4} \quad \text{or} \quad m \equiv 3 \pmod{4}\\
K_{\chi_2}[\P^1_3] &: m \equiv 1 \pmod{4}  \quad \text{or} \quad m \equiv 2 \pmod{4}\\
\end{cases}\]
as $\SL(2,\Z)$-modules.
\end{Theorem}
\begin{proof}
\begin{pflist}
\item
We proceed as in the proof of Theorem~\ref{ModGroupChar2}. As we said already there, we know from Corollary~\ref{Group} that $\TFT^m(\Sigma_{1}) \cong H^m(S_3, M^1)$, and in view of the discussion in Paragraph~\ref{SymmThree}, we have
\[H^m(S_3, M^1) \cong H^m(S_3, M^1(O_1)) \oplus H^m(S_3, M^1(O_2)) \oplus H^m(S_3, M^1(O_3)).\]
We have already treated the case $m=0$ in Paragraph~\ref{DegZero}, so let us now assume that~\mbox{$m>0$}. By Lemma~\ref{ResChar3}, we have for the first summand that
\[H^m(S_3, M^1(O_1)) \cong H^m(A_3, M^1(O_1))^{S_3/A_3} \cong H^m(A_3, K)^{S_3/A_3}.\]
From the discussion following Proposition~\ref{CohomCycl}, we get
that~$H^m(A_3, K) \cong K$. To compute the fixed points, we apply the formula given in Proposition~\ref{CohomCyclSubgr}:

The unique nonidentity element in~$S_3/A_3$ is represented by any~$t_i$, and the formula
$t_i s t_i = s^{-1}$ shows that the number~$k$ that appears in Proposition~\ref{CohomCyclSubgr} is equal \mbox{to~$k = -1 \equiv 2 \pmod{3}$}. This proposition therefore implies that~$t_i$ acts on the space $H^m(A_3, K) \cong K$ by multiplication with~$(-1)^{m/2}$ if~$m$ is even, and by multiplication with~$(-1)^{(m+1)/2}$ if~$m$ is odd. This shows that
\[H^m(S_3, M^1(O_1)) \cong
\begin{cases}
K &: m \equiv 0 \pmod{4} \quad \text{or} \quad m \equiv 3 \pmod{4}\\
0 &: m \equiv 1 \pmod{4}  \quad \text{or} \quad m \equiv 2 \pmod{4}\\
\end{cases}\]
Clearly, the $\SL(2,\Z)$-module structure on this space is trivial in all cases.

\item
For the second summand, we note that the formula $s t_i s^{-1} = t_{s(i)}$ implies for $i=1$ that the bijection
\[A_3 \to I_2, \qquad \id \mapsto t_1, \qquad s \mapsto t_2, \qquad s^{-1} \mapsto t_3\]
is $A_3$-equivariant. From Corollary~\ref{CohomReg}, we therefore get
\[H^m(A_3, K[I_2]) \cong H^m(A_3, K[A_3]) = 0,\]
which implies that
\[H^m(S_3, K[I_2]) \cong H^m(A_3, K[I_2])^{S_3/A_3} = 0\]
by Lemma~\ref{ResChar3}. Now Corollary~\ref{DecompCohCyl} gives that
$H^m(S_3, M^1(O_2)) = 0$.

\item
For the third summand, we have seen in Paragraph~\ref{SymmThree} that the general decomposition from Paragraph~\ref{CyclSub} yields that
\begin{align*}
M^1(O_3)
\cong K[s + s^{-1}] \ot K_{\chi_1}[\P_3^{1}]
\oplus K[s - s^{-1}] \ot K_{\chi_2}[\P_3^{1}],
\end{align*}
where~$S_3$ acts only on the first tensor factor and~$\SL(2, \Z)$ only on the second one. The action of~$S_3$ on~$K[s + s^{-1}]$ is clearly trivial, and we have determined the arising cohomology groups already in the treatment of the first summand above. The action of~$S_3$ on~$K[s - s^{-1}]$ is isomorphic to the sign representation, and the arising cohomology groups can also be determined via Proposition~\ref{CohomCyclSubgr}. But because~$t_i$ now acts as multiplication by~$-1$, we find that its induced action on~$H^m(A_3, K[s - s^{-1}]) \cong K$ is by multiplication with~$-(-1)^{m/2}$ if~$m$ is even, and by multiplication with~$-(-1)^{(m+1)/2}$ if~$m$ is odd. Since
$H^m(S_3, K[s - s^{-1}]) \cong H^m(A_3, K[s - s^{-1}])^{S_3/A_3}$ by Lemma~\ref{ResChar3}, we obtain
\[H^m(S_3, K[s - s^{-1}]) \cong
\begin{cases}
0 &: m \equiv 0 \pmod{4} \quad \text{or} \quad m \equiv 3 \pmod{4}\\
K &: m \equiv 1 \pmod{4}  \quad \text{or} \quad m \equiv 2 \pmod{4}\\
\end{cases}\]
In total, we get from Corollary~\ref{DecompCohEig} that
\begin{align*}
H^m(S_3, M^1(O_3))
&\cong H^m(S_3, K[s + s^{-1}]) \ot K_{\chi_1}[\P_3^{1}] \oplus
H^m(S_3, K[s - s^{-1}]) \ot K_{\chi_2}[\P_3^{1}] \\
&\cong
\begin{cases}
K[\P_3^{1}] &: m \equiv 0 \pmod{4} \quad \text{or} \quad m \equiv 3 \pmod{4}\\
K_{\chi_2}[\P_3^{1}] &: m \equiv 1 \pmod{4}  \quad \text{or} \quad m \equiv 2 \pmod{4}\\
\end{cases}
\end{align*}
since $K_{\chi_1}[\P_3^{1}] \cong K[\P_3^{1}]$. Our assertion now follows by combining the results for the three summands.
\qedhere
\end{pflist}
\end{proof}

We note that, although both~$K_{\chi_1}[\P_3^{1}]$ and~$K_{\chi_2}[\P_3^{1}]$ are four-dimensional representations of the modular group~$\SL(2, \Z)$, these two representations are not isomorphic, as~$\fs^2$ acts as the identity in the first representation and as the multiplication by~$-1$ in the second representation. This shows that the mapping class group representations that appear in higher degree are, in general, different from those that appear in degree zero.

\subsection{The Yoneda algebra in characteristic~$3$} \label{YonAlgChar3}
In Paragraph~\ref{YonChar2}, we have determined the right module structure of~$H(S_3, M^1)$ over~$H(S_3, K)$ in the case where~$K$ has characteristic~$2$. To carry out the same analysis in the case where~$K$ has characteristic~$3$, we first need to understand the ring structure of~$H(S_3, K)$, which in the present case will be constructed as a subring of~$H(A_3, K)$. So we now assume that the characteristic of~$K$ is~$3$.

As in Paragraph~\ref{YonChar2}, we let~$V$ be a two-dimensional vector space over~$K$ with basis~$b_1,b_2$, and consider the unique $K$-linear map~$f \colon V \to V$ with the property that
\[f(b_1) = b_1 \qquad \text{and} \qquad f(b_2) = b_1 + b_2.\]
This map has now order~$3$, and we can therefore introduce an $A_3$-module structure on~$V$ by requiring that~$s$ acts via~$f$. As in the previous case, the unique $K$-linear map from~$K$ to~$V$ that maps~$1_K$ to~$b_1$ is $A_3$-equivariant, and similarly the unique $K$-linear map from~$V$ to~$K$ that maps~$b_1$ to zero and~$b_2$ to~$1_K$ is $A_3$-equivariant. We therefore have an extension
\[E \colon 0 \xlongrightarrow{} K \xlongrightarrow{} V \xlongrightarrow{} K
\xlongrightarrow{} 0\]
that represents an element in~$\Yext_{A_3}^1(K,K) \cong H^1(A_3, K)$. By the same argument as in the previous case, we can see that this element is not zero.

However, in contrast to the previous case, the class of this extension does not generate the cohomology ring~$H(A_3, K)$: We also need to consider the extension
\[F \colon 0 \xlongrightarrow{} K \xlongrightarrow{} K[A_3] \xlongrightarrow{D.} K[A_3] \xlongrightarrow{\varepsilon} K \xlongrightarrow{} 0\]
where, as in Paragraph~\ref{PerResCycl}, $\varepsilon$ denotes the standard augmentation of the group algebra, $D.$ denotes the multiplication
by the difference~$D \deq s - 1$, and the first map is uniquely determined by the property that it maps~$1_K$ to the norm~$e + s + s^{-1}$.

It can be shown that the classes~$\bar{E}$ and~$\bar{F}$ in~$\Yext_{A_3}(K,K)$ commute and that~\mbox{$\bar{E}^2 = 0$} (cf.~\cite[Prop.~3.5.5, p.~67]{Be}; \cite[Sec.~3.2, p.~26]{E}). If~$K[x,y]$ denotes the polynomial algebra in two commuting indeterminates~$x$ and~$y$, the universal property of the polynomial algebra therefore yields that there is a unique algebra homomorphism from~$K[x,y]$ to~$H(S_3, K)$ that maps~$x$ to~$\bar{E}$ and~$y$ to~$\bar{F}$. The following lemma, whose proof can be found in the references just cited, states the basic properties of this algebra homomorphism:
\begin{Lemma} \label{IsomChar3}
The algebra homomorphism~$K[x,y] \to H(A_3, K)$ that maps~$x$ to~$\bar{E}$ and~$y$
to~$\bar{F}$ is surjective. Its kernel is the ideal generated by~$x^2$.
\end{Lemma}

The universal property of the polynomial algebra also implies that there is a unique algebra homomorphism from~$K[x,y]$ to itself that maps~$x$ to~$-x$ and~$y$ to~$-y$. This algebra homomorphism fixes~$x^2$ and therefore preserves the ideal that it generates. In view of the preceding lemma, we therefore have an algebra homomorphism
\[\gamma \colon H(A_3, K) \to H(A_3, K)\]
satisfying $\gamma(\bar{E}) = - \bar{E}$ and $\gamma(\bar{F}) = -\bar{F}$. This algebra homomorphism obviously has order~$2$, so that we have the decomposition
\[H(A_3, K) = S \oplus J\]
into the eigenspace~$S$ of~$\gamma$ for the eigenvalue~$1$ and the eigenspace~$J$ of~$\gamma$ for the eigenvalue~$-1$. The fact that~$\gamma$ is an algebra homomorphism implies that~$S$ is a subalgebra of~$H(A_3, K)$ and that~$J$ is a module over the commutative algebra~$S$. These spaces relate to the cohomology ring of~$S_3$ as follows:

\begin{Proposition} \label{CohomRingSym}
\begin{enumerate}
\item
For the trivial representation~$K$ of~$S_3$, the restriction mapping yields an isomorphism
$\Res^{S_3}_{A_3} \colon H(S_3, K) \to S$.

\item
For the sign representation~$K_\varepsilon$ of~$S_3$, the restriction mapping yields an isomorphism
$\Res^{S_3}_{A_3} \colon H(S_3, K_\varepsilon) \to J$.

\item
$S = K[\bar{E} \bar{F}, \bar{F}^2]$

\item
As an $S$-module, $J$ is generated by~$\bar{E}$ and~$\bar{F}$.
\end{enumerate}
\end{Proposition}
\begin{proof}
\begin{pflist}
\item
If~$m=2j$ is an even number, we have
\[\gamma(\bar{F}^j) = (-1)^j \bar{F}^j = (-1)^{m/2} \bar{F}^j.\]
As it follows from Lemma~\ref{IsomChar3} that~$\bar{F}^j$ is a basis element of~$H^m(A_3, K)$, which is one-dimensional according to the discussion following Proposition~\ref{CohomCycl}, we see that~$\gamma$ acts on~$H^m(A_3, K)$ by multiplication with~$(-1)^{m/2}$. In other words, it acts as the identity if~$m \equiv 0 \pmod{4}$ and as multiplication by~$-1$ if~$m \equiv 2 \pmod{4}$.

If~$m=2j+1$ is odd, we have
\[\gamma(\bar{E} \bar{F}^j) = (-1)^{j+1} \bar{E} \bar{F}^j
= (-1)^{(m+1)/2} \bar{E} \bar{F}^j.\]
Arguing as in the even case, we see that~$\gamma$ acts on~$H^m(A_3, K)$ by multiplication
with~$(-1)^{(m+1)/2}$. In other words, it acts as the identity if~$m \equiv 3 \pmod{4}$ and as multiplication by~$-1$ if~$m \equiv 1 \pmod{4}$.

Now we have already discussed in the proof of Theorem~\ref{ModGroupChar3} that it follows from Proposition~\ref{CohomCyclSubgr} that the unique nonidentity element in~$S_3/A_3$ acts on~$H^m(A_3, K)$ by multiplication with~$(-1)^{m/2}$ if~$m$ is even and by multiplication with~$(-1)^{(m+1)/2}$ if~$m$ is odd. In other words, it acts exactly by~$\gamma$. Therefore the first assertion follows from Lemma~\ref{ResChar3}.

\item
We have also discussed in the proof of Theorem~\ref{ModGroupChar3} that the unique nonidentity element in~$S_3/A_3$ acts on~$H^m(A_3, K_\varepsilon)$ by multiplication with~$-(-1)^{m/2}$ if~$m$ is even, and by multiplication with~$-(-1)^{(m+1)/2}$ if~$m$ is odd. In other words, it acts by~$-\gamma$. Therefore an element
of~$H^m(A_3, K_\varepsilon)$ is fixed by this action if and only if it is an eigenvector of~$\gamma$ corresponding to the eigenvalue~$-1$, and so the second assertion follows again from Lemma~\ref{ResChar3}. Note that~$K_\varepsilon$ also restricts to the trivial representation of~$A_3$.

\item
To prove the third assertion, we first note that~$\bar{F}^2 \in S$
and~$\bar{E} \bar{F} \in S$, so that $K[\bar{E} \bar{F}, \bar{F}^2] \subset S$. But the calculations in the first step show that
\[S = \bigoplus_{j=0}^\infty H^{4j}(A_3, K)
\, \oplus \, \bigoplus_{j=0}^\infty H^{4j+3}(A_3, K).\]
A basis of~$H^{4j}(A_3, K)$ is~$\bar{F}^{2j}$, and a basis of~$H^{4j+3}(A_3, K)$
is~$(\bar{E} \bar{F}) \bar{F}^{2j}$. So the two sets that appear in the third assertion are actually equal.

\item
It also follows from the calculations in the first step that
\[J = \bigoplus_{j=0}^\infty H^{4j+2}(A_3, K)
\, \oplus \, \bigoplus_{j=0}^\infty H^{4j+1}(A_3, K).\]
A basis of~$H^{4j+2}(A_3, K)$ is~$\bar{F} \bar{F}^{2j}$, and a basis of~$H^{4j+1}(A_3, K)$
is~$\bar{E} \bar{F}^{2j}$. This implies the fourth assertion.
\qedhere
\end{pflist}
\end{proof}

As a slightly surprising consequence, we get that the cohomology rings of~$S_3$ and~$A_3$ are actually isomorphic:
\begin{Cor} \label{CohomRingIsom}
$H(S_3, K) \cong H(A_3, K)$
\end{Cor}
\begin{proof}
We have seen in Lemma~\ref{IsomChar3} that~$H(A_3, K)$ is isomorphic
to~$K[x,y]/(x^2)$, and we have seen in Proposition~\ref{CohomRingSym} that~$H(S_3, K)$ is isomorphic to~$S$. To establish our assertion, it is therefore sufficient to show that~$S$ is isomorphic to~$K[x,y]/(x^2)$. For this, we consider the algebra homomorphism from the polynomial algebra~$K[x,y]$ to~$S$ that maps~$x$ to~$\bar{E} \bar{F}$ and~$y$ to~$\bar{F}^2$. This map is surjective by Proposition~\ref{CohomRingSym}. Since~$x^2$ is mapped
to~$(\bar{E} \bar{F})^2 = \bar{E}^2 \bar{F}^2$, it induces an algebra homomorphism from $K[x,y]/(x^2)$ to~$S$ that is still surjective. But~$K[x,y]/(x^2)$ has a basis consisting of the residue classes of the elements~$y^j$ and~$x y^j$, for~$j \in \N_0$. These elements are mapped to~$\bar{F}^{2j}$ and~$(\bar{E} \bar{F}) \bar{F}^{2j}$, respectively, and we have seen in the proof of Proposition~\ref{CohomRingSym} that these elements form a basis of~$S$. So the induced algebra homomorphism maps a basis to a basis and is therefore bijective.
\end{proof}

For the cohomology ring of~$S_3$, this means that also
$H(S_3, K) \cong K[x,y]/(x^2)$. Under this isomorphism, the class of~$x$ corresponds to an element of cohomological degree~$3$, while the class of~$y$ corresponds to an element of cohomological degree~$4$.

\vspace{2mm}

We have seen in Proposition~\ref{CohomRingSym} that~$J$ is generated as a, say, right $S$-module by~$\bar{E}$ and~$\bar{F}$. There is therefore a surjective $S$-module homomorphism
\[f \colon S \oplus S \to J\]
that satisfies $f(1,0) = \bar{E}$ and $f(0,1) = \bar{F}$, and we have
$J \cong S^2/\ker(f)$. We can determine this kernel explicitly:

\vspace{1mm}

\begin{Proposition} \label{Kern}
The kernel of~$f$ is generated by~$(\bar{E} \bar{F},0)$ and~$(\bar{F}^2,-\bar{E} \bar{F})$ as a right \mbox{$S$-module}.
\end{Proposition}
\begin{proof}
We have~$f(\bar{E} \bar{F},0) = \bar{E} (\bar{E} \bar{F}) = 0$ and
$f(\bar{F}^2,-\bar{E} \bar{F}) = \bar{E} \bar{F}^2 - \bar{F} (\bar{E} \bar{F}) = 0$,
so that these two elements are indeed in the kernel. Now let us assume that there is an element of the kernel that is not in the submodule generated by these two elements.
Since $S = K[\bar{E} \bar{F}, \bar{F}^2]$ by Proposition~\ref{CohomRingSym}, we can subtract from this element some elements from the submodule so that the resulting element has a scalar~$\lambda \in K$ as its first component, and is by hypothesis still not contained in the submodule considered. The fact that this element is in the kernel then implies that~$\lambda \bar{E}$ is a multiple of~$\bar{F}$, which forces both~$\lambda$ and the coefficient of~$\bar{F}$ to be zero. But this means that the element considered is zero, which is a contradiction.
\end{proof}

\pagebreak

We note that, in view of the fact that~$S \cong H(S_3, K)$ via the restriction map, we can consider~$J$ as a right module over~$H(S_3, K)$. This module is isomorphic
to~$H(S_3, K_\varepsilon)$:
\begin{Proposition} \label{IsoMod}
$\Res^{S_3}_{A_3} \colon H(S_3, K_\varepsilon) \to J$ is an isomorphism of~$H(S_3, K)$-modules.
\end{Proposition}
\begin{proof}
We have established the bijectivity in Proposition~\ref{CohomRingSym} above. Linearity over~$H(S_3, K)$ follows from the commutativity of the square
\begin{center}
\begin{tikzcd} {}
H(S_3, K_\varepsilon) \times H(S_3, K) \arrow{r}{}
\arrow[swap]{d}{\Res^{S_3}_{A_3} \times \Res^{S_3}_{A_3}} &
H(S_3, K_\varepsilon) \arrow{d}{\Res^{S_3}_{A_3}} \\
J \times S \arrow{r}{} & J
\end{tikzcd}
\end{center}
in which the horizontal maps are Yoneda products. In other words, the diagram commutes because the Yoneda product is compatible with the restriction map.
\end{proof}

\subsection{The Yoneda module structure in characteristic~$3$} \label{YonChar3}
We now turn to the right module structure of~$H(S_3, M^1)$ over~$H(S_3, K)$ for fields of characteristic~$3$. As in the case of characteristic~$2$ treated in Paragraph~\ref{YonChar2}, the decomposition $M^1 = M^1(O_1) \oplus M^1(O_2) \oplus M^1(O_3)$ from Paragraph~\ref{SymmThree} leads to the decomposition
\[H(S_3, M^1) = H(S_3, M^1(O_1)) \oplus H(S_3, M^1(O_2)) \oplus H(S_3, M^1(O_3))\]
into submodules for the action of~$H(S_3, K)$. In the present case, the structure of these submodules is as follows:
\begin{Theorem} \label{ModStrucChar3}
\begin{enumerate}
\item
$H(S_3, M^1(O_1))$ is free of rank~$1$ over~$H(S_3, K)$. The module action induces an $\SL(2, \Z)$-equivariant isomorphism
\[H^0(S_3, M^1(O_1)) \ot H(S_3, K) \to H(S_3, M^1(O_1)).\]

\item
$H(S_3, M^1(O_2)) = H^0(S_3, M^1(O_2)) \cong K[\P^1_2]$ is a trivial $H(S_3, K)$-module.

\item
$H(S_3, M^1(O_3))$ is the direct sum of a free module of rank~$4$ and~$H(S_3, K_\varepsilon)^4$. We have
\[H(S_3, M^1(O_3)) \cong H(S_3, K) \ot K_{\chi_1}[\P_3^{1}] \oplus H(S_3, K_\varepsilon) \ot K_{\chi_2}[\P_3^{1}]\]
as modules over both~$H(S_3, K)$ and~$\SL(2, \Z)$.
\end{enumerate}
\end{Theorem}
\begin{proof}
\begin{pflist}
\item
The first assertion holds for exactly the same reasons as the corresponding assertion in Theorem~\ref{ModStrucChar2}, because~$M^1(O_1) \cong K$. For the second assertion, we have already seen in the proof of Theorem~\ref{ModGroupChar3} that \mbox{$H^m(S_3, M^1(O_2)) = 0$} if~$m>0$, so that
$H(S_3, M^1(O_2)) = H^0(S_3, M^1(O_2))$. In Paragraph~\ref{DegZero}, we have seen that $H^0(S_3, M^1(O_2)) = M^1(O_2)^{S_3} \cong K[\P^1_2]$ as an $\SL(2, \Z)$-module; this space is three-dimensional over~$K$. To show that the
$H(S_3, K)$-module structure is trivial, we argue as in the proof of Theorem~\ref{ModStrucChar2}: $H(S_3, M^1(O_2))$ is a graded module over~$H(S_3, K)$ that has only zero terms in higher degree, which implies that $H^m(S_3, K)$ must act as zero if~$m>0$. In other words, $H(S_3, K)$ acts by pullback along the augmentation, which is the asserted triviality.

\item
For the third assertion, recall from Paragraph~\ref{SymmThree} or the proof of Theorem~\ref{ModGroupChar3} that
\begin{align*}
M^1(O_3)
\cong K[s + s^{-1}] \ot K_{\chi_1}[\P_3^{1}] \oplus K[s - s^{-1}] \ot K_{\chi_2}[\P_3^{1}].
\end{align*}
If~$K_\varepsilon$ denotes the sign representation of~$S_3$, as in Paragraph~\ref{YonAlgChar3}, this means that $M^1(O_3) \cong K^4 \oplus K_\varepsilon^4$ as $S_3$-modules, so that
\[H(S_3, M^1(O_3)) \cong H(S_3, K)^4 \oplus H(S_3, K_\varepsilon)^4\]
as modules over~$H(S_3, K)$. The first summand is clearly free of rank~$4$ over~$H(S_3, K)$. Passing back to tensor products, we can write this isomorphism in the form
\[H(S_3, M^1(O_3)) \cong H(S_3, K) \ot K_{\chi_1}[\P_3^{1}] \oplus H(S_3, K_\varepsilon) \ot K_{\chi_2}[\P_3^{1}]\]
in which it is not only $H(S_3, K)$-linear, but also $\SL(2, \Z)$-equivariant, where on the right the $\SL(2, \Z)$-action is on the second tensor factor and the $H(S_3, K)$-action is on the first tensor factor. In this way, we have essentially deduced a simple version of the universal coefficient theorem from the additivity of cohomology, as in the proofs of Theorem~\ref{ThmDecomp} or Corollary~\ref{DecompCohEig}.
\qedhere
\end{pflist}
\end{proof}

This theorem shows in particular that~$H(S_3, M^1)$ is not generated in degree~$0$, as the nonzero terms of its summand
\[H(S_3, K_\varepsilon)^4 \cong H(S_3, K_\varepsilon) \ot K_{\chi_2}[\P_3^{1}]\] appear only in degrees~$m$ that satisfy either \mbox{$m \equiv 1 \pmod{4}$}
or~$m \equiv 2 \pmod{4}$, as we know from Theorem~\ref{ModGroupChar3} and its proof. In view of the discussion at the end of Paragraph~\ref{Char3}, this means that the submodule of~$H(S_3, M^1)$ that is generated by~$H^0(S_3, M^1)$ over the Yoneda algebra~$H^0(S_3, K)$ does not contain summands that are isomorphic to~$K_{\chi_2}[\P_3^{1}]$. In other words, the mapping class group representations that appear in higher degree are in general not only different from those that appear in degree zero, but they can also not be generated from those in degree zero by taking Yoneda products with~$\TFT^\infty(\Sigma_0) = H(S_3, K)$.

\subsection{The torus with one insertion} \label{TorIns}
In view of Proposition~\ref{GroupCohom} and the discussion following it, Theorem~\ref{ModStrucChar3} above also describes the right module structure of~$\Ext_D(K, L)$ over the algebra~$\Ext_D(K, K)$, where, as before, $D=D(S_3)$ is the Drinfel'd double of~$S_3$ and~$L$ is the coend in the category of finite-dimensional $D$-modules discussed in Paragraph~\ref{EndomCoend}. As we just saw, the submodule generated by~$\Ext^0_D(K, L)$ misses precisely the summand~$H(S_3, K_\varepsilon)^4$. However, also~$H(S_3, K_\varepsilon)^4$ can be generated in degree zero using Yoneda products, but in a different way: Remaining in characteristic~$3$, we consider the $D$-module~\mbox{$X \deq I(e, K_\varepsilon)$} from Proposition~\ref{PropRepDrinf} in Paragraph~\ref{RepDrinf}. If we use the module~$X$ instead of the trivial module~$K$, we can generate~$H(S_3, K_\varepsilon)^4$ as follows:
\begin{Proposition} \label{GenX}
The map
\[\Ext^0_D(X, L) \ot \Ext_D(K, X) \to \Ext_D(K, L)\]
given by the Yoneda product is an injection with image isomorphic to~$H(S_3, K_\varepsilon)^4$.
\end{Proposition}
\begin{proof}
\begin{pflist}
\item
It is immediate from Corollary~\ref{TrivConjClass} that~$X_e = K_\varepsilon$, and in view of the functor adjunction described after that corollary, we have
\[\Ext^0_D(X, L) = \Hom_D(X, L) \cong \Hom_{S_3}(K_\varepsilon, L_e)
\cong \Hom_{S_3}(K_\varepsilon, M^1),\]
where the last isomorphism results from Theorem~\ref{ThmModTilde} and Proposition~\ref{GLin}. In view of the discussion following Proposition~\ref{GroupCohom}, we therefore have to show that the map
\[\Hom_{S_3}(K_\varepsilon, M^1) \ot H(S _3, K_\varepsilon) \to H(S_3, M^1) \]
given by the Yoneda product is an injection with image~$H(S_3, K_\varepsilon)^4$. But in the Yoneda picture that we discussed in Paragraph~\ref{YonProd}, this is the map
\[\Hom_{S_3}(K_\varepsilon, M^1) \ot \Yext_{S_3}(K, K_\varepsilon) \to
\Yext_{S_3}(K, M^1) \]
that arises, as in the proof of Theorem~\ref{ModStrucChar2}, from the functoriality of~$\Yext$, i.e., by applying an element of
$\Hom_{S_3}(K_\varepsilon, M^1)$ to an extension in $\Yext_{S_3}(K, K_\varepsilon)$.

\item
We have seen in Paragraph~\ref{SymmThree} that
$M^1 = M^1(O_1) \oplus M^1(O_2) \oplus M^1(O_3)$,
and we have also seen there that~$M^1(O_1) \cong K$,
that
\[M^1(O_2) \cong K[I_2] \otimes_{K} K[\Omega^{1,2}] \cong K[I_2]^3\]
and that~$M^1(O_3) \cong
K[s + s^{-1}] \ot K_{\chi_1}[\P_3^{1}] \oplus K[s - s^{-1}] \ot K_{\chi_2}[\P_3^{1}]
\cong K^4 \oplus K_\varepsilon^4$ as $S_3$-modules. Because~$K$ has characteristic~$3$, the unit element~$1_K$ is the only third root of unity, and therefore the action of~$s$ on~$K[I_2]$ can only have~$1_K$ as a possible eigenvalue. From this fact, it is not difficult to see that~$t_1 + t_2 + t_3$ spans the only one-dimensional $S_3$-submodule of~$K[I_2]$, so that in particular there is no submodule isomorphic to~$K_\varepsilon$. This implies that
\begin{align*}
&\Hom_{S_3}(K_\varepsilon, M^1) \cong \\
&\Hom_{S_3}(K_\varepsilon, M^1(O_1)) \oplus \Hom_{S_3}(K_\varepsilon, M^1(O_2))
\oplus \Hom_{S_3}(K_\varepsilon, M^1(O_3)) \cong \\
&\Hom_{S_3}(K_\varepsilon, M^1(O_3)) \cong
\Hom_{S_3}(K_\varepsilon, K^4 \oplus K_\varepsilon^4) \cong
\Hom_{S_3}(K_\varepsilon, K_\varepsilon^4).
\end{align*}
From this chain of isomorphisms, we see that the image of the map
\[\Hom_{S_3}(K_\varepsilon, M^1) \ot H(S _3, K_\varepsilon) \to H(S_3, M^1) \]
is precisely the summand~$H(S_3, K_\varepsilon)^4$ of $H(S_3, M^1(O_3))$ that we found in Theorem~\ref{ModStrucChar3}.
\qedhere
\end{pflist}
\end{proof}

We note that the map that appears in this proposition is equivariant for the action of the mapping class group, where the action of the mapping class group on the left-hand side is defined by acting on the first factor. This follows from the fact that the mapping class group acts in the present case by postcomposition, as was pointed out at the end of Paragraph~2.6. Alternatively, as the mapping class group is here the modular group~$\SL(2 ,\Z)$, this also follows from the explicit form of the action of the modular group that we have used throughout and just reviewed in the preceding proof. The map is also equivariant for the right action of the Yoneda algebra~$\Ext_D(K,K)$, where the action of the Yoneda algebra on the left-hand side is defined by acting on the second factor. This holds for the same reasons as those given in the proof of Proposition~\ref{MapClassYon}.

From the point of view of derived topological field theory, the mapping that appears in the proposition above can be interpreted as obtained from the gluing of two surfaces. To discuss this, we need the following lemma, in which we assume that~$\CC$ is a category satisfying the hypotheses from Paragraph~2.1:
\begin{Lemma} \label{ExtAdj}
Suppose that~$X$, $Y$, and~$Z$ are objects in~$\CC$. Then there is, for every~$m \in \N_0$, an isomorphism
\[\Ext^m(X^* \ot Y, Z) \cong \Ext^m(Y, X \ot Z)\]
that is natural in~$X$, $Y$, and~$Z$.
\end{Lemma}
\begin{proof}
Since tensoring is exact and preserves projectives, as we mentioned in Paragraph~2.1, this lemma can be proved exactly as in~\cite[Rem.~9.3.2, p.~191]{W2}, which treats the case of Hopf algebras. The lemma can also be deduced from a generalization of~\cite[Chap.~IV, Sec.~12, Eq.~(12.1), p.~163]{HiS} to categories.
\end{proof}

With the help of this lemma, the map that appears in Proposition~\ref{GenX} above takes in degree~$m$ the form
\[\Ext^0_D(X, L) \ot \Ext^m_D(X^*, K) \to \Ext^m_D(K, L).\]
According to the definitions in Paragraph~3.2, the space $\Ext^0_D(X, L)$ that appears on the left-hand side is equal to~$\TFT^0(\Sigma_{1,1}^X)$, the block space associated with the torus with one boundary component labeled by~$X$. The space~$\Ext^m_D(K, L)$ that appears on the right-hand side is equal to~$\TFT^m(\Sigma_{1,0})$, the derived block space associated with the torus without boundary. The derived block space~$\TFT^m(\Sigma_{0,1}^{X^*})$ associated with the sphere with one boundary component labeled by~$X^*$, or equivalently a disk whose boundary circle is labeled by~$X^*$, is~$\Ext^m_D(X^*, K)$, so that we get a map
\[\TFT^0(\Sigma_{1,1}^X) \ot \TFT^m(\Sigma_{0,1}^{X^*}) \to \TFT^m(\Sigma_{1,0}).\]
Although we have not constructed a derived modular functor here, but only its corresponding mapping class group representations, a derived modular functor would assign this map to the gluing of a disk into the boundary component of a torus to obtain a torus without boundary.

\newpage

\appendix
\section{Appendix} \label{Sec:App}
\subsection{The periodic resolution for the cyclic group} \label{PerResCycl}
One of our basic computational tools is the periodic resolution for the cyclic group, and it is the purpose of this paragraph in the appendix to collect its basic properties to the extent that we need them in the main text. So suppose that~$U = \langle t \rangle \cong \Z_n$ is a finite cyclic group of order~$n$ with generator~$t$. Then there is a periodic free resolution
\[\cdots \xlongrightarrow{d_4} P_3 \xlongrightarrow{d_3} P_2 \xlongrightarrow{d_2} P_1 \xlongrightarrow{d_1} P_0 \xlongrightarrow{\varepsilon} K \to 0\]
of the base field~$K$ in which all modules~$P_m$ are equal to the group algebra~$K[U]$. To describe the boundary operators, we introduce two elements in the group algebra~$K[U]$, namely the norm
$N \deq 1 + t + \dots + t^{n-1}$ and the difference $D \deq t-1$, and define
$d_m(x) = Dx$ if~$m$ is odd and $d_m(x) = Nx$  if~$m$ is even. The augmentation map $\varepsilon$ is the usual counit in a group algebra, defined as
$\varepsilon(\sum_{i=0}^{n-1} \alpha_i t^i) = \sum_{i=0}^{n-1}\alpha_i$. This resolution is discussed in many textbooks on homological algebra, for example in~\cite[Chap.~IV, \S~7, p.~121f]{ML1}. If~$N.$ and~$D.$ denote the action by~$N$ and~$D$, respectively, we get the following result for the cohomology groups:
\begin{Proposition} \label{CohomCycl}
For a $K[U]$-module~$M$, we have
\[H^m(U,M) \cong
\begin{cases}
M^{U} &:  m = 0\\
M^{U}/N.M &: m \text{ even}, m \neq 0\\
\ker(N.)/D.M &: m \text{ odd}
\end{cases}\]
\end{Proposition}
In the case that the module is trivial, the action of~$N$ is multiplication by~$n$, while the action of~$D$ is zero. If the characteristic of the base field divides~$n$, the action of~$N$ is also zero, and the preceding result yields $H^m(U,M) = M$ in all cases. If the characteristic of the base field does not divide~$n$, the action of~$N$ is an isomorphism, and we obtain $H^0(U,M) = M$ and $H^m(U,M) = 0$ if~$m>0$.

The computation that shows that the above periodic resolution is exact is precisely what is needed to deduce the following corollary from this proposition:
\begin{Cor} \label{CohomReg}
For $m>0$, we have $H^m(U,K[U]) = 0$.
\end{Cor}
This equation holds in fact not only for cyclic groups, but for any finite group: The group algebra~$K[U]$ is a Frobenius algebra (cf.~\cite[Examp.~16.56, p.~442]{La}), and therefore a quasi-Frobenius algebra. Thus it is injective as a module over itself (cf.~\cite[Thm.~15.1, p.~409, Thm.~15.9, p.~412]{La}). As a consequence, we get
\[H^m(U,K[U]) = \Ext_{K[U]}^m(K,K[U]) = 0\]
(cf.~\cite[Chap.~IV, Prop.~7.2, p.~140]{HiS}).

As also mentioned in Paragraph~\ref{CyclSub}, every automorphism~$f \colon U \to U$ has the property that~$f(t) = t^k$ for a number~$k$ that is relatively prime to~$n$. Group cohomology assigns to~$f$ a map
\[H^m(f) \colon H^m(U,M) \to H^m(U,M')\]
where~$M'$ is the $U$-module with the same underlying set as~$M$, but the new module structure \mbox{$x \scdot y \deq f(x).y$} given by pullback along~$f$, where the low dot denotes the old module structure (cf.~\cite[Sec.~9.5.1, p.~564f]{Ro}). We want to describe this map from the perspective of our periodic resolution. To this end, we define $f_m \colon P_m \to P'_m$ as
\[f_m(x) \deq
\begin{cases}
k^{\tfrac{m}{2}} f(x) &: m \text{ even}\\
k^{\tfrac{m-1}{2}} N_k f(x) &: m \text{ odd}
\end{cases}\]
where $N_k \deq 1 + t + \dots + t^{k-1}$ is the truncated norm. We claim that~$(f_m)$ constitutes a chain map between the two projective resolutions:
\begin{Lemma}
The diagram
\begin{center}
\begin{tikzcd}
\cdots \arrow{r}{d_4} & P_3 \arrow{r}{d_3} \arrow{d}{f_3} & P_2 \arrow{r}{d_2} \arrow{d}{f_2} & P_1 \arrow{r}{d_1} \arrow{d}{f_1} & P_0 \arrow{r}{\varepsilon} \arrow{d}{f_0}
& K \arrow[equal]{d} \\
\cdots \arrow{r}{d_4} & P'_3 \arrow{r}{d_3} & P'_2 \arrow{r}{d_2} & P'_1 \arrow{r}{d_1} & P'_0 \arrow{r}{\varepsilon} & K
\end{tikzcd}
\end{center}
is commutative.
\end{Lemma}
\begin{proof}
We first note that the maps~$f_m$ are indeed module homomorphisms: If~$m$ is even, this holds because we have
\begin{align*}
f_m(xy) = k^{\tfrac{m}{2}} f(xy) = k^{\tfrac{m}{2}} f(x) f(y) = x \scdot f_m(y)
\end{align*}
and if~$m$ is odd, it holds similarly because
\begin{align*}
f_m(xy) = k^{\tfrac{m-1}{2}} N_k f(xy) = k^{\tfrac{m-1}{2}} N_k f(x) f(y)
= f(x).f_m(y) = x \scdot f_m(y),
\end{align*}
where we have used that in both equations the unmodified module structure is just given by multiplication.

To show the commutativity of the rightmost square, we note that both~$f_0 = f$ and~$\varepsilon$ are algebra homomorphisms, so that it suffices to verify the claim for the generator~$t$, for which it holds since
\mbox{$\varepsilon(f(t)) = \varepsilon(t^k) = 1 = \varepsilon(t)$}.
For the case where~$m$ is odd, we note that $D N_k = t^k - 1 = f(D)$ and therefore
\begin{align*}
d_m(f_m(x)) = k^{\tfrac{m-1}{2}} D N_k f(x) = k^{\tfrac{m-1}{2}} f(D) f(x)
= k^{\tfrac{m-1}{2}} f(Dx) = f_{m-1}(d_m(x)).
\end{align*}
For the case where~$m \geq 2$ is even, we note that $f(N) = N$, because both sides are the sum over all elements of the cyclic group, and that $N_k N = k N$. Therefore, we have
\begin{align*}
d_m(f_m(x)) = k^{\tfrac{m}{2}} N f(x) = k^{\tfrac{m-2}{2}} N_k N f(x)
= k^{\tfrac{m-2}{2}} N_k f(Nx) = f_{m-1}(d_m(x)).
\end{align*}
This establishes our claim.
\end{proof}

To compute the induced map in cohomology, we first note that~$M^U = M'^U$. Furthermore, the equation~$f(N) = N$ already recorded in the preceding proof implies that~\mbox{$N.M = N \scdot M= N.M'$}. Finally, in~$M/D.M$, both~$t$ and~$t^k$ act as the identity. This means that under the isomorphism described in Proposition~\ref{CohomCycl} above, the spaces~$H^m(U,M)$ and~$H^m(U,M')$ are mapped to the same space in all cases. As we will show now, $H^m(f)$ is on this space equal to the multiplication by a scalar. If we denote multiplication by a scalar by the scalar itself, $H^m(f)$ has the following form:
\begin{Proposition} \label{Hmf}
\[H^m(f) =
\begin{cases}
k^{\tfrac{m}{2}} &: m \text{ even}\\
k^{\tfrac{m+1}{2}}  &: m \text{ odd}
\end{cases}\]
\end{Proposition}
\begin{proof}
To compute the induced map in cohomology, we must compute the composition
\begin{center}
\begin{tikzcd}
\cdots & \Hom_U(P_2, M) \arrow{l}{d_3^*} \arrow{d}{\id} & \Hom_U(P_1, M)\arrow{l}{d_2^*} \arrow{d}{\id} & \Hom_U(P_0, M) \arrow{l}{d_1^*} \arrow{d}{\id} \\
\cdots & \Hom_U(P'_2, M') \arrow{l}{d_3^*} \arrow{d}{f_2^*} & \Hom_U(P'_1, M') \arrow{l}{d_2^*} \arrow{d}{f_1^*} & \Hom_U(P'_0, M') \arrow{l}{d_1^*} \arrow{d}{f_0^*}\\
\cdots & \Hom_U(P_2, M') \arrow{l}{d_3^*} & \Hom_U(P_1, M') \arrow{l}{d_2^*} & \Hom_U(P_0, M') \arrow{l}{d_1^*}
\end{tikzcd}
\end{center}
in which~$f_m^*$ denotes precomposition with~$f_m$. If we define $g_m \colon M \to M'$ as
\[g_m(x) \deq
\begin{cases}
k^{\tfrac{m}{2}} x &: m \text{ even}\\
k^{\tfrac{m-1}{2}} N_k.x &: m \text{ odd}
\end{cases}\]
where, as before, the low dot denotes the original module action in~$M$, we have that the diagram
\begin{center}
\begin{tikzcd}
\Hom_U(P_m, M) \arrow{r}{} \arrow{d}{f_m^*} & M \arrow{d}{g_m} \\
\Hom_U(P_m, M') \arrow{r}{} & M'
\end{tikzcd}
\end{center}
commutes if the horizontal isomorphisms are given by evaluation at~$1$. The preceding diagram then becomes
\begin{center}
\begin{tikzcd}
\cdots & M \arrow{l}{D.} \arrow{d}{g_4} & M \arrow{l}{N.} \arrow{d}{g_3} &
M \arrow{l}{D.} \arrow{d}{g_2} & M \arrow{l}{N.} \arrow{d}{g_1} &
M \arrow{l}{D.} \arrow{d}{g_0}\\
\cdots & M' \arrow{l}{D.} & M' \arrow{l}{N.} & M' \arrow{l}{D.} & M' \arrow{l}{N.} & M' \arrow{l}{D.}
\end{tikzcd}
\end{center}
where~$N.$ and~$D.$ are now given by the action of~$N$ and~$D$ on~$M$ and~$M'$, respectively, which means for the lower row that~$D.$ is the original action of~$f(D)$. This implies, in our description of the cohomology, that the corresponding map is induced by~$g_m$. But for~$m$ even,~$g_m$ is the multiplication by~$k^{m/2}$ anyway, and because~$t$ acts trivially on~$\ker(N.)/D.M$, the map~$g_m$ induces the multiplication by~$k^{(m+1)/2}$ on this quotient if~$m$ is odd.
\end{proof}

We note that a special case of the preceding proposition can be found in~\cite[Chap.~V, Sec.~3, Exerc.~4, p.~114]{Br}; however, the proof outlined there does not generalize directly. A similar exercise can be found in~\cite[Chap.~IV, \S~7, Exerc.~6, p.~124]{ML1}, and special cases are also stated in~\cite[Lem.~3, p.~343]{S1} and~\cite[Prop.~8.1, p.~283]{S2}.

Suppose now that our cyclic group~$U$ is a normal subgroup of a finite group~$G$. For $a \in G$, conjugation then induces the automorphism $f \colon U \to U, b \mapsto a^{-1} b a$ that maps~$t$ to~$t^k$ for a suitable number~$k$. If~$M$ is a $K[G]$-module, it is also a $K[U]$-module by restriction, and the map~$\alpha \colon M' \to M,~x \mapsto a.x$ is $K[U]$-linear, as we have
\[\alpha(b \scdot x) = \alpha(f(b).x) = (a (a^{-1} b a)).x = (ba).x = b.\alpha(x)\]
for $b \in U$ and~$x \in M$. In the terminology of~\cite[loc.~cit.]{Ro}, this means that~$(f,\alpha)$ is a cocompatible pair. It induces the following map in cohomology:
\begin{Proposition} \label{CohomCyclSubgr}
We have
\[H^m(f,\alpha)(\overline{x}) =
\begin{cases}
k^{\tfrac{m}{2}} \,  \overline{a.x} &: m \text{ even}, m \neq 0\\
k^{\tfrac{m+1}{2}}  \overline{a.x} &: m \text{ odd}
\end{cases}\]
For $m=0$, we have $H^0(f,\alpha)(x) = a.x$ for~$x \in M^U$.
\end{Proposition}
\begin{proof}
With the notation used in the preceding proof, we have the cochain map
\begin{center}
\begin{tikzcd}
\Hom_U(P_m, M) \arrow{r}{\id} & \Hom_U(P'_m, M') \arrow{r}{f_m^*} &
\Hom_U(P_m, M') \arrow{r}{\alpha_*} & \Hom_U(P_m, M)
\end{tikzcd}
\end{center}
After evaluating at the unit element, this translates into the cochain map
\begin{center}
\begin{tikzcd}
\cdots & M \arrow{l}{D.} \arrow{d}{\alpha \circ g_4}
& M \arrow{l}{N.} \arrow{d}{\alpha \circ g_3}
& M \arrow{l}{D.} \arrow{d}{\alpha \circ g_2}
& M \arrow{l}{N.} \arrow{d}{\alpha \circ g_1}
& M \arrow{l}{D.} \arrow{d}{\alpha \circ g_0}\\
\cdots & M \arrow{l}{D.} & M \arrow{l}{N.} & M \arrow{l}{D.} & M \arrow{l}{N.} & M \arrow{l}{D.}
\end{tikzcd}
\end{center}
which induces the map~$H^m(f, \alpha) \colon H^m(U, M) \to H^m(U, M)$ in cohomology. Our assertion therefore follows from Proposition~\ref{Hmf}.
\end{proof}

\subsection{Acyclic assembly} \label{AcyclAss}
In the remaining parts of this appendix, we provide some foundational material for homological algebra in arbitrary categories that is necessary to connect the Yoneda products discussed in Paragraph~\ref{YonProd} with the more frequently used cup products. For most of this material, we are unfortunately not aware of suitable references in the literature.

In this paragraph, we consider an arbitrary abelian category~$\CC$. A double complex in~$\CC$ is a family of objects~$(X_{ij})_{i,j \in \Z}$ together with differentials
$d^1_{ij} \colon X_{ij} \to X_{i-1,j}$, called the horizontal differentials, and
$d^2_{ij} \colon X_{ij} \to X_{i,j-1}$, called the vertical differentials, so that
\[d^1_{i-1,j} \circ d^1_{ij} = 0 \qquad d^2_{i,j-1} \circ d^2_{ij} = 0
\qquad d^1_{i,j-1} \circ d^2_{ij} + d^2_{i-1,j} \circ d^1_{ij} = 0\]
for all $i,j \in \Z$ (cf.~\cite[Examp.~1.2.4, p.~7]{We}). We consider here only first-quadrant double complexes, which by definition satisfy $X_{ij} = 0$ if~$i < 0$ or~$j < 0$. For these, the two notions of total complexes (cf.~\cite[Examp.~1.2.6, p.~8]{We}) coincide, and this (unique) total complex is defined as the chain complex
\[T_n \deq \bigoplus_{\substack{i,j=0 \\i+j=n}}^n X_{ij} =
\bigoplus_{i=0}^n X_{i,n-i},\]
whose differential~$d_n$ is the unique morphism that makes the diagram
\begin{center}
\begin{tikzcd} [column sep = 60mm]{}
T_n = \bigoplus\limits_{k=0}^n X_{k,n-k} \arrow{d}{}\arrow{r}{d_n}
&T_{n-1} = \bigoplus\limits_{k=0}^{n-1} X_{k,n-1-k} \arrow{d}{}  \\
X_{i,n-i} \oplus X_{i+1,n-1-i}
\arrow[swap]{r}{d^2_{i,n-i} + d^1_{i+1,n-1-i}}& X_{i,n-1-i}
\end{tikzcd}
\end{center}
commutative, where $i=0,\dots,n-1$ and the vertical morphisms arise from the canonical projections to the components. In other words, we define the morphism
\mbox{$d_n \colon T_n \to T_{n-1}$} by specifying what its composition with each canonical projection is. By the universal property of a product, this is sufficient, and the finite direct sums that appear above are indeed direct products (cf.~\cite[Chap.~II, Prop.~9.1, p.~76]{HiS}).

In this situation, we have the following categorical version of the acyclic assembly lemma:
\begin{Lemma} \label{AcyclAssLem}
Suppose that, for all $i \in \Z$, the vertical complex $(X_{ij})_{j \in \Z}$ with differential~$d^2$ is exact. Then the total complex~$(T_n)_{n \in \Z}$ is exact.
\end{Lemma}
\begin{proof}
\begin{pflist}
\item
We define the auxiliary objects $Y_{m,n} \deq \bigoplus_{k=0}^m X_{k,n-k}$, so that we have
$Y_{n,n} = T_n$ and even $Y_{m,n} = T_n$ for~$m \ge n$. As in the case of~$d_n$ above, we use the universal property of a product to define the morphisms \mbox{$d_{m,n} \colon Y_{m,n} \to Y_{m,n-1}$} via the commutative diagrams
\begin{center}
\begin{tikzcd} [column sep = 60mm]{}
Y_{m,n} = \bigoplus\limits_{k=0}^m X_{k,n-k} \arrow{d}{}\arrow{r}{d_{m,n}}
&Y_{m,n-1} = \bigoplus\limits_{k=0}^m X_{k,n-1-k} \arrow{d}{}  \\
X_{i,n-i} \oplus X_{i+1,n-1-i}
\arrow[swap]{r}{d^2_{i,n-i} + d^1_{i+1,n-1-i}}& X_{i,n-1-i}
\end{tikzcd}
\end{center}
for $i = 0, 1, \dots, m - 1$ and
\begin{center}
\begin{tikzcd} [column sep = 60mm]{}
Y_{m,n} = \bigoplus\limits_{k=0}^m X_{k,n-k} \arrow{d}{} \arrow{r}{d_{m,n}}
&Y_{m,n-1} = \bigoplus\limits_{k=0}^m X_{k,n-1-k} \arrow{d}{}  \\
X_{m,n-m}
\arrow[swap]{r}{d^2_{m,n-m}}& X_{m,n-1-m}
\end{tikzcd}
\end{center}
for $i=m$, so that $d_{n,n} = d_n$ and even $d_{m,n} = d_n$ for~$m \ge n$. To establish our assertion, it will therefore be sufficient to show that the sequence
\begin{align*}
Y_{m,n+1} \xrightarrow{\; \; d_{m,n+1} \; \;} Y_{m,n}
\xrightarrow{\; \; d_{m,n} \; \;} Y_{m,n-1}
\end{align*}
is exact, which we prove by induction on~$m$. For $m=0$, we have
$Y_{0,n} = X_{0,n}$ and $d_{0,n} = d^2_{0,n}$, so that the base case holds by hypothesis.

\item
For the inductive step from~$m$ to~$m+1$, we first note that we have
\[Y_{m,n} = Y_{m-1,n} \oplus X_{m,n-m}.\]
As mentioned above, finite direct sums and finite direct products coincide,
and therefore we have canonical inclusions
$\iota^1_{m,n} \colon Y_{m-1,n} \to Y_{m,n}$ and
$\iota^2_{m,n} \colon X_{m,n-m} \to Y_{m,n}$
as well as canonical projections
\[\pi^1_{m,n} \colon Y_{m,n} \to Y_{m-1,n} \qquad \text{and} \qquad
\pi^2_{m,n} \colon Y_{m,n} \to X_{m,n-m}.\]
With the help of these inclusions and projections, morphisms between finite direct sums can be represented by matrices. In the case of
$d_{m,n} \colon Y_{m,n} \to Y_{m,n-1}$, this matrix is
\[\begin{pmatrix}
\pi^1_{m,n-1} \circ d_{m,n} \circ \iota^1_{m,n}
& \pi^1_{m,n-1} \circ d_{m,n} \circ \iota^2_{m,n} \\
\pi^2_{m,n-1} \circ d_{m,n} \circ \iota^1_{m,n}
& \pi^2_{m,n-1} \circ d_{m,n} \circ \iota^2_{m,n}
\end{pmatrix} =
\begin{pmatrix}
d_{m-1,n} & \iota^2_{m-1,n-1} \circ d^1_{m,n-m} \\
0 & d^2_{m,n-m}
\end{pmatrix}.
\]
From this description, we see in particular that
\begin{align} \tag{I} \label{Id1}
d_{m,n} \circ \iota^2_{m,n} \circ d^1_{m+1,n-m} =
\iota^2_{m,n-1} \circ d^2_{m,n-m} \circ d^1_{m+1,n-m}.
\end{align}
The fact that $d_{m,n-1} \circ d_{m,n} = 0$ then follows inductively from the computation
\begin{align*}
&\begin{pmatrix}
d_{m-1,n-1} & \iota^2_{m-1,n-2} \circ d^1_{m,n-1-m} \\
0 & d^2_{m,n-1-m}
\end{pmatrix}
\begin{pmatrix}
d_{m-1,n} & \iota^2_{m-1,n-1} \circ d^1_{m,n-m} \\
0 & d^2_{m,n-m}
\end{pmatrix} \\
&= \begin{pmatrix}
0 & d_{m-1,n-1} \circ \iota^2_{m-1,n-1} \circ d^1_{m,n-m} + \iota^2_{m-1,n-2} \circ d^1_{m,n-1-m} \circ d^2_{m,n-m} \\
0 & 0
\end{pmatrix}
\end{align*}
and Identity~(\ref{Id1}) just derived.

\item
We now denote the kernel of~$d^2_{m,n}$ by~$K^2_{m,n}$ and consider the diagram
\begin{center}
\begin{tikzcd} [column sep = 15mm]{}
Y_{m,n+1} \oplus X_{m+1,n-m} \arrow{r}{d_{m+1,n+1}}
&Y_{m,n} \oplus X_{m+1,n-m-1} \arrow{r}{d_{m+1,n}}
&Y_{m,n-1} \oplus X_{m+1,n-m-2} \\
Y_{m,n+1} \oplus X_{m+1,n-m} \arrow{r}{d_{m+1,n+1}} \arrow{u}{\id}
&Y_{m,n} \oplus K^2_{m+1,n-m-1}
\arrow{r}{\pi^1_{m+1,n-1} \circ \, d_{m+1,n}} \arrow{u}{}
&Y_{m,n-1} \arrow[swap]{u}{\iota^1_{m+1,n-1}} \\
Y_{m,n+1} \oplus X_{m+1,n-m}
\arrow{r}{f_{m,n}}
\arrow{u}{\id}
&Y_{m,n} \oplus X_{m+1,n-m}
\arrow{r}{g_{m,n}}
\arrow{u}{\id \oplus \, d^2_{m+1,n-m}}
&Y_{m,n-1} \arrow[swap]{u}{\id} \\
Y_{m,n+1} \oplus X_{m+1,n-m}
\arrow[swap]{r}{d_{m,n+1} \oplus \, \id}
\arrow{u}{\id}
&Y_{m,n} \oplus X_{m+1,n-m} \arrow[swap]{r}{d_{m,n} \circ \pi^1}
\arrow[swap]{u}{h_{m,n}}
&Y_{m,n-1} \arrow[swap]{u}{\id}
\end{tikzcd}
\end{center}
where~$\pi^1 \colon Y_{m,n} \oplus X_{m+1,n-m} \to Y_{m,n}$ is the projection to the first component. The other morphisms that appear in the diagram can be described via matrices as follows: The morphism $f_{m,n}$ is defined via the
$2 \times 2$-matrix
\[f_{m,n} \deq \begin{pmatrix}
d_{m,n+1} & \iota^2_{m,n} \circ d^1_{m+1,n-m} \\
0 & \id
\end{pmatrix},\]
while~$g_{m,n}$ is defined via the $1 \times 2$-matrix
\[g_{m,n} \deq \begin{pmatrix}
d_{m,n} & \iota^2_{m,n-1} \circ d^1_{m+1,n-m-1} \circ d^2_{m+1,n-m}
\end{pmatrix},
\]
and finally $h_{m,n}$ via the $2 \times 2$-matrix
\[h_{m,n} \deq \begin{pmatrix}
\id & \iota^2_{m,n} \circ d^1_{m+1,n-m} \\
0 & \id
\end{pmatrix}.\]
We then have $f_{m,n} = h_{m,n} \circ (d_{m,n+1} \oplus \id)$ and
\mbox{$d_{m+1,n+1} = (\id \oplus \, d^2_{m+1,n-m}) \circ f_{m,n}$}, so that the lower left and the middle left square commute. Furthermore, we have
\[d_{m+1,n} \circ (\id \oplus \, d^2_{m+1,n-m})  =
\begin{pmatrix}
d_{m,n} & \iota^2_{m,n-1} \circ d^1_{m+1,n-m-1} \circ d^2_{m+1,n-m}\\
0 & 0
\end{pmatrix},\]
which shows that the right middle square commutes. Identity~(\ref{Id1}) above implies that $g_{m,n} \circ h_{m,n} = d_{m,n} \circ \pi^1$, so that the lower right square commutes. The commutativity of the two upper squares follows relatively directly from the definitions, so that the entire diagram commutes. Note that these identities imply that
\[g_{m,n} \circ f_{m,n} = g_{m,n} \circ h_{m,n} \circ (d_{m,n+1} \oplus \id)
= d_{m,n} \circ \pi^1 \circ (d_{m,n+1} \oplus \id) = 0,\]
so that the composition in the third row, and in fact in all rows, gives zero.

\item
In this diagram, the lower row is exact (in the middle), because it is the direct sum of the exact sequences $X_{m+1,n-m} \to X_{m+1,n-m} \to 0$ and
\begin{align*}
Y_{m,n+1} \xrightarrow{\; \; d_{m,n+1} \; \;}  Y_{m,n}
\xrightarrow{\; \; d_{m,n} \; \;} Y_{m,n-1}
\end{align*}
that is exact by the induction hypothesis. Since~$h_{m,n}$ is invertible with
\[h_{m,n}^{-1} = \begin{pmatrix}
\id & - \iota^2_{m,n} \circ d^1_{m+1,n-m} \\
0 & \id
\end{pmatrix},\]
the third row is also exact.

In the right middle square, we have that $\id \oplus \, d^2_{m+1,n-m}$ is an epimorphism, because $d^2_{m+1,n-m} \colon X_{m+1,n-m} \to K^2_{m+1,n-m-1}$ is an epimorphism by the exactness of the vertical complex. The weak four lemma (cf.~\cite[Chap.~XII, \S~3, Lem.~3.1 (ii), p.~364]{ML1}), applied with two zero objects on the left, now implies that
$\id \oplus \, d^2_{m+1,n-m}$ induces an epimorphism between the kernels of the two horizontal morphisms in this square. But because the third row is exact and the left middle square commutes, this implies that the second row is exact. Since
\[d_{m+1,n} =
\begin{pmatrix}
d_{m,n} & \iota^2_{m,n-1} \circ d^1_{m+1,n-m-1} \\
0 & d^2_{m+1,n-m-1}
\end{pmatrix},
\]
the two horizontal morphisms in the upper right square have isomorphic kernels, which implies that the first, top row is exact. But this is precisely the exactness of the sequence
\begin{align*}
Y_{m+1,n+1} \xrightarrow{\; \; d_{m+1,n+1} \; \;} Y_{m+1,n}
\xrightarrow{\; \; d_{m+1,n} \; \;}  Y_{m+1,n-1}
\end{align*}
that we wanted to establish.
\qedhere
\end{pflist}
\end{proof}

In the case where~$\CC$ is the category of modules over a ring, the preceding lemma can be found in~\cite[Lem.~2.7.3, p.~59]{We} in an even more general form. In fact, with the help of the Freyd-Mitchell embedding theorem (cf.~\cite[Chap.~VI, Thm.~7.2, p.~151]{Mi}) or similar embedding theorems, the result above can be deduced from its special case for a category of modules. We have given the proof here in order to have a relatively short, self-contained treatment of the question that does not rely on embedding theorems. We note that other proofs that also do not rely on embedding theorems are already available in the literature (cf.~\cite[Cor.~6.8, p.~87]{Bg}).

\subsection{Tensor products of resolutions} \label{TensRes}
We now return to the consideration of categories that satisfy the assumptions spelled out in Paragraph~2.1, and are therefore in particular monoidal. For our treatment of the cup product in Paragraph~\ref{CupProd} below, we will need that the tensor product of two projective resolutions is again a projective resolution, a fact that is also of independent interest. In the case of the cohomology of groups or more generally of Hopf algebras, this fact is usually deduced from the K\"unneth theorem (cf.~\cite[Sec.~3.2, p.~57]{Be}; \cite[Sec.~2.5, p.~17]{E}; \cite[Sec.~9.3, p.~191]{W2}), an argument that does not carry over directly to the categorical setting. We therefore provide here a different argument that is an adaptation of a line of reasoning for modules found in~\cite[Thm.~2.7.2, p.~58]{We} and is based on the acyclic assembly lemma stated as Lemma~\ref{AcyclAssLem}.

So suppose that
\[\cdots \overset{d_3}{\longrightarrow} P_2
\overset{d_2}{\longrightarrow} P_1
\overset{d_1}{\longrightarrow} P_0 \overset{\xi_X}{\longrightarrow} X \to 0\]
is a resolution of the object~$X$, not necessarily projective, and that
\[\cdots \overset{\delta_3}{\longrightarrow} Q_2
\overset{\delta_2}{\longrightarrow} Q_1
\overset{\delta_1}{\longrightarrow} Q_0 \overset{\xi_Y}{\longrightarrow} Y \to 0\]
is a resolution of the object~$Y$. In the same way as for modules, we can then define a double complex by setting
\[X_{ij} \deq P_i \ot Q_j\]
for~$i \ge 0$ and~$j \ge 0$ and $X_{ij} \deq 0$ if~$i < 0$ or~$j < 0$ (cf.~\cite[Chap.~V, Sec.~1, p.~168]{HiS}). The horizontal and vertical differentials are given by
\[d^1_{ij} \deq d_i \ot \id  \qquad \text{and} \qquad
d^2_{ij} \deq (-1)^i \; \id \ot \, \delta_j \]
respectively. The associated total complex
\[T_n \deq \bigoplus_{\substack{i,j=0 \\i+j=n}}^n P_{i} \ot Q_{j} =
\bigoplus_{i=0}^n P_{i} \ot Q_{n-i} \]
then satisfies $T_0 = P_0 \ot Q_0$ and therefore has an augmentation
$\xi_X \ot \xi_Y \colon T_0 \to X \ot Y$. Using this augmentation, we claim the following:
\begin{Theorem} \label{TensProdRes}
$(T_n)$ is a resolution of~$X \ot Y$. The compositions
\begin{align*}
T_n \xrightarrow{\phantom{\id \ot \xi_Y}}
P_n \ot Q_0 \xrightarrow{\id \ot \xi_Y} P_n \ot Y
\end{align*}
of the projections to the last summand and $\id \ot \xi_Y$ yield a chain map that is a quasi-isomorphism between the two resolutions~$(T_n)$
and~$(P_n \ot Y)$ of~$X \ot Y$.
\end{Theorem}
\begin{proof}
\begin{pflist}
\item
We define $Y_0 \deq \ker(\xi_Y)$, so that the sequence
$Y_0 \rightarrowtail Q_0 \overset{\xi_Y}{\twoheadrightarrow} Y$
is short exact. We can then form the chain complexes
\[
Q'_j \deq
\begin{cases}
Q_{j} &: j > 0 \\
Y_0 &: j = 0 \\
0 &: j < 0
\end{cases}
\qquad \text{and} \qquad
Q''_j \deq
\begin{cases}
0 &: j > 0 \\
Y &: j = 0 \\
0 &: j < 0
\end{cases}
\]
The differentials of~$(Q'_j)$ are $\delta'_j \deq \delta_j$ if $j > 1$, while $\delta'_1 \colon Q_1 \to Y_0$ is induced by $\delta_1 \colon Q_1 \to Q_0$. The differentials~$\delta''_j$ of~$(Q''_j)$ are all equal to zero. Note that~$(Q'_j)$ is exact.

By construction, there is a chain map from~$(Q'_j)$ to~$(Q_j)$ that is the identity in nonzero degrees and is the inclusion~$Y_0 \to Q_0$ in degree zero. Similarly, there is a chain map from~$(Q_j)$ to~$(Q''_j)$ that is zero in nonzero degrees and is the augmentation~\mbox{$\xi_Y \colon Q_0 \to Y$} in degree zero. That this is indeed a chain map follows from the fact that $\xi_Y \circ \delta_1 = 0$.

\item
Exactly as for~$(Q_j)$ above, we now form the corresponding double complexes
$X'_{ij} \deq P_i \ot Q'_j$ and $X''_{ij} \deq P_i \ot Q''_j$ together with their associated total complexes~$(T'_n)$ and~$(T''_n)$. The two chain maps introduced above then lead to two chain maps between these total complexes, one from~$(T'_n)$ to~$(T_n)$ and one from~$(T_n)$ to~$(T''_n)$. As we have~$T''_n = P_n \ot Y$, the second chain map is exactly the one that appears in the assertion. Furthermore, since
\[T'_n \deq \bigoplus_{\substack{i,j=0 \\i+j=n}}^n P_{i} \ot Q'_{j} =
\bigoplus_{i=0}^n P_{i} \ot Q'_{n-i} =
\left( \bigoplus_{i=0}^{n-1} P_{i} \ot Q_{n-i} \right) \oplus (P_{n} \ot Y_0),\] the sequence $T'_n \rightarrowtail T_n \twoheadrightarrow T''_n$ is short exact, because it is the direct sum of the short exact sequence
\[\bigoplus_{i=0}^{n-1} P_{i} \ot Q_{n-i} \rightarrow
\bigoplus_{i=0}^{n-1} P_{i} \ot Q_{n-i} \rightarrow 0\]
and the sequence
\[P_{n} \ot Y_0 \rightarrowtail P_{n} \ot Q_0 \twoheadrightarrow P_{n} \ot Y\]
that is also short exact because of the exactness of the tensor product mentioned in Paragraph~2.1. From this exactness of the tensor product, we also get that, for all $i \in \Z$, the vertical complex~$(X'_{ij})_{j \in \Z}$ is exact. Therefore Lemma~\ref{AcyclAssLem} implies that the total complex~$(T'_n)$ is exact.

\item
Associated with our short exact sequence of chain complexes is a long exact homology sequence
\begin{align*}
\cdots \xrightarrow{\phantom{xxx}} H_{n+1}(T'') \xrightarrow{\phantom{xxx}} H_n(T') \xrightarrow{\phantom{xxx}} H_n(T) \xrightarrow{\phantom{xxx}} H_n(T'') \xrightarrow{\phantom{xxx}} H_{n-1}(T')
\xrightarrow{\phantom{xxx}} \cdots
\end{align*}
Note that this fact, which is usually only shown for modules, is correct in an arbitrary abelian category (cf.~\cite[Chap.~VI, \S~8, p.~154]{Mi}; \cite[Chap.~VIII, \S~4, Lem.~5, p.~206ff]{ML2} as well as \cite{Beyl} and \cite{FHH} for other treatments). Since~$(T'_n)$ is exact, we have that $H_n(T') = 0$. Therefore the long exact homology sequence implies that the chain map from~$(T_n)$ to~$(T''_n)$ induces an isomorphism $H_n(T) \rightarrow H_n(T'')$
in homology, i.e., is a quasi-isomorphism (cf.~\cite[Def.~1.1.2, p.~3]{We}). But again by the exactness of the tensor product, we then have that
\[H_n(T) \cong H_n(T'') \cong H_n(P) \ot Y \cong
\begin{cases}
0 &: n > 0 \\
X \ot Y &: n = 0 \\
\end{cases}\]
which proves our assertion.
\qedhere
\end{pflist}
\end{proof}

In this context, we note that the analogous compositions
\begin{align*}
T_n \xrightarrow{\phantom{\xi_X \ot \id}} P_0 \ot Q_n \xrightarrow{\xi_X \ot \id}  X \ot Q_n
\end{align*}
of the projections to the first summand and $\xi_X \ot \id$ also yield a chain map that is a quasi-isomorphism, as can be shown by a similar argument.

It now follows directly that tensor products of projective resolutions are again projective resolutions:
\begin{Cor} \label{TensProj}
If~$(P_i)$ is a projective resolution of~$X$ and $(Q_j)$ is a projective resolution of~$Y$, then $(T_n)$ is a projective resolution of~$X \ot Y$.
\end{Cor}
\begin{proof}
As we have discussed in Paragraph~2.1, the objects~$P_i \ot Q_j$ are projective, and a direct sum of projective objects is projective
(cf.~\cite[Chap.~I, Prop.~4.5, p.~24]{HiS}; \cite[Chap.~II, Prop.~14.3, p.~70]{Mi}). Therefore, each object~$T_n$ of the total complex is projective.
\end{proof}

\subsection{Cup products} \label{CupProd}
The definition of the Yoneda product in Paragraph~\ref{YonProd} did not require that~$\CC$ is a tensor category. But if we assume that~$\CC$ satisfies the conditions stated in Paragraph~2.1, we can introduce another product, the so-called cup product. As just recalled in the proof of Theorem~\ref{TensProdRes}, the tensor product is exact in each variable. So if the extensions
\[E \colon 0 \to Y \to Z_{m-1} \to \dots \to Z_1 \to Z_0 \to X \to 0\]
and
\[F \colon 0 \to Y' \to Z'_{l-1} \to \dots \to Z'_1 \to Z'_0 \to X' \to 0\]
represent elements of~$\Yext^m(X,Y)$ and~$\Yext^l(X',Y')$ respectively, the sequences
\[E \ot Y' \colon 0 \to Y \ot Y' \to Z_{m-1} \ot Y' \to \dots \to Z_1 \ot Y' \to Z_0 \ot Y' \to X \ot Y' \to 0\]
and
\[X \ot F \colon 0 \to X \ot Y' \to X \ot Z'_{l-1} \to \dots \to X \ot Z'_1 \to X \ot Z'_0 \to X \ot X' \to 0\]
are still extensions, and by forming their Yoneda product
$(E \ot Y') (X \ot F)$, we obtain an operation
\[\cup \colon \Yext^m(X,Y) \times \Yext^l(X',Y') \to \Yext^{m+l}(X \ot X',Y \ot Y')\]
called the cup product. Note that in the case where $X=Y'=\1$, the unit object of the category, the cup product coincides with the Yoneda product in the sense that we have~$E \cup F = EF$ up to the unit constraints, as already mentioned in Paragraph~\ref{RedGroupCohom}.

The cup product can also be described via a different extension: We define $Z_m \deq Y$ and $Z'_l \deq Y'$ and form the tensor product complex
$$U_n \deq \bigoplus_{i=0}^m \bigoplus_{\substack{j=0 \\i+j=n}}^l Z_{i} \ot Z'_{j}$$
for $n=0,\ldots,m+l$ not of~$E$ and~$F$, but of the complexes in which the last terms~$X$ and~$X'$ are missing. We then have $U_0 = Z_0 \ot Z'_0$ and
$U_{m+l} = Z_m \ot Z'_l = Y \ot Y'$. In particular, we get a morphism
$U_0 = Z_0 \ot Z'_0 \to X \ot X'$. In this way, we obtain an extension:
\begin{Proposition} \label{RepCupProd}
The sequence
\[ 0 \to Y \ot Y' \to U_{m+l-1} \to \dots \to U_1 \to U_0 \to X \ot X' \to 0\]
is an extension that is equivalent to the Yoneda product~$(E \ot Y') (X \ot F)$.
\end{Proposition}
\begin{proof}
The assertion that this sequence is an extension follows from Theorem~\ref{TensProdRes}, as every extension can be viewed as a resolution by adding zeros on the left. The fact that this extension is equivalent to the one described above can be shown as in the case of Hopf algebras, an argument found in \cite[Prop.~3.2.1, p.~58]{Be} and \cite[Lem.~9.3.4, p.~191]{W2}.
In greater detail, we define for $n=0,\ldots,l-1$ the morphism
\[\phi_n \colon U_n \to X \ot Z'_{n}\]
as the composition of the projection from~$U_n$ to its summand~$Z_{0} \ot Z'_{n}$ and from there to $X \ot Z'_{n}$, using the given morphism in~$E$. For $n=l,\ldots,m+l-1$, we define
\[\phi_n \colon U_n \to Z_{n-l} \ot Y'\]
as the projection from~$U_n$ to its summand $Z_{n-l} \ot Z'_{l} = Z_{n-l} \ot Y'$. To see that~$(\phi_n)$ is a morphism of extensions, or in other words a chain map, it is necessary to check the commutativity of several squares, which can be organized into the following five cases:
\begin{parlist}
\item
On the very right, the required commutative diagram is
\begin{center}
\begin{tikzcd} [column sep = large]{}
U_{0} = Z_0 \ot Z'_0 \arrow{r}{} \arrow{d}{\phi_0} & X \ot X' \arrow{d}{\id} \\
X \ot Z'_0 \arrow{r}{} & X \ot X'
\end{tikzcd}
\end{center}
which commutes by construction.

\item
For $n=0,\ldots,l-2$, the required commutative diagram is
\begin{center}
\begin{tikzcd} [column sep = large]{}
(Z_{1} \ot Z'_n) \oplus (Z_{0} \ot Z'_{n+1}) \arrow{r}{} \arrow{d}{}
& Z_0 \ot Z'_n \arrow{d}{} \\
U_{n+1} \arrow{r}{} \arrow{d}{\phi_{n+1}} & U_{n} \arrow{d}{\phi_n} \\
X \ot Z'_{n+1} \arrow{r}{} &  X \ot Z'_n
\end{tikzcd}
\end{center}
in which the upper square comes from the definition of the differential of the total complex introduced at the beginning of Paragraph~\ref{AcyclAss}. The diagram commutes because the composition $Z_1 \to Z_0 \to X$ is zero.

\item
For $n=l-1$, the required commutative diagram is
\begin{center}
\begin{tikzcd} [column sep = large]{}
(Z_{1} \ot Z'_{l-1}) \oplus (Z_{0} \ot Z'_{l}) \arrow{r}{} \arrow{d}{}
& Z_0 \ot Z'_{l-1} \arrow{d}{} \\
U_{l} \arrow{r}{} \arrow{d}{\phi_{l}} & U_{l-1} \arrow{d}{\phi_{l-1}} \\
Z_0 \ot Z'_{l} \arrow{r}{} &  X \ot Z'_{l-1}
\end{tikzcd}
\end{center}
in which the lower horizontal morphism splices the two extensions together. It commutes for the same reason as in the previous case.

\item
For $n=l,\ldots,m+l-2$, the required diagram is
\begin{center}
\begin{tikzcd} [column sep = large]{}
U_{n+1} \arrow{r}{} \arrow{d}{\phi_{n+1}} & U_{n} \arrow{d}{\phi_n} \\
Z_{n+1-l} \ot Z'_{l} \arrow{r}{} &  Z_{n-l} \ot Z'_{l}
\end{tikzcd}
\end{center}
which commutes because the second summand in the differential for the total tensor product complex is zero in this case.

\item
On the very left, we have $n=m+l-1$ and $U_{m+l} = Z_{m} \ot Z'_{l} = Y \ot Y'$, so that the diagram
\begin{center}
\begin{tikzcd} [column sep = large]{}
U_{m+l} \arrow{r}{} \arrow{d}{\id} & U_{m+l-1} \arrow{d}{\phi_{m+l-1}} \\
Z_{m} \ot Z'_{l}  \arrow{r}{} &  Z_{m-1} \ot Z'_{l}
\end{tikzcd}
\end{center}
commutes for the same reason as in the previous case.
\qedhere
\end{parlist}
\end{proof}

\subsection{Cup products via projective resolutions} \label{CupProdRes}
The cup product can also be described via projective resolutions: For a given projective resolution
\[\cdots \xlongrightarrow{d_3} P_2 \xlongrightarrow{d_2} P_1 \xlongrightarrow{d_1} P_0
\xlongrightarrow{\xi} X\]
of the object~$X$, we have discussed in Paragraph~\ref{YonProd} that the extension~$E$ considered in Paragraph~\ref{CupProd} above corresponds to the cohomology class of the
morphism~\mbox{$\phi_m \in \Hom_{\CC}(P_m,Y)$} that arises from the commutative diagram
\begin{center}
\begin{tikzcd}[column sep = small] {}
\cdots \arrow{r}{} &
P_{m+2} \arrow{d}{\phi_{m+2}} \arrow{r}{d_{m+2}} &
P_{m+1} \arrow{d}{\phi_{m+1}} \arrow{r}{d_{m+1}} &
P_m \arrow{d}{\phi_m} \arrow{r}{d_m} &
P_{m-1} \arrow{d}{\phi_{m-1}} \arrow{r}{} &
\cdots \arrow{r}{} &
P_1 \arrow{d}{\phi_1} \arrow{r}{d_1} &
P_0 \arrow{d}{\phi_0} \arrow{r}{\xi} &
X \arrow{d}{\id_X} \arrow{r}{} &
0 &\\
\cdots \arrow{r}{} &
0 \arrow{r}{} &
0 \arrow{r}{} &
Y \arrow{r}{} &
Z_{m-1} \arrow{r}{} &
\cdots \arrow{r}{} &
Z_1 \arrow{r}{} &
Z_0 \arrow{r}{} &
X \arrow{r}{} &
0
\end{tikzcd}
\end{center}

Similarly, if
\[\cdots \xlongrightarrow{d'_3} Q_2 \xlongrightarrow{d'_2} Q_1 \xlongrightarrow{d'_1} Q_0
\xlongrightarrow{\xi'} X'\]
is a projective resolution of~$X'$, the extension~$F$ also considered above corresponds to the cohomology class of the morphism~$\psi_l \in \Hom_{\CC}(Q_l,Y')$ that arises from the commutative diagram
\begin{center}
\begin{tikzcd}[column sep = small] {}
\cdots \arrow{r}{} &
Q_{l+2} \arrow{d}{\psi_{l+2}} \arrow{r}{d'_{l+2}} &
Q_{l+1} \arrow{d}{\psi_{l+1}} \arrow{r}{d'_{l+1}} &
Q_l \arrow{d}{\psi_l} \arrow{r}{d'_l} &
Q_{l-1} \arrow{d}{\psi_{l-1}} \arrow{r}{} &
\cdots \arrow{r}{} &
Q_1 \arrow{d}{\psi_1} \arrow{r}{d'_1} &
Q_0 \arrow{d}{\psi_0} \arrow{r}{\xi'} &
X' \arrow{d}{\id_{X'}} \arrow{r}{} &
0 &\\
\cdots \arrow{r}{} &
0 \arrow{r}{} &
0 \arrow{r}{} &
Y' \arrow{r}{} &
Z'_{l-1} \arrow{r}{} &
\cdots \arrow{r}{} &
Z'_1 \arrow{r}{} &
Z'_0 \arrow{r}{} &
X' \arrow{r}{} &
0
\end{tikzcd}
\end{center}
From Corollary~\ref{TensProj}, we know that the total complex
\[T_n \deq \bigoplus_{\substack{i,j=0 \\i+j=n}}^n P_{i} \ot Q_{j}\]
is a projective resolution of~$X \ot X'$. By functoriality of the total tensor product complex, the tensor products~$\phi_i \ot \psi_j$ then yield a chain map from~$T_n$ to the total complex
\[U_n \deq \bigoplus_{i=0}^m \bigoplus_{\substack{j=0 \\i+j=n}}^l Z_{i} \ot Z'_{j}\]
that, according to Proposition~\ref{RepCupProd}, represents the cup product.
Since we have $U_{m+l} = Z_{m} \ot Z'_{l}$, the arising morphism $T_{m+l} \to U_{m+l}$ has at most one nonvanishing component, namely
$$\phi_m \ot \psi_l \colon P_m \ot Q_l \to U_{m+l}$$
and vanishes on all summands $P_{i} \ot Q_{j}$ of~$T_{m+l}$ that satisfy $i+j=m+l$, but not~$i=m$ and~$j=l$.

In view of the relation between the Yoneda description and the description via projective resolutions discussed in Paragraph~\ref{YonProd}, this means that the cup product in the latter description is induced by the map
$$\Hom(P_m, Y) \times \Hom(Q_l, Y') \rightarrow \Hom(T_{m+l}, Y \ot Y')$$
that assigns to a pair $(\phi_m, \psi_l)$ of morphisms the morphism
$T_{m+l} \to Y \ot Y'$ that is equal to $\phi_m \ot \psi_l$ on $P_m \ot Q_l$
and vanishes on all summands $P_{i} \ot Q_{j}$ of~$T_{m+l}$ that satisfy $i+j=m+l$, but not~$i=m$ and~$j=l$.

In the case where $X = X' = \1$, the cup product induces an operation
\[\cup \colon \Ext^m(\1,Y) \times \Ext^l(\1,Y') \to \Ext^{m+l}(\1,Y \ot Y')\]
because $\1 \cong \1 \ot \1$. But even if we use a specific projective resolution of the unit object on the left-hand side, so that $P_i = Q_i$, the description of the cup product that we have given above uses on the right-hand side the projective resolution~$(T_n)$. By the comparison theorem already mentioned in Paragraph~3.1, there must then be a quasi-isomorphism from~$(P_n)$ to~$(T_n)$ that lifts the isomorphism~$\1 \cong \1 \ot \1$ to the projective resolutions. Such a quasi-isomorphism is often called a diagonal approximation (cf.~\cite[Sec.~3.2, p.~57]{Be}) or, at least in special cases, an Alexander-Whitney map. For group cohomology, explicit diagonal approximations for the standard resolution can be found, for example, in~\cite[Sec.~3.4, p.~64]{Be} or~\cite[Sec.~3.2, p.~25]{E}, and for the periodic resolution discussed in Paragraph~\ref{PerResCycl} in~\cite[Sec.~3.5, p.~67]{Be} or~\cite[loc.~cit.]{E}.

\end{document}